\documentclass[11pt]{amsart}

\date{\today}

\usepackage{amsmath, amsthm, amscd, amsfonts, amssymb, color}
\usepackage[bookmarksnumbered, colorlinks]{hyperref}
\usepackage{tikz}
  \tikzstyle{block}=[draw, rectangle, rounded corners, text centered, minimum height=10mm, node distance=8em, text width=9em]
\tikzstyle{line}=[draw,<->, thick]
\newcommand{\xn}{(x_n)_n}
\newcommand{\yn}{(y_n)_n}
\newcommand{\A}{\textsf{A}}

\newcommand{\la}{\langle}
\newcommand{\ra}{\rangle}

\newcommand{\R}{\textup{R}}
\newcommand{\N}{\textup{N}}

\newcommand{\B}{\mathbf{B}}

\theoremstyle{theorem}
\newtheorem{theorem}{Theorem}[section]
\newtheorem{cor}[theorem]{Corollary}
\newtheorem{prop}[theorem]{Proposition}
\newtheorem{lemma}[theorem]{Lemma}
\theoremstyle{remark}
\newtheorem{remark}[theorem]{Remark}
\newtheorem{example}[theorem]{Example}
\newtheorem{definition}[theorem]{Definition}

\date{\today}
\title{Frames and outer frames for Hilbert $C^*$-modules}
\author[Lj. Aramba\v si\' c]
{Ljiljana Aramba\v si\' c}
\author[D. Baki\' c]
{Damir Baki\' c$^{*}$}
\thanks{$^{*}$ Corresponding author}
\address{Department of Mathematics, University of Zagreb,
Bijeni\v cka cesta 30, 10000 Zagreb, Croatia.}
\email{arambas@math.hr}
\email{bakic@math.hr}

\begin{document}

\begin{abstract}
The goal of the present paper is to extend the theory of frames for countably generated Hilbert $C^*$-modules over arbitrary $C^*$-algebras. In investigating the non-unital case we introduce the concept of outer frame as a sequence in the multiplier module $M(X)$ that has the standard frame property when applied to elements of the ambient module $X$. Given a Hilbert $\A$-module $X$, we prove that there is a bijective correspondence of the set of all adjointable surjections from the generalized Hilbert space $\ell^2(\A)$ to $X$ and the set consisting of all both frames and outer frames for $X$. Building on a unified approach to frames and outer frames we then obtain new results on dual frames, frame perturbations, tight approximations of frames and finite extensions of Bessel sequences.
\end{abstract}


\maketitle

\section{Introduction}

\vspace{.2in}

A Hilbert $C^*$-module over a $C^*$-algebra $\A$ is a right $\A$-module $X$ equipped with an $\A$-valued inner product $\la \cdot,\cdot \ra : X\times X \to \A$ such that $X$ is a Banach space with respect to the norm $\|x\|=\|\la x,x\ra\|^{\frac{1}{2}}.$ Recall that the inner product on $X$ has the properties
\begin{enumerate}
\item $\la x,x\ra \geq 0$;
\item $\la x,x\ra =0 \Leftrightarrow x=0$;
\item $\la x,y+z\ra=\la x,y\ra+\la x,z\ra$;
\item $\la x,ya\ra=\la x,y\ra a$;
\item $\la x,y\ra=\la y,x\ra^*$;
\end{enumerate}
that are satisfied for all $x,y,z\in X$ and $a\in \A.$

A Hilbert $\A$-module $X$ is said to be full if the closed linear span of the set $\{\la x,y\ra : x,y\in X\}$ is all of $\A.$ We say that $X$ is countably generated if there exists a sequence $\xn$ in $X$ such that the closed linear span of the set $\{x_na:n\in \Bbb N, a\in \A\}$ is equal to $X.$ A subclass AFG consists of algebraically finitely generated Hilbert $\A$-modules, i.e.,  those $X$ for which there exists a finite sequence $(x_n)_{n=1}^N$ such that $X=\{\sum_{n=1}^Nx_na_n:a_n\in \A\}.$
The most important examples of Hilbert $C^*$-modules over a $C^*$-algebra $\A$ are:
\begin{itemize}
\item $X=\A$ with the inner product $\la a,b\ra=a^*b$;
\item $X=\A^N=\A \oplus \ldots \oplus \A,$ $N\in \Bbb N,$ ($N$  copies of $\A$) with the inner product $\la (a_1,\ldots,a_N), (b_1,\ldots,b_N)\ra=\sum_{n=1}^Na_n^*b_n$;
\item $X=\ell^2(\A)$ - the generalized Hilbert space over $\A.$ Recall that $\ell^2(\A)$ consists of all sequences $(a_n)_n$ of elements of $\A$ such that the series $\sum_{n=1}^{\infty}a_n^*a_n$ converges in norm, and the inner product on $\ell^2(\A)$ is defined by $\la (a_n)_n,(b_n)_n\ra=\sum_{n=1}^{\infty}a_n^*b_n.$ Given $a\in \A,$ we shall denote by $a^{(n)}\in \ell^2(\A)$ the sequence $(0,\ldots,0,a,0,\ldots)$ with $a$ on the $n$-th position and zeros elsewhere.
\end{itemize}

If $X$ and $Y$ are Hilbert $\A$-modules we denote by $\Bbb B(X,Y)$ the Banach space of all adjointable operators from $X$ to $Y.$ Given an adjointable operator $T$ we denote by $\R(T)$ and $\N(T)$ the range and the null-space of $T,$ respectively.

For $x\in X$ and $y\in Y$ let $\theta_{y,x}\in \Bbb B(X,Y)$ denote the map defined by $\theta_{y,x}(z)=y\la x,z\ra,\,z\in X.$ The linear span of all $\theta_{y,x}$'s is denoted by $\Bbb F(X,Y),$ while its closure is denoted by $\Bbb K(X,Y).$ These two classes of adjointable operators are referred to as the classes of "finite rank" and generalized compact operators, respectively. In the case $Y=X$ we simply write $\Bbb F(X),$ $\Bbb K(X)$ and $\Bbb B(X).$

For basic facts on Hilbert $C^*$-modules we refer the reader to  \cite{L,MT,RW,W-O}.

\begin{definition}\label{very first fd}
Let $X$ be a Hilbert $C^*$-module. A (possibly finite) sequence $(x_n)_n$ in $X$ is called a \emph{frame} for $X$ if there exist positive constants $A$ and $B$ such that
\begin{equation}\label{frame definition}
A \la x,x\ra \le \sum_{n=1}^{\infty}\la x,x_n\ra \la x_n,x\ra \le B \la x,x\ra,\quad\forall x\in X,
\end{equation}
where the sum in the middle converges in norm.
If only the second inequality in (\ref{frame definition}) is satisfied, we say that $(x_n)_n$ is a \emph{Bessel sequence}.
The constants $A$ and $B$ are called \emph{frame bounds}. If $A=B=1,$ i.e.,  if
\begin{equation}\label{Parseval}
\sum_{n=1}^{\infty}\la x,x_n\ra \la x_n,x\ra = \la x,x\ra,\quad\forall x\in X,
\end{equation}
the sequence $(x_n)_n$ is called a \emph{Parseval frame} for $X.$
\end{definition}

\vspace{.1in}

Notice that when we take for the underlying $C^*$-algebra of coefficients the field of complex numbers, i.e.,  when $X$ is a Hilbert space, (\ref{frame definition}) becomes
$$
A\|x\|^2\leq \sum_{n=1}^{\infty}|\langle x_n,x\rangle|^2\leq B\|x\|^2,\quad\forall x\in X,
$$
which means that $\xn$ is a standard Hilbert space frame.

Frames for Hilbert spaces were introduced by R.J.~Duffin and A.C.~Schaeffer in 1952. In 1980's frames begun to play an important role in wavelet and Gabor analysis. Since then, frames are an important tool in both theoretical and applied mathematics. Frames for Hilbert $C^*$-modules were introduced by M.~Frank and D.~Larson; the basic modular frame theory is developed in \cite{FL1, FL2,FL3}. In particular, it was proved in Example~3.5 in \cite{FL2} that frames exist in every finitely or countably generated Hilbert $C^*$-module. The proof is based on the Kasparov stabilization theorem (\cite{W-O}, Theorem~15.4.6).

\vspace{.1in}

In the rest of this introductory section we summarize basic facts concerning modular frames.

\vspace{.1in}

First, observe that we do not require in Definition~\ref{very first fd} that $X$ is a full Hilbert $\A$-module. Also, the underlying $C^*$-algebra $\A$ may be non-unital. When this is the case we can consider the minimal unitization $\tilde{\A}$ of $\A$ whose elements we write in the form $a+\lambda e,$ $a\in \A,$ $\lambda \in \Bbb C,$ where $e$ denotes the unit in $\tilde A.$  It is well known that $X$ can be regarded as a Hilbert $\tilde{\A}$-module with the same inner product and the action given by $x(a+\lambda e)=xa+\lambda x$ for all $x\in X,$ $a\in \A,$ and $\lambda \in \Bbb C.$ However, one should keep in mind that $X$ is never full over $\tilde{\A}.$

Secondly, we point out that we assume the norm-convergence of the series in (\ref{frame definition}). Since in each $C^*$-algebra a convergent series of positive elements necessarily converges unconditionally, we could index our frame by any countable set instead of $\Bbb N.$ Observe also that the unconditional convergence of the series in (\ref{frame definition}) implies that the family $\{\la x,x_n\ra \la x_n,x\ra : n\in \Bbb N\}$ is summable in $\A$ for each $x$ in $X.$

Given a frame $\xn$ for a Hilbert $C^*$-module $X,$ we define the analysis operator $U:X  \to \ell^2(\A)$ (resp.~$U:X  \to \A^N$ if $(x_n)_{n=1}^N$ is a finite frame) by
$$Ux=(\la x_n,x\ra)_n.$$
It is well known that $U$ is an adjointable map and that the adjoint operator $U^*$ - that is called the synthesis operator - is given by $$U^*((a_n)_n)=\sum_{n=1}^{\infty}x_na_n,\quad \textup{resp. } U^*((a_n)_{n=1}^N)=\sum_{n=1}^{N}x_na_n.$$
In particular, if $\A$ is unital, we have $U^*e^{(n)}=x_n,$ $n\in \Bbb N.$ Here and in the sequel we denote by $e$ the unit element in a unital $C^*$-algebra.

Furthermore, the above defining condition (\ref{frame definition}) implies that $U$ is bounded from below; hence, $\R(U)$ is a closed submodule of $\ell^2(\A).$ Using Corollary~15.3.9 from \cite{W-O} we now have $\ell^2(\A)=\R(U)\oplus\N(U^*).$ Moreover, by \cite[Theorem~15.3.8]{W-O}, the range of $U^*$ is also closed and since $U$ is, being bounded from below, an injection, we conclude that $U^*$ is surjective. These properties imply that $U^*U$ is an invertible operator in $\Bbb B(X).$

In particular, note that the analysis operator $U$ of a Parseval frame $\xn$ is an isometry; hence, when this is the case we have $U^*U=I,$ where $I$ is the identity operator on $X.$

\vspace{.1in}

We are now ready to state a result that provides the most fundamental property of modular frames. This was first proved by M.~Frank and D.~Larson for Hilbert $C^*$-modules over unital $C^*$-algebras, but it is easily seen that the result extends to the non-unital case too.

\begin{theorem}\label{full reconstruction}
Let  $(x_n)_n$ be a frame for a Hilbert $C^*$-module $X$ with the analysis operator $U.$ Then
\begin{equation}\label{canonical reconstruction}
x=\sum_{n=1}^{\infty}x_n\la (U^*U)^{-1}x_n,x\ra=\sum_{n=1}^{\infty}(U^*U)^{-1}x_n\la x_n,x\ra,\quad\forall x\in X.
\end{equation}
In particular, $(x_n)_n$ is a Parseval frame for $X$ if and only if
\begin{equation}\label{Parseval reconstruction}
x=\sum_{n=1}^{\infty}x_n\la x_n,x\ra,\quad\forall x\in X.
\end{equation}
\end{theorem}

\vspace{.1in}

Given a frame $\xn$ for a Hilbert $C^*$-module $X,$ each sequence $(y_n)_n$ that satisfies the equality
$$x=\sum_{n=1}^{\infty}x_n\la y_n,x\ra,\quad \forall x\in X,$$
is called a dual of $\xn.$ In general, a frame $\xn$ may posses many duals. The first equality in (\ref{canonical reconstruction}) tells us that the sequence $((U^*U)^{-1}x_n)_n$ is a dual of $\xn.$ This sequence is called the canonical dual frame; namely, $((U^*U)^{-1}x_n)_n$ is indeed a frame for $X$ since $(U^*U)^{-1}$ is an adjointable surjection \cite[Theorem~2.5]{A}.

Observe also that the first equality in (\ref{canonical reconstruction}) shows that each frame $\xn$ for $X$ generates $X.$ In particular, a Hilbert $C^*$-module that admits frames is necessarily countably generated. Analogously, each Hilbert $C^*$-module that possesses a finite frame is an AFG Hilbert $C^*$-module.

Theorem~\ref{full reconstruction} is a natural generalization of the corresponding result for Hilbert spaces. However, further properties of modular frames cannot be derived by simply following the frame theory for Hilbert spaces. To see what the obstacles are, let us first note two basic facts concerning frames in Hilbert spaces.

\begin{remark}\label{what fails}
Let $H$ be a Hilbert space.
\begin{itemize}
\item[(a)] Suppose that $\xn$ is a sequence in $H$ such that the series $\sum_{n=1}^{\infty}c_nx_n$ converges for each sequence $(c_n)_n\in \ell^2(\Bbb C).$ Then $\xn$ is a Bessel sequence (\cite[Corollary~3.2.4]{C}).
\item[(b)] Each bounded surjection $T\in \Bbb B(\ell^2(\Bbb C),H)$ is the synthesis operator of some frame for $H$ (\cite[Theorem~5.5.5]{C}).
\end{itemize}
\end{remark}

Unfortunately, both statements from the preceding remark can fail if the complex field is replaced by a general $C^*$-algebra; that is, the above statements are not generally true for Hilbert $C^*$-modules. This is demonstrated by two examples that follow. First we need a lemma.

\begin{lemma}\label{tko konvergira}
Let $\A$ be a $C^*$-algebra and let $(t_n)_n$ be a sequence in the multiplier algebra $M(\A)$ of $\A.$ The following two statements are equivalent:
\begin{enumerate}
\item[(a)] The series $\sum_{n=1}^{\infty}t_na_n$ is norm-convergent for each $(a_n)_n\in \ell^2(\A).$
\item[(b)] The sequence $(\sum_{n=1}^Nt_nt_n^*)_N$ is bounded.
\end{enumerate}
\end{lemma}
\begin{proof}
Let us assume (a). Then $T:\ell^2(\A) \to \A,$ $T((a_n)_n)=\sum_{n=1}^{\infty}t_na_n$ is a well defined $\A$-linear operator, where $\A$ is regarded as a Hilbert $C^*$-module over itself. Consider also, for each $N\in \Bbb N,$ the operators $T_N:\ell^2(\A) \to \A,$ $T_N((a_n)_n)=\sum_{n=1}^{N}t_na_n.$ Observe that $T_N$'s are adjointable operators; their adjoints are given by $T_N^*a=(t_1^*a,\ldots,t_N^*a,0,0,\ldots),$ $a\in \A.$ In particular, all $T_N$ are bounded. Obviously, the sequence $(T_N)_N$ converges to $T$ in the strong operator topology. By the uniform boundedness principle there exists a positive constant $M$ such that $\|T_N\|\le M$ for all $N \in \Bbb N,$ and $\|T\| \leq M.$

Recall now that for each $t\in M(\A)$ we have
$$
\|t\|=\sup \{\|ta\|: a\in \A, \|a\|\leq 1\}.
$$
Fix $N\in \Bbb N.$ Then for every $a\in \A$ it holds
$$
\Big\|\sum_{n=1}^Nt_nt_n^*a\Big\|=\|T(T_N^*a)\|\leq
M\Big\| \sum_{n=1}^Na^*t_nt_n^*a\Big\|^{\frac{1}{2}}
\leq M\|a^*\|^{\frac{1}{2}}\Big\| \sum_{n=1}^Nt_nt_n^*a\Big\|^{\frac{1}{2}}.
$$
By taking supremum on both sides over all $a\in \A,$ $\|a\|\leq 1,$ we get
$$
\Big\|\sum_{n=1}^Nt_nt_n^*\Big\|\leq M \Big\| \sum_{n=1}^Nt_nt_n^*\Big\|^{\frac{1}{2}},
$$
and hence
$$
\Big\|\sum_{n=1}^Nt_nt_n^*\Big\|\leq M^2.
$$

Let us now suppose (b), that is, let $\Big\|\sum_{n=1}^Nt_nt_n^*\Big\|\leq M$ for some $M>0$ and all $N \in \Bbb N.$ Take any sequence $(a_n)_n\in \ell^2(\A).$ For $\varepsilon >0$ we can find $N_0\in \Bbb N$ such that
$$
N_2\geq N_1\geq N_0 \Rightarrow \Big\|\sum_{n=N_1+1}^{N_2}a_n^*a_n\Big\|<\varepsilon.
$$
From this we conclude, for all $N_2\geq N_1\geq N_0,$
\begin{eqnarray*}
\Big\|\sum_{n=1}^{N_2}t_na_n-\sum_{n=1}^{N_1}t_na_n\Big\|^2&=&\Big\|\sum_{n=N_1+1}^{N_2}t_na_n\Big\|^2\\
  &\leq&\Big\|\sum_{n=N_1+1}^{N_2}t_nt_n^*\Big\|\,\Big\|\sum_{n=N_1+1}^{N_2}a_n^*a_n\Big\|\\
  &<&M\varepsilon.
\end{eqnarray*}
where the first inequality above is obtained by applying the Cauchy-Schwarz inequality in the Hilbert $C^*$-module $M(\A)^{N_2-N_1}.$ Thus, $(\sum_{n=1}^{N}t_na_n)_N$ is a Cauchy sequence, and hence convergent.
\end{proof}

\begin{example}\label{fail 1}
Here we demonstrate an example of a sequence $\xn$ in a Hilbert $C^*$-module $X$ over a $C^*$-algebra $\A$ such that the series $\sum_{n=1}^{\infty}x_na_n$ converges in $X$ for all
$(a_n)_n\in \ell^2(\A),$ but which is not Bessel.

Take an infinite dimensional separable Hilbert space $H$ with an orthonormal basis $(\epsilon_n)_n.$ For $n\in \Bbb N$ let $e_n$ denote the one-dimensional projection onto the subspace $\text{span}\,\{\epsilon_n\}$.

Consider $X=\Bbb B(H)$ as a Hilbert $C^*$-module over itself. Obviously, $\sum_{n=1}^Ne_ne_n^*$ is the orthogonal projection onto
$\textup{span}\,\{\epsilon_1,\ldots,\epsilon_N\}$; thus, the sequence $(\sum_{n=1}^Ne_ne_n^*)_N$ is bounded. By the preceding lemma, the series $\sum_{n=1}^{\infty}e_na_n$ converges for each sequence $(a_n)_n\in \ell^2(\Bbb B(H)).$ However, $(e_n)_n$ is not a Bessel sequence in $X.$ Namely, if it were Bessel, that would imply that the series $\sum_{n=1}^{\infty}\la a,e_n\ra \la e_n,a\ra=\sum_{n=1}^{\infty}a^*e_na$ converges in norm for each $a\in X.$ In particular, this norm-limit should coincide with the strong limit of the series $\sum_{n=1}^{\infty}a^*e_na$ which is, obviously, equal to $a^*a$.  But, this is impossible for each non-compact operator $a$ on $H.$ So, if we put $x_n=e_n, \,n\in \Bbb N$, the sequence $(x_n)_n$ has the desired properties.
\end{example}

\begin{remark}\label{Heuser}
A well-known result (the Heuser lemma) on square summable sequences of scalars states: if $(c_n)_n$ is a sequence of complex numbers such that the series
$\la (c_n)_n, (a_n)_n\ra=\sum_{n=1}^{\infty}c_na_n$ is convergent for each $(a_n)_n\in \ell^2(\Bbb C),$ then $(c_n)_n\in \ell^2(\Bbb C).$ The preceding example shows that an analogous result does not hold in the generalized Hilbert space $\ell^2(\A).$ Namely, the sequence $(x_n)_n$ from the preceding example has the property that the series
$\la (x_n)_n, (a_n)_n\ra=\sum_{n=1}^{\infty}x_na_n$ is convergent for each $(a_n)_n\in \ell^2(\B(H)),$ but $(x_n)_n$ does not belong to $\ell^2(\B(H)).$
\end{remark}

\begin{example}\label{fail 2}
Here we demonstrate an example of an adjointable surjection from $\ell^2(\A)$ to a Hilbert $\A$-module $X$ which is not the synthesis operator of any frame for $X.$

Consider an infinite dimensional separable Hilbert space $H$ such that $H=\bigoplus_{n=1}^{\infty}H_n,$ where $\textup{dim}\,H_n=\infty$ for each $n\in \Bbb N.$ Note that the elements of $H$ can be identified as sequences $(\xi_n)_n$ such that $\xi_n\in H_n,$ $n\in \Bbb N,$ and $\sum_{n=1}^{\infty}\|\xi_n\|^2<\infty.$ Let $X=\Bbb K(H),$ where $\Bbb K(H)$ denotes the $C^*$-algebra of all compact operators on $H.$

Let $s_n\in \Bbb B(H)$ denote the isometry with the final space $H_n$ for every $n\in \Bbb N.$ Observe that $s_n^*s_n=e$ ($e$ stands for the identity operator on $H$), while $s_ns_n^*=p_n,$ where $p_n$ denotes the orthogonal projection onto $H_n.$ Since the sequence $(\sum_{n=1}^Np_n)_N$ converges to $e$ in the strong operator topology, a standard argument shows that the sequence $(\sum_{n=1}^Np_na)_N$ converges in norm to $a$ for each compact operator $a.$ Thus, for each  $a\in \Bbb K(H)$ we have $\sum_{n=1}^{\infty}p_na=a$ in the sense of norm-convergence.

Consider now
$$T: \ell^2(\Bbb K(H)) \to \Bbb K(H),\quad T((a_n)_n)=\sum_{n=1}^{\infty}s_na_n.$$
By Lemma~\ref{tko konvergira}, $T$ is well defined. Moreover, $T$ is an adjointable operator; its adjoint $T^*$ is given by
$$T^*a=(\la s_n,a\ra)_n=(s_n^*a)_n.$$
Note that $T^*$ is well defined since we have, by the conclusion from the preceding paragraph, $$\sum_{n=1}^{\infty}a^*s_ns_n^*a=a^*\sum_{n=1}^{\infty}p_na=a^*a,\quad \forall a\in \Bbb K(H).$$
This also shows that $T^*$ is an isometry; hence, $T$ is a surjection.

We now claim that there does not exist a frame $\xn$ for $X=\Bbb K(H)$ whose synthesis operator is $T.$ To see this, suppose the opposite: let $\xn$ be a frame for $X$ such that $T((a_n)_n)=\sum_{n=1}^{\infty}x_na_n$ for each $(a_n)_n\in \ell^2(\Bbb K(H)).$ Then we have
$\sum_{n=1}^{\infty}x_na_n=\sum_{n=1}^{\infty}s_na_n$ for each $(a_n)_n\in \ell^2(\Bbb K(H)).$ In particular, if we take arbitrary $n\in \Bbb N,$ $a\in \Bbb K(H),$ and
$a^{(n)}\in \ell^2(\Bbb K(H)),$ we get $x_na=s_na$ for all $n\in \Bbb N$ and $a\in \Bbb K(H).$ Since $\Bbb K(H)$ acts non-degenerately on $H,$ this is enough to conclude $x_n=s_n$ for all $n\in \Bbb N.$ But this is obviously impossible since each $s_n$ is a non-compact operator.
\end{example}

As the above two examples show, both statements from Remark~\ref{what fails} can fail in Hilbert $C^*$-modules.
We shall address these problems more thoroughly at the beginning of Section 3 and in Section 5 in the discussion following Remark~\ref{konus}. We will show that in the study of Hilbert $C^*$-modules over non-unital $C^*$-algebras some difficulties arise from sequences and operators with properties as in the preceding two examples and that these difficulties cannot be circumvented by simply adjoining the unit element to the underlying $C^*$-algebra $\A$ and regarding the original Hilbert $C^*$-module $X$ as a module over the unital $C^*$-algebra $\tilde{\A}.$

\vspace{.1in}

The paper is organized as follows. In Section 2 we discuss further basic properties of frames. In particular, we describe in Proposition~\ref{parseval frames are approximate units} and Theorem~\ref{second construction method} the interrelation of Parseval frames for a Hilbert $C^*$-module $X$ with increasing approximate units for $\Bbb K(X).$

In Section 3 we introduce the concept of an outer frame for a Hilbert $C^*$-module $X,$ a concept that naturally fits into the picture when one studies frames for Hilbert $C^*$-modules over non-unital $C^*$-algebras.
We show in Theorem~\ref{stvarna korespondencija frameova i sinteza} that there is a bijective correspondence of the set of all adjointable surjections from the generalized Hilbert space $\ell^2(\A)$ to a Hilbert $\A$-module $X$ and the set consisting of all both frames and outer frames for $X$.

In Section 4 we describe all frames that are dual to a given frame. It turns out that in these considerations one has to take into account outer frames discussed in the preceding section.
In particular, we describe in Theorems~\ref{general dual} and \ref{general dual 1} (synthesis operators of) all frames and outer frames that are dual to a given frame or an outer frame. At the end of Section 4 we discuss frames and outer frames with a unique dual.

Section 5 is devoted to frame perturbations and tight approximations of frames. Again, outer frames naturally fit into the picture when discussing the non-unital case. After proving a perturbation result (Theorem~\ref{okolina1}), we obtain in Propositions~\ref{Parsevalov frame u okolini} and \ref{napeti frame u okolini} the best Parseval resp.~tight approximation of a frame or an outer frame in terms of the distance of the corresponding analysis/synthesis operators.

Finally, in the concluding Section 6 we investigate finite extensions of Bessel sequences to frames and outer frames. In Theorems~\ref{fin_ext_outer} and \ref{fin_ext_Pars_outer} we characterize those Bessel sequences that admit such extensions to frames.

\vspace{.1in}
Throughout the paper $\A$ will denote an arbitrary $C^*$-algebra. We do not assume that $\A$ is unital and this particular assumption will be explicitelly stated when needed. The multiplier algebra of $\A$ will be denoted by $M(\A).$
By an approximate unit for a $C^*$-algebra $\A$ we understand a net $(e_{\lambda})_{\lambda}$ of positive elements in the unit ball of $\A$ such that $\lim_{\lambda}e_{\lambda}a=a$, for all $a \in \textsf{A}.$ Approximate unit is increasing if $e_\lambda\le e_\mu$ whenever $\lambda \le \mu.$ Recall that a $C^*$-algebra $\A$ has a countable approximate unit precisely when it is $\sigma$-unital (i.e.,  when there exists a strictly positive element in $\A$).

Given a $C^*$-algebra $\A$ and the generalized Hilbert space  $\ell^2(\A)$ over $\A$, we denote by $c_{00}(\A)$ the set of all finite sequences in  $\ell^2(\A)$, i.e.
$$c_{00}(\A)=\{(a_n)_n:a_n \in \A, a_n=0, \,\forall n>N \mbox{ for some } N\in \Bbb N\}.$$ Clearly, $c_{00}(\A)$ is norm-dense in $\ell^2(\A)$.

We tacitly assume that the class of countably generated Hilbert $C^*$-modules includes all AFG Hilbert $C^*$-modules (and, obviously, when we work with finite frames for AFG modules the convergence questions become superfluous). We shall explicitelly indicate when a particular discussion is concerned with AFG modules exclusively.

\vspace{0.2in}

\section{Basic properties and characterizations}

\vspace{0.2in}

The frame condition (\ref{frame definition}) from Definition~\ref{very first fd} involves two inequalities concerning order in the underlying $C^*$-algebra that are not always easy to verify. However, it turns out that it suffices to check the corresponding inequalities in norm (\cite[Theorem~2.6]{A} and \cite[Proposition~3.8]{Jing}). In fact, as we shall see in our Theorem~\ref{frame relaxed} below, even more is true.
Our first theorem is concerned with Bessel sequences. We show that, in order to prove that a sequence $\xn$ in a Hilbert $\A$-module $X$ is Bessel, one has only to verify that the sequence $(\la x_n,x\ra)_n$ belongs to $\ell^2(\A)$ for each $x\in X.$

\begin{theorem}\label{Bessel relaxed}
Let $(x_n)_n$ be a sequence in a Hilbert $\A$-module $X.$ Then the following two conditions are equivalent:
\begin{itemize}
\item[(a)] $(x_n)_n$ is a Bessel sequence.
\item[(b)] The series $\sum_{n=1}^{\infty}\la x,x_n\ra \la x_n,x\ra$ converges for all $x$ in $X.$
\end{itemize}

If $(x_n)_n$ is a Bessel sequence, its analysis operator
$$U:X\to\ell^2(\A),\quad U(x)=(\la x_n,x\ra)_n,$$
is well defined and adjointable and the adjoint operator $U^*$ is given by
\begin{equation}\label{Bessel-adjoint}
U^*((a_n)_n)=\sum_{n=1}^{\infty}x_na_n,\quad \forall (a_n)_n\in\ell^2(\A),
\end{equation}
where the series $\sum_{n=1}^{\infty}x_na_n$ converges unconditionally for all $(a_n)_n \in \ell^2(\A).$ In particular, if $(e_{\lambda})_{\lambda}$ is an approximate unit for $\A,$ then $U^*e_{\lambda}^{(n)}=x_ne_{\lambda}$ and $\lim_{\lambda}U^*e_{\lambda}^{(n)}=x_n$ for each $n\in \Bbb N.$ Consequently, the sequence $(x_n)_n$ is bounded and $\|x_n\|\le \|U\|$ for all $n$ in $\Bbb N.$ Finally, if $\A$ is unital then $x_n=U^*e^{(n)}$ for all $n \in \Bbb N.$
\end{theorem}

\begin{proof}
Suppose that (b) is satisfied. Then the operator $U: X \to \ell^2(\A),$ $Ux=(\la x_n,x\ra)_n,$ is well defined and $\A$-linear. We now show that $U$ has a closed graph.
Let $(y,(a_n)_n)=\lim_{k\rightarrow \infty}(y_k,Uy_k),$ where $y_k,y\in X,\,(a_n)_n\in \ell^2(\textsf{A}).$
For each $m\in \Bbb N$ and all $k\in \Bbb N$ we have
\begin{eqnarray*}
(a_m-\la x_m,y_k\ra)^*(a_m-\la x_m,y_k\ra)&\le&\sum_{n=1}^{\infty}(a_n-\la x_n,y_k\ra)^*(a_n-\la x_n,y_k\ra)\\
&=&\la (a_n)_n-Uy_k,(a_n)_n-Uy_k\ra.
\end{eqnarray*}
Taking norms on both sides we get
$$ \|a_m-\la x_m,y_k\ra\|^2\le \|(a_n)_n-Uy_k\|^2.$$
By assumption $(a_n)_n=\lim_{k\rightarrow\infty}Uy_k$ and $y=\lim_{k\rightarrow\infty}y_k,$ so we get
$$
a_m=\lim_{k\rightarrow\infty}\la x_m,y_k\ra=\la x_m,y\ra.
$$
As $m$ was arbitrary, this shows that $(a_n)_n=Uy.$ So, the graph of $U$ is closed and hence $U$ is a bounded operator.
Now, by \cite[Theorem~2.8]{pas}, it follows that $\la Ux,Ux\ra\le \|U\|^2\la x,x\ra$ for all $x\in X$; thus, $(x_n)_n$ is a Bessel sequence.

\vspace{.1in}

Let us now show, for each $(a_n)_n\in \ell^2(\textsf{A}),$ the unconditional convergence of the series $\sum_{n=1}^{\infty}x_na_n.$ Take arbitrary finite set $F\subseteq \Bbb N$ and denote by $|F|$ the cardinality of $F.$ Then
\begin{eqnarray*}
\Big\| \sum_{n\in F}x_na_n\Big\|^2&=&\sup\left\{\Big\| \Big\langle \sum_{n\in F}x_na_n,y \Big\rangle \Big\|^2:y\in X,\,\|y\|\le 1  \right\}\\
&=&\sup\left\{\Big\|\sum_{n\in F}a_n^*\la x_n,y\ra \Big\|^2:y\in X,\,\|y\|\le 1  \right\}\\
&&\textup{(by applying the Cauchy-Schwarz inequality in $\textsf{A}^{|F|}$)}\\
&\le&\sup\left\{\Big\|\sum_{n\in F}a_n^*a_n\Big\| \,\Big\| \sum_{n\in F}\la y,x_n\ra \la x_n,y\ra \Big\|:y\in X,\,\|y\|\le 1  \right\}\\
&\le&\sup\left\{\Big\|\sum_{n\in F}a_n^*a_n\Big\| \,\Big\| \sum_{n=1}^{\infty}\la y,x_n\ra \la x_n,y\ra \Big\|:y\in X,\,\|y\|\le 1  \right\}\\
&=&\Big\|\sum_{n\in F}a_n^*a_n\Big\| \ \sup\left\{\| Uy \|^2:y\in X,\,\|y\|\le 1  \right\}\\
&=&\|U\|^2 \Big\|\sum_{n\in F}a_n^*a_n\Big\|.
\end{eqnarray*}
Since the series $\sum_{n=1}^{\infty}a_n^*a_n$ converges unconditionally, the family
$\{ a_n^*a_n: n\in \Bbb N\}$ is summable. Hence, the inequality $\left\| \sum_{n\in F}x_na_n\right\|^2\le \|U\|^2 \left\|\sum_{n\in F}a_n^*a_n\right\|$ that we have obtained for each finite subset $F$ of $\Bbb N,$ shows summability of the family
$\{ x_na_n: n\in \Bbb N\},$ which is equivalent to the unconditional convergence of the series $\sum_{n=1}^{\infty}x_na_n.$

\vspace{.1in}

It is now easy to prove that $U$ is adjointable, since we now know that the operator given in \eqref{Bessel-adjoint}
is well defined, and satisfies
$$\la Ux,(a_n)_n\ra =\sum_{n=1}^\infty \la x,x_n\ra a_n=  \la x,\sum_{n=1}^\infty x_n a_n\ra=\la x,U^*((a_n)_n)\ra$$
for all $x\in X$ and $(a_n)_n\in\ell^2(\A).$

The remaining assertions are evident.
\end{proof}

A direct consequence of the preceding theorem is the following characterization of frames for Hilbert $C^*$-modules.

\begin{theorem}\label{frame relaxed}
Let $(x_n)_n$ be a sequence in a Hilbert $\textsf{A}$-module $X.$ Then the following two conditions are equivalent:
\begin{itemize}
\item[(a)] $(x_n)_n$ is a frame for $X.$
\item[(b)] The series $\sum_{n=1}^{\infty}\la x,x_n\ra \la x_n,x\ra$ converges for all $x \in X$ and there exists a constant $A>0$ such that
$A\|x\|^2\le \left\| \sum_{n=1}^{\infty}\la x,x_n\ra \la x_n,x\ra \right\|$ for all $x$ in $X.$
\end{itemize}
\end{theorem}
\begin{proof}
Immediate from Theorem~\ref{Bessel relaxed} and \cite[Theorem~2.6]{A} (or \cite[Proposition~3.8]{Jing}).
\end{proof}

\vspace{.1in}

Another useful characterization of frames arises from a correspondence of Parseval frames for a Hilbert $C^*$-module $X$ with approximate units for $\Bbb K(X).$

It is easy to check that, regarding a $C^*$-algebra $\A$ as a Hilbert $C^*$-module over itself, a sequence $(a_n)_n$ of elements of $\A$ is a Parseval frame for $\A$ precisely when the sequence $(\sum_{n=1}^N a_na_n^*)_N$ is an approximate unit for $\A.$
This observation was extended in \cite[Theorem~1.4]{KLZ} to a wider class of Hilbert $C^*$-modules. We show in the following proposition that it remains true for all Hilbert $C^*$-modules.

\begin{prop}\label{parseval frames are approximate units}
Let $X$ be a Hilbert $C^*$-module. Then a sequence $(x_n)_n$ of elements of $X$ is a Parseval frame for $X$ if and only if the sequence $(\sum_{n=1}^N\theta_{x_n,x_n})_N$ is an approximate unit for $\Bbb K(X).$
\end{prop}
\begin{proof}
Suppose that $(x_n)_n$ is a Parseval frame for $X.$ Let $F_N=\sum_{n=1}^N\theta_{x_n,x_n},$ $N\in \Bbb N.$ Obviously, $0\le F_N\le F_{N+1}$ for all $N\in \Bbb N.$ From
$$ \la F_Nx ,x\ra =\sum_{n=1}^N\la x,x_n\ra \la x_n,x\ra\le \sum_{n=1}^{\infty}\la x,x_n\ra \la x_n,x\ra =\la x,x\ra, \quad\forall x \in X,$$
it follows $F_N\le I$ for all $N\in\Bbb N$ (where $I$ denotes the identity operator on $X$).

Let us now fix $v$ and $w$ from $X.$ For any $x\in X$ we have
\begin{eqnarray*}
\|(\theta_{v,w}-F_N\theta_{v,w})x\|&=&
\left\| v\la w,x\ra - \sum_{n=1}^Nx_n\la x_n,v\ra \la w,x\ra \right\|\\
&=&\left\| \left(v-\sum_{n=1}^Nx_n\la x_n,v\ra\right) \la w,x\ra \right\|.
\end{eqnarray*}
Since $(x_n)_n$ is a Parseval frame for $X,$ this shows, by the reconstruction property (\ref{Parseval reconstruction}) from Theorem~\ref{full reconstruction}, that $\|\theta_{v,w}-F_N\theta_{v,w}\| \rightarrow 0$ as $N\rightarrow \infty.$ Clearly, this implies $\|\theta-F_N\theta\| \rightarrow 0$ as $N\rightarrow \infty$ for each $\theta\in\Bbb F(X).$
Finally, take arbitrary $T\in \Bbb K(X)$ and $\varepsilon >0.$ Then we can find $\theta \in \Bbb F(X)$ such that $\|T-\theta\|<\frac{\varepsilon}{3}.$ Further,
there exists $N_0$ such that $\|\theta-F_N\theta\|<\frac{\varepsilon}{3}$ whenever $N\ge N_0.$ Then, for each $N\ge N_0$ we have
$$\|T-F_NT\|\le \| T-\theta\|+\|\theta-F_N\theta\|+\|F_N\theta-F_NT\|<\frac{\varepsilon}{3}+\frac{\varepsilon}{3}+\frac{\varepsilon}{3}=\varepsilon,$$
so $(F_N)_N$ is an approximate unit for $\Bbb K(X).$

To prove the converse, suppose that $(\sum_{n=1}^N\theta_{x_n,x_n})_N$ is an approximate unit for $\Bbb K(X).$ Then
$\theta_{y,y}=\lim_{N\to \infty}\sum_{n=1}^N\theta_{x_n,x_n}\theta_{y,y}$ for each $y\in X.$ In particular,
\begin{equation}\label{yyy}
\theta_{y,y}(y)=\lim_{N\to \infty}\sum_{n=1}^N\theta_{x_n,x_n}\theta_{y,y}(y), \quad \forall y\in X.
\end{equation}
Recall from Proposition~2.31 in \cite{RW} that each $x\in X$ can be written in the form $x=y\la y,y\ra=\theta_{y,y}(y)$ for some $y\in X.$ Then \eqref{yyy} becomes
$$x=\lim_{N\rightarrow \infty}\sum_{n=1}^Nx_n\la x_n,x \ra, \quad\forall x\in X,$$
so by Theorem~\ref{full reconstruction}, $(x_n)_n$ is a Parseval frame for $X.$
\end{proof}

\vspace{.1in}
The preceding proposition extends to arbitrary frames in a standard way (see also \cite[Theorem~1.4]{KLZ}). In the corollary that follows we shall use the strict convergence in $\Bbb B(X)$ with respect to the ideal of generalized compact operators $\Bbb K(X).$

\begin{cor}
\label{frames and approximate units}
Let $X$ be a Hilbert $C^*$-module. Then a sequence $(x_n)_n$ of elements of $X$ is a frame for $X$ if and only if the sequence $(\sum_{n=1}^N\theta_{x_n,x_n})_N$ strictly converges to some invertible operator in $\Bbb B(X).$
\end{cor}
\begin{proof}
Let $(x_n)_n$ be a frame for $X$ and $U$ its analysis operator. We apply Proposition~\ref{parseval frames are approximate units} to the Parseval frame $((U^*U)^{-\frac{1}{2}}x_n)_n$ (the sequence $((U^*U)^{-\frac{1}{2}}x_n)_n$ is indeed a Parseval frame for $X$, see \cite{FL1}). By Proposition~\ref{parseval frames are approximate units} we conclude
that the sequence $(\sum_{n=1}^N\theta_{(U^*U)^{-\frac{1}{2}}x_n,(U^*U)^{-\frac{1}{2}}x_n})_N$ is an approximate unit for $\Bbb K(X).$ From this we deduce that $(\sum_{n=1}^N\theta_{(U^*U)^{-\frac{1}{2}}x_n,(U^*U)^{-\frac{1}{2}}x_n})_N$ strictly converges to the identity operator $I.$ Since we have $$\sum_{n=1}^N\theta_{(U^*U)^{-\frac{1}{2}}x_n,(U^*U)^{-\frac{1}{2}}x_n}=(U^*U)^{-\frac{1}{2}}(\sum_{n=1}^N\theta_{x_n,x_n})(U^*U)^{-\frac{1}{2}}, \quad\forall N\in \Bbb N,$$ it follows that the sequence $(\sum_{n=1}^N\theta_{x_n,x_n})_N$ strictly converges to the invertible operator $(U^*U)^{-1}\in\Bbb B(X).$

Conversely, if $(\sum_{n=1}^N\theta_{x_n,x_n})_N$ strictly converges to some invertible  $T\in\Bbb B(X),$ then $T$ is necessarily positive, so it follows that the increasing sequence $(\sum_{n=1}^N\theta_{T^{-\frac{1}{2}}x_n,T^{-\frac{1}{2}}x_n})_N$ strictly converges to the identity operator on $X.$ In other words, the sequence $(\sum_{n=1}^N\theta_{T^{-\frac{1}{2}}x_n,T^{-\frac{1}{2}}x_n})_N$
is an approximate unit for $\Bbb K(X).$
By Proposition~\ref{parseval frames are approximate units} it follows that $(T^{-\frac{1}{2}}x_n)_n$ is a Parseval frame for $X.$ Finally, applying \cite[Theorem~2.5]{A}, we conclude that $(x_n)_n$ is a frame for $X.$
\end{proof}

\vspace{.1in}

Next we show that every countably generated Hilbert $C^*$-module $X$ admits approximate units for $\Bbb K(X)$ of the form as in Proposition~\ref{parseval frames are approximate units}. In the proof we shall make use of the left Hilbert $C^*$-module structure on $X$ arising from the action of generalized compact operators.

\begin{theorem}\label{Parseval by Brown}
Let $X$ be a countably generated Hilbert $\textsf{A}$-module.
There exists a sequence $(x_n)_n$ in $X$ such that  $(\sum_{n=1}^N\theta_{x_n,x_n})_N$ is an approximate unit for $\Bbb K(X).$
\end{theorem}
\begin{proof}
Since $X$ is countably generated over $\textsf{A},$  Proposition~6.7 from \cite{L} implies that the $C^*$-algebra $\Bbb K(X)$ is $\sigma$-unital.

Now recall that $X$ is also a full left Hilbert $\Bbb K(X)$-module with the action $(T,x) \mapsto Tx,\,T\in \Bbb K(X),\, x\in X,$ and the inner product $[x,y]=\theta_{x,y}.$ The resulting norm $_\Bbb K\|x\|=\|\theta_{x,x}\|^{\frac{1}{2}}$ coincides with the original norm on $X$ that arises from the right module structure over $\textsf{A}.$

We now apply Lemma~7.3 from \cite{L} to the full left Hilbert $\Bbb K(X)$-module $X$: there exists a sequence $\xn$ in $X$ such that the sequence $(\sum_{n=1}^N[x_n,x_n])_N,$ that is, $(\sum_{n=1}^N\theta_{x_n,x_n})_N$, is an approximate unit for $\Bbb K(X).$
\end{proof}

\vspace{.1in}

Observe that the existence of frames in countably generated Hilbert $C^*$-modules can now be reproved by using Theorem~\ref{Parseval by Brown} and Proposition~\ref{parseval frames are approximate units}.

We also have the following easy consequence of Theorem~\ref{Parseval by Brown}.

\begin{cor}\label{oblik pozitivnog kompaktnog operatora}
Let $X$ be a countably generated Hilbert $\textsf{A}$-module. For each positive operator $T\in \Bbb K(X)$ there exists a sequence $(y_n)_n$ in $X$ such that $T=\sum_{n=1}^{\infty}\theta_{y_n,y_n},$ where this series converges in norm.
\end{cor}
\begin{proof}
Let $T\in \Bbb K(X),$ $T\ge 0.$ Using the approximate unit from the preceding theorem we have $T^{\frac{1}{2}}=\lim_{N\rightarrow \infty} \sum_{n=1}^NT^{\frac{1}{2}}\theta_{x_n,x_n}.$
Multiplying by $T^{\frac{1}{2}}$ from the right hand side we get $T=\lim_{N\rightarrow \infty} \sum_{n=1}^NT^{\frac{1}{2}}\theta_{x_n,x_n}T^{\frac{1}{2}}.$ This shows that $(y_n)_n,$ where $y_n=T^{\frac{1}{2}}x_n,\,n\in \Bbb N,$ is a sequence with the desired property.
\end{proof}

\vspace{.1in}

The following result shows that approximate units as in Theorem~\ref{Parseval by Brown}, although of a very special form, not only exist (provided that $X$ is countably generated), but can be derived from any increasing countable approximate unit in $\Bbb K(X).$
To prove this, we first need an auxiliary result on approximate units in $C^*$-algebras.

\begin{lemma}\label{auxiliary apprunit}
Let $(e_n)_n$ be a sequence in a $C^*$-algebra $\textsf{A}$ such that $0\leq e_n\leq e_{n+1}$ and $\|e_n\|\leq 1$ for all $n\in \Bbb N.$ If there exists a subsequence $(e_{p(n)})_n$ of $(e_n)_n$ which is an approximate unit for $\textsf{A},$ then $(e_n)_n$ is also an approximate unit for $\textsf{A}.$
\end{lemma}
\begin{proof}
Let $a\in \textsf{A}.$ First observe that $\lim_{n\to \infty}\|e_{p(n)}a-a\|=0$ implies $\lim_{n\to \infty}\|a^*e_{p(n)}a-a^*a\|=0.$
Fix $\varepsilon >0$ and find $n_0\in \Bbb N$ such that
$\|a^*e_{p(n)}a-a^*a\| <\varepsilon$ for all $n\ge n_0.$
Since $(e_n)_n$ increases and $\|e_n\|\le 1$ for all $n,$ we have
$$0\le a^*a-a^*e_na \le a^*a-a^*e_{p(n_0)}a,\quad \forall n\ge p(n_0),$$
so
\begin{equation}\label{druga}
\|a^*a-a^*e_na \|\le \| a^*a-a^*e_{p(n_0)}a\|<\varepsilon, \quad \forall n\ge p(n_0).
\end{equation}
We now continue our computation in $\tilde{\textsf{A}},$ if needed. Observe that $\|e-e_n\|\leq 1$ for all $n\in \Bbb N.$
Since $$\|a-e_na \|^2=\|(e-e_n)^\frac{1}{2}(e-e_n)^\frac{1}{2}a \|^2\le \|(e-e_n)^\frac{1}{2}a \|^2=\|a^*a-a^*e_na \|,$$
\eqref{druga} gives us $\lim_{n\to \infty}\|a-e_na \|=0.$
\end{proof}

\vspace{.1in}

\begin{theorem}\label{second construction method}
Let $X$ be a Hilbert $\textsf{A}$-module.
\begin{enumerate}
\item[(a)] If a sequence $(E_N)_N$ is an increasing approximate unit for $\Bbb K(X)$ then there exists a sequence $(x_n)_n$  in $X$ and an increasing sequence of natural numbers $(p(N))_N$ with the properties $\sum_{n=1}^{p(N)}\theta_{x_n,x_n} \leq E_N$ and
    $\left\| \sum_{n=1}^{p(N)}\theta_{x_n,x_n} - E_N \right\|<\frac{1}{N}$ for all $N\in \Bbb N.$
\item[(b)] If $(x_n)_n$ is any sequence in $X$ as in (a), then $(x_n)_n$ is a Parseval frame for $X.$
\end{enumerate}
\end{theorem}
\begin{proof}
To prove (a), suppose that $(E_N)_N$ is an increasing approximate unit for $\Bbb K(X).$ By Corollary~\ref{oblik pozitivnog kompaktnog operatora}, there exists a sequence $(y_n^1)_n$ in $X$ such that
$E_1=\sum_{n=1}^{\infty}\theta_{y_n^1,y_n^1}.$ Find $p(1)$ such that
$$
\left\|E_1- \sum_{n=1}^{p(1)}\theta_{y_n^1,y_n^1}\right\|<1.
$$
Put $x_n=y_n^1$ for $n=1,2,\ldots, p(1)$ and $F_{p(1)}=\sum_{n=1}^{p(1)}\theta_{x_n,x_n}.$ Then $F_{p(1)}\le E_1$ and $\|F_{p(1)}-E_1\|< 1.$

Now observe that $F_{p(1)}\le E_1\le E_2$ implies $E_2-F_{p(1)}\ge 0.$ Again by Corollary~\ref{oblik pozitivnog kompaktnog operatora}, there exists a sequence $(y_n^2)_n$ in $X$ such that $E_2-F_{p(1)}=\sum_{n=1}^{\infty}\theta_{y_n^2,y_n^2}.$ Choose $M$ such that
$$
\left\|E_2- F_{p(1)} -\sum_{n=1}^{M}\theta_{y_n^2,y_n^2}\right\|<\frac{1}{2}.
$$
Denote $p(2)=p(1)+M$ and $x_{p(1)+n}=y_n^2,\,n=1,2,\ldots,M.$ Let $F_{p(2)}=F_{p(1)}+\sum_{n=1}^{M}\theta_{y_n^2,y_n^2}= \sum_{n=1}^{p(2)}\theta_{x_n,x_n}.$
Then, by construction, we have $F_{p(2)}\le E_2$ and $\|F_{p(2)}-E_2\|<\frac{1}{2}.$

Proceed by induction to obtain $F_{p(N)}=\sum_{n=1}^{p(N)}\theta_{x_n,x_n}$ with the properties $F_{p(N)}\le E_N$ and $\|F_{p(N)}-E_N\|<\frac{1}{N}.$

\vspace{.1in}

Let us now prove (b).
First note that $0\le F_{p(N)}\le E_N$ implies $\|F_{p(N)}\|\le \|E_N\|\leq 1$ for all $N \in \Bbb N.$
By a routine approximation argument one shows that $\|F_{p(N)}T-T\|\rightarrow 0$ for each $T \in \Bbb K(X),$ that is, $(F_{p(N)})_N$ is an approximate unit for $\Bbb K(X).$

Put $F_N=\sum_{n=1}^{N}\theta_{x_n,x_n}$ for each $N\in \Bbb N.$ Clearly, $0\leq F_N\leq F_{N+1}$ and, because of $F_N\leq F_{p(N)}\le E_N$, we also have $\|F_N\|\leq 1.$ By Lemma~\ref{auxiliary apprunit}, $(F_N)_N$ is an approximate unit for $\Bbb K(X).$
Proposition~\ref{parseval frames are approximate units} now implies that $(x_n)_n$ is a Parseval frame for $X.$
\end{proof}


\vspace{.1in}

We end this section by an example of a Parseval frame for $\ell^2(\A)$ (where $\A$ is an arbitrary $\sigma$-unital $C^*$-algebra) and the corresponding approximate unit for $\Bbb K(\ell^2(\A))$ that arises from Proposition~\ref{parseval frames are approximate units}. First we need a useful auxiliary result.

\begin{lemma}\label{frame property on a dense submodule}
Let $Y$ be a dense submodule of a Hilbert $\A$-module $X.$ Suppose that a sequence $(x_n)_n$ in $X$ has the property
$$
A\la y,y\ra \leq \sum_{n=1}^{\infty}\la y,x_n\ra\la x_n,y\ra \leq B\la y,y\ra,\quad \forall y\in Y,
$$
for some positive constants $A$ and $B.$ Then $(x_n)_n$ is a frame for $X$ with frame bounds $A$ and $B.$
\end{lemma}
\begin{proof}
Let us define $U_0:Y \rightarrow \ell^2(\A)$ by $U_0y=(\la x_n,y\ra)_n.$
Clearly, $U_0$ is well defined, $\A$-linear, bounded and bounded from below. Let $U:X\rightarrow \ell^2(\A)$ be the continuation of $U_0.$ Note that $\|U\|=\|U_0\|\leq \sqrt{B}.$ Similarly, for $x\in X,$ if $(y_n)_n$ is a sequence in $Y$ such that $x=\lim_{n\rightarrow\infty}y_n,$ we have
$$
\|Ux\|=\lim_{n\rightarrow\infty}\|U_0y_n\|\geq \lim_{n\rightarrow\infty}\sqrt{A}\|y_n\|=\sqrt{A}\|x\|.
$$
We now prove that $U$ is an adjointable operator. First, put 
$$U^*(a_1,\ldots,a_N,0,0,\ldots)=\sum_{n=1}^Nx_na_n,\quad \forall(a_1,\ldots,a_N,0,0,\ldots)\in c_{00}(\A).$$
By a routine verification one shows that
\begin{equation}\label{prvi korak do adjointa}
\la Uy,z\ra=\la y,U^*z\ra,\quad \forall y\in Y,\quad\forall z \in c_{00}(\A).
\end{equation}
Suppose now that $x=\lim_{n\rightarrow\infty}y_n$ with $y_n\in Y.$ Then, by (\ref{prvi korak do adjointa}), we have $\la Uy_n,z\ra=\la y_n,U^*z\ra$ for all $n\in \Bbb N$ and $z \in c_{00}(\A).$ By letting $n \rightarrow \infty$
we obtain
\begin{equation}\label{drugi korak do adjointa}
\la Ux,z\ra=\la x,U^*z\ra,\quad\forall x\in X,\quad\forall z \in c_{00}(\A).
\end{equation}
We now show that $U^*$ is bounded on $c_{00}(\A).$ Let $z\in c_{00}(\A).$ Then
\begin{eqnarray*}
\|U^*z\|&=&\sup\{\|\la x,U^*z\ra\|: x\in X, \|x\|\leq 1\}\\
 &\stackrel{\eqref{drugi korak do adjointa}}{=}&\sup\{\|\la Ux,z\ra\|: x\in X, \|x\|\leq 1\}\\
 &\leq&\sqrt{B}\|z\|.
\end{eqnarray*}
This enables us to extend $U^*$ by continuity to $\ell^2(\A).$ It is now evident that (\ref{drugi korak do adjointa}) extends to the same equality that holds true for all $x\in X$ and $z\in \ell^2(\A).$ Thus, $U$ is an adjointable operator.
Since $Ux=(\la x_n,x\ra)_n$ for all $x\in X$ and $A\|x\|^2\le \|Ux\|^2\le B\|x\|^2$ for all $x\in X,$ it only remains to apply Theorem~2.6 from \cite{A}.
\end{proof}

\begin{example}
Let $\A$ be a $\sigma$-unital $C^*$-algebra and let $(e_n)_n$ be an increasing approximate unit for $\A$; put additionally $e_0=0.$ Let $$f_n=(e_n-e_{n-1})^{\frac{1}{2}},\quad n\in \Bbb N.$$
Consider $\A$ as a Hilbert $C^*$-module over itself.
Since $$\sum_{n=1}^N\theta_{f_n,f_n}=\sum_{n=1}^Nf_nf_n^*=\sum_{n=1}^N(e_n-e_{n-1})=e_N,\quad \forall N\in\Bbb N,$$
we conclude from Proposition~\ref{parseval frames are approximate units} that $(f_n)_n$ is a Parseval frame for $\A.$

For $n,j\in \Bbb N$ consider the system
$$
f_n^{(j)}=(0,\ldots,0,f_n,0,\ldots)\in \ell^2(\A)\,\,\, (\text{$f_n$ on $j$-th position, $0$'s elsewhere}).
$$
We will show that the system $(f_n^{(j)})_{n,j=1}^{\infty}$ is a Parseval frame for $\ell^2(\A).$

Let us first organize our system $(f_n^{(j)})_{n,j=1}^{\infty}$ into a sequence.
This can be done in a standard way by enumerating elements along finite diagonals of an infinite matrix starting from the upper left corner. Put
$$
p(m)=\frac{1}{2}m(m+1),\quad m=0,1,2,\ldots
$$
and observe that each natural number $n$ can be written in a unique way as
$$
n=p(m-1)+k_m,\quad m\in \Bbb N,\quad k_m\in \{1,2,\,\ldots ,m\}.
$$
We now put
$$
x_n=x_{p(m-1)+k_m}=f_{m+1-k_m}^{(k_m)},\quad n\in \Bbb N.
$$
This gives us a sequence
$$
f_1^{(1)}, f_2^{(1)}, f_1^{(2)}, f_3^{(1)}, f_2^{(2)}, f_1^{(3)}, f_4^{(1)}, f_3^{(2)}, f_2^{(3)}, f_1^{(4)},\ldots
$$

Let us now show that
\begin{equation}\label{parseval dense}
y=\sum_{n=1}^{\infty}x_n\la x_n,y\ra,\quad \forall y\in c_{00}(\A).
\end{equation}

Fix an arbitrary $y=(a_1,\ldots ,a_m,0,0,\ldots) \in c_{00}(\A).$ Let $\varepsilon>0.$
Since $(f_n)_n$ is a Parseval frame for $\A,$ we can find $N_0\in \Bbb N$ such that
\begin{equation}\label{m parseval approx}
N\geq N_0 \Rightarrow \left\|a_i-\sum_{n=1}^Nf_n\la f_n,a_i \ra\right\|<\frac{\varepsilon}{m},\ \forall i=1,2,\ldots, m.
\end{equation}
For such $N_0$ consider $p(N_0+m).$ We now claim that
\begin{equation}\label{epsilon prag}
N\geq p(N_0+m) \Rightarrow \left\|y-\sum_{n=1}^Nx_n\la x_n,y\ra\right\|<\varepsilon.
\end{equation}
To show this, first observe that each term in the sum $\sum_{n=1}^Nx_n\la x_n,y\ra$ is of the form $f_j^{(k)}\left\langle f_j^{(k)},y\right\rangle$ which is in fact $\left(f_j\la f_j,a_k\ra\right)^{(k)}$ - an element of $\ell^2(\A)$ with
$f_j\la f_j,a_k\ra$ on $k$-th position and $0$'s elsewhere.

Let us first prove (\ref{epsilon prag}) for $N=p(N_0+m).$ Since in this case we have $N=p(N_0+m-1)+(N_0+m),$ the last $N_0+m$ members among
$x_1,x_2,\ldots,x_N$ are
$$
f_{N_0+m}^{(1)},\ldots,f_{N_0+1}^{(m)},\ldots,f_{1}^{(N_0+m)}.
$$
Thus
\begin{eqnarray*}
\left\|y-\sum_{n=1}^N x_n\la x_n,y\ra\right\|&\le & \left\| a_1-\sum_{n=1}^{N_0+m}f_n\la f_n,a_1\ra \right\|+ \left\| a_{2}-\sum_{n=1}^{N_0+m-1}f_n\la f_n,a_{2}\ra \right\|\\
&& +\ldots +
\left\| a_m-\sum_{n=1}^{N_0+1}f_n\la f_n,a_m\ra \right\|\\
&\stackrel{(\ref{m parseval approx})}{<}&\frac{\varepsilon}{m}+\frac{\varepsilon}{m}+\ldots +\frac{\varepsilon}{m}=\varepsilon.
\end{eqnarray*}

Next we  prove (\ref{epsilon prag}) for $N> p(N_0+m).$ By the same reasoning as above we get natural numbers $N_1,N_2,\ldots,N_m>N_0$ such that
\begin{eqnarray*}
\left\|y-\sum_{n=1}^N x_n\la x_n,y\ra\right\|&\le & \left\| a_1-\sum_{n=1}^{N_1}f_n\la f_n,a_1\ra \right\|+
\left\| a_{2}-\sum_{n=1}^{N_2}f_n\la f_n,a_{2}\ra \right\|\\
&&+\ldots + \left\| a_m-\sum_{n=1}^{N_m}f_n\la f_n,a_m\ra \right\|\\
&\stackrel{(\ref{m parseval approx})}{<}&\frac{\varepsilon}{m}+\frac{\varepsilon}{m}+\ldots +\frac{\varepsilon}{m}=\varepsilon.
\end{eqnarray*}

This proves (\ref{parseval dense}). In particular, by taking inner products by $y$ in (\ref{parseval dense}) we obtain
\begin{equation}\label{cetvrti korak}
\la y,y\ra=\sum_{n=1}^{\infty}\la y,x_n\ra \la x_n,y\ra,\quad \forall y\in c_{00}(\A).
\end{equation}
The desired conclusion, namely that $(x_n)_n$ is a Parseval frame for $\ell^2(\A),$ now follows directly from the preceding lemma.

Observe that the same construction can be done starting from an arbitrary Parseval frame $(f_n)_n$ for $\A$ - one can easily check that all the above arguments apply without changes.

\vspace{.1in}

By Proposition~\ref{parseval frames are approximate units} we now know that the sequence $(\sum_{n=1}^N\theta_{x_n,x_n})_N$ is an approximate unit for $\Bbb K(\ell^2(\A)).$ Since each subsequence of an approximate unit is an approximate unit itself, we conclude that the sequence $(\sum_{n=1}^{p(N)}\theta_{x_n,x_n})_N$ is also an approximate unit for $\Bbb K(\ell^2(\A)).$

Finally, note that the operators $T_N=\sum_{n=1}^{p(N)}\theta_{x_n,x_n},\, N\in \Bbb N,$ are in fact of a very simple form. Indeed, by an easy computation one gets
$$
T_N((a_n)_n)=(\sum_{n=1}^Nf_n^2a_1, \sum_{n=1}^{N-1}f_n^2a_2,\ldots,f_1^2a_N,0,0,\ldots),\quad\forall (a_n)_n\in \ell^2(\A);
$$
in other words,
$$
T_N((a_n)_n)=(e_Na_1, e_{N-1}a_2,\ldots,e_1a_N,0,0,\ldots),\quad\forall (a_n)_n\in \ell^2(\A).
$$
\end{example}

\vspace{0.2in}

\section{Outer frames}

\vspace{0.2in}

Recall from the introduction that each frame $(x_n)_n$ for a Hilbert $\A$-module $X$ gives rise to an adjointable surjection (namely, the corresponding synthesis operator) from $\ell^2(\A)$ to $X.$
We open this section with the converse statement - a fact that is, although simple, of great importance in frame theory.
We point out that here the underlying $C^*$-algebra $\A$ must be unital (cf.~Example~\ref{fail 2}).

\begin{prop}\label{frames vs surjections}
Let $X$ be a Hilbert $C^*$-module over a unital $C^*$-algebra $\A$ and let $T\in \Bbb B(\ell^2(\A),X)$ be a surjection. Then there is a frame $(x_n)_n$ for $X$ whose synthesis operator is equal to $T.$
\end{prop}
\begin{proof}
Put $x_n=Te^{(n)},\,n \in \Bbb N.$ By \cite[Theorem~2.5]{A}, the sequence $(x_n)_n$ is a frame for $X.$ Denote the corresponding analysis operator by $U.$ Then we have, for all $x\in X$ and $n\in \Bbb N,$
$$
\la U^*e^{(n)},x\ra=\la e^{(n)},Ux\ra=\la x_n,x\ra=\la Te^{(n)}, x\ra,
$$
which implies $U^*=T.$
\end{proof}

In order to obtain the non-unital version of Proposition~\ref{frames vs surjections}, recall that each Hilbert $C^*$-module $X$ over a non-unital $C^*$-algebra $\A$ can be regarded as a Hilbert $C^*$-module over a unital $C^*$-algebra $\tilde{\A}.$ Since frames for a Hilbert $\A$-module $X$ and frames for a Hilbert $\tilde{\A}$-module $X$ coincide, we conclude from Proposition~\ref{frames vs surjections} that each surjection in $\Bbb B(\ell^2(\tilde{\A}),X)$ serves as the synthesis operator of some frame for $X.$

In some situations this conclusion enables us to reduce the non-unital case to the unital one. However, as we shall see in the subsequent sections, this still does not resolve the difficulty observed in Example~\ref{fail 2}. Namely, there are surjections from $\Bbb B(\ell^2(\A),X)$ which cannot be extended to adjointable operators from $\ell^2(\tilde{\A})$ to $X$ (which is precisely the case with the surjection $T$ from Example~\ref{fail 2}). On the other hand, such surjections, as the same example indicates, might be associated with sequences that behave as frames and the only difference is that the members of such frame-like sequences need not belong to the original module $X.$

The preceding discussion suggests that our study of frames for Hilbert $C^*$-modules over non-unital $C^*$-algebras requires a more general setting.
Thus, we shall extend our considerations to multiplier Hilbert $C^*$-modules.

To avoid unnecessary complications, we shall restrict ourselves in the analysis that follows to {\em infinite sequences}. At the end of this section we shall make appropriate comments on the corresponding results concerning finite frames.

\vspace{.1in}

First, in the remark that follows, we include for reader's convenience the most important facts concerning multiplier Hilbert $C^*$-modules (see \cite{BG} and \cite{BG1}).

\vspace{.1in}

\begin{remark}\label{multiplier modules} Let $X$ be a Hilbert $\textsf{A}$-module.

(a)  There exists a Hilbert $M(\textsf{A})$-module $M(X)$ containing $X$ as the ideal submodule associated with the ideal $\textsf{A}$ in $M(\textsf{A})$; i.e.,  $X=M(X)\textsf{A}.$ It turns out that
$$
X=\{x\in M(X):\la x,v\ra \in \A,\,\forall v\in M(X)\}.
$$
The extended module $M(X)$ is called the multiplier module of $X.$ It is known that $M(X)$ can be naturally identified with $\Bbb B(\textsf{A},X).$ If $\textsf{A}$ is unital, or if $X$ is AFG, $M(X)$ coincides with $X.$
    For each $v\in M(X)$ we have
    $$\|v\|=\sup\{\|va\|:a\in \textsf{A},\,\|a\|\le 1\}
    =\sup\{\|\la v,x\ra\|:x\in X,\,\|x\|\le 1\}.$$
In particular, if $\A$ is a $C^*$-algebra and if one takes $X=A$, then it turns out that the multiplier module $M(X)$ coincides with $M(\A)$.

\vspace{.1in}
(b) The strict topology on $M(X)$ is locally convex topology generated by the family of seminorms $v \mapsto \|va\|,\,a \in \textsf{A},$ and $v \mapsto \| \langle v,x \rangle \|,\, x \in X.$ The multiplier module $M(X)$ is complete with respect to the strict topology. If $(e_{\lambda})_{\lambda}$ is an approximate unit for $\textsf{A},$ then each $v \in M(X)$ satisfies $v = \mbox{(strict)}\lim_{\lambda} ve_{\lambda}.$ Hence, $X$ is strictly dense in $M(X).$ In fact, $M(X)$ is the strict completion of $X.$

\vspace{.1in}

(c) For the generalized Hilbert space $\ell^2(\A)$ over $\A$ we get
$$M(\ell^2(\textsf{A}))=\left\{(c_n)_n \in
M(\textsf{A})^{\Bbb N}:
\sum_{n=1}^{\infty}c_n^*c_n \mbox{\,converges strictly}\right\}$$ and the $M(\textsf{A})$-valued inner product on $M(\ell^2(\textsf{A}))$ is given by
$$\la (c_n)_n,(d_n)_n\ra=\mbox{(strict)}\,\sum_{n=1}^{\infty}c_n^*d_n.$$
The set $c_{00}(M(\textsf{A}))$ of all finite sequences of elements of $M(\textsf{A})$ is strictly dense in $M(\ell^2(\textsf{A})).$

\vspace{.1in}

(d) If $Y$ is a Hilbert $\textsf{A}$-module, each operator $T \in \Bbb B(X,Y)$ has an extension $T_M \in \Bbb B(M(X),M(Y)).$ The extended operator $T_M$ is obtained as the strict continuation of $T$; hence, it is uniquely determined.
    The map $T\mapsto T_M$ is a bijection of $\Bbb B(X,Y)$ and  $\Bbb B(M(X),M(Y))$ such that $\|T_M\|=\|T\|$ and $(T_M)^*=(T^*)_M$ for all $T$ in $\Bbb B(X,Y)$.
\end{remark}

\vspace{.1in}

We now introduce the concept of an outer frame for Hilbert $C^*$-modules. In comparison with frames for $X$ the difference is that the
elements of an outer frame for $X$ are merely members of a larger module $M(X)$ and need not belong to $X.$

\begin{definition}\label{real outer frames}
Let $X$ be a Hilbert $C^*$-module. A sequence $(v_n)_n$ in $M(X)$ is called an \emph{outer frame} for $X$ if   $v_n\in M(X)\setminus X$ for at least one  $n\in\Bbb N,$ and if there exist positive constants $A$ and $B$ such that
\begin{equation}\label{outer-df}
A\la x,x\ra\le \sum_{n=1}^{\infty}\langle x,v_n\rangle \langle v_n,x\rangle\le B\la x,x\ra,\quad\forall x \in X,
\end{equation}
where the series $\sum_{n=1}^{\infty}\langle x,v_n\rangle \langle v_n,x\rangle$ converges in norm of $\A.$

If $A=B=1,$ the sequence $(v_n)_n$ is called an \emph{outer Parseval frame} for $X.$

A sequence $(v_n)_n$ is said to be an \emph{outer Bessel sequence} if only the second inequality in (\ref{outer-df}) is satisfied.
\end{definition}

\vspace{.1in}
Notice that each $\la v_n,x\ra$ belongs to $\A$ for every $x \in X,$ even for those $n$ for which $v_n\in M(X)\setminus X$; this is a consequence of Remark~\ref{multiplier modules}(a).

We also note that outer Parseval frames (though, not under that name) appeared already in \cite{RT} in the context of a generalized version of Kasparov's stabilization theorem.

\begin{remark}\label{no outer frames}
By definition, outer frames do not exist if $X$ is strictly complete, i.e.,  if $M(X)=X$ (by Remark~\ref{multiplier modules}(a), this is the case when $A$ is unital, or when $X$ is AFG).

If $X$ is a countably generated Hilbert $\A$-module such that $M(X)\neq X$ then outer frames exist in abundance. To obtain an outer frame for $X$ we can simply add any vector from $M(X)\setminus X$ to an arbitrary frame for $X.$
\end{remark}

\vspace{.1in}

Let us now show that the sequence from Example~\ref{fail 2} is an outer frame.

\begin{example}\label{fail 2 revisited}
Let us keep  the notations from Example~\ref{fail 2}.
We have seen that
$\lim_{N\rightarrow \infty}\|a-\sum_{n=1}^Np_na\|=0$ for each  $a\in \Bbb K(H).$
This conclusion can be rewritten in the modular context in the form
$$a=\lim_{N\rightarrow \infty}\sum_{n=1}^Np_na=\lim_{N\rightarrow \infty}\sum_{n=1}^N s_n\la s_n,a\ra=
\sum_{n=1}^{\infty} s_n\la s_n,a\ra,
$$
with the norm convergence of the series at the end (recall that the norm on the Hilbert $\Bbb K(H)$-module $\Bbb K(H)$ coincides with the original, i.e.,  operator norm on $\Bbb K(H)$).
By taking the inner product of both sides by $a$ we get
$$
\la a,a\ra=\sum_{n=1}^{\infty} \la a,s_n\ra \la s_n,a\ra,\quad \forall a\in \Bbb K(H).
$$
Thus, $(s_n)_n$ is, being a sequence in $\Bbb B(H)\setminus \Bbb K(H),$ an outer Parseval frame for $\Bbb K(H).$
\end{example}

\vspace{.1in}

We begin our study of outer frames by introducing their analysis and synthesis operators. It turns out that these operators have the same properties as the corresponding operators for frames.

\begin{prop}\label{analysis op outer f}
Let $(v_n)_n$ be an outer frame for a Hilbert $\A$-module $X.$ Then its analysis operator
$$
U: X \to \ell^2(\textsf{A}),\quad U(x)=(\la v_n,x\ra)_n,
$$
is well defined, adjointable and bounded from below. The synthesis operator $U^*$ is surjective and satisfies
$$U^*((a_n)_n)=\sum_{n=1}^{\infty}v_na_n,\quad \forall (a_n)_n\in \ell^2(\textsf{A}),$$ where this series converges in norm.
\end{prop}
\begin{proof}
By defining inequalities \eqref{outer-df}, the operator $U$ is well defined, $\textsf{A}$-linear, bounded by $\sqrt{B},$ and bounded from below by $\sqrt{A}.$ Let us show that $U$ is an adjointable operator.
For $N\in\Bbb N$ and any $y=(a_1,\ldots,a_N,0,\ldots) \in c_{00}(\textsf{A}),$ we put
$$U^*((a_1,\ldots,a_N,0,\ldots))=\sum_{n=1}^Nv_na_n.$$
Observe that all $v_na_n$ belong to $X$ since $M(X)\textsf{A}=X$ (see Remark~\ref{multiplier modules}(a)).
By a routine verification one concludes that
\begin{equation}\label{adjoint phase0}
\langle U^*y,x\rangle=\langle y,Ux\rangle,\quad \forall x\in X,\quad \forall y\in c_{00}(\textsf{A}).
\end{equation}
We now claim that $U^*$ is bounded on $c_{00}(\A).$ Indeed, we have for each $y\in c_{00}(\A)$
\begin{eqnarray*}
\|U^*y\|&=&\sup\{\|\la U^*y,x\ra\|:x\in X, \|x\|\le 1\}\\
&\stackrel{(\ref{adjoint phase0})}{=}&\sup\{\|\la y,Ux\ra\|:x\in X, \|x\|\le 1\}\\
&\le& \sqrt{B} \|y\|.
\end{eqnarray*}
This enables us to extend $U^*$ to all of $\ell^2(\textsf{A})$ by continuity. Moreover, one easily concludes that equality  (\ref{adjoint phase0}) extends then to
$$
\langle U^*y,x\rangle=\langle y,Ux\rangle,\quad \forall x\in X,\quad \forall y\in \ell^2(\textsf{A}).
$$
This proves that $U$ is an adjointable operator.
The preceding discussion also shows that $U^*$ is given by $U^*((a_n)_n)=\sum_{n=1}^{\infty}v_na_n$ for all $(a_n)_n\in \ell^2(\textsf{A}).$ Since $U$ is bounded from below, $U^*$ is surjective.
\end{proof}

\vspace{.1in}

An immediate consequence of (the proof of) the preceding proposition is the corresponding statement concerning outer Bessel sequences.

\begin{cor}\label{adjointable operators are Bessel sequences 1}
Let $(v_n)_n$ be an outer Bessel sequence for a Hilbert $\A$-module $X.$ Then its analysis operator
$$
U: X \to \ell^2(\textsf{A}),\quad U(x)=(\la v_n,x\ra)_n,
$$
is a well defined adjointable operator. The synthesis operator $U^*$ satisfies
$$U^*((a_n)_n)=\sum_{n=1}^{\infty}v_na_n,\quad \forall (a_n)_n\in \ell^2(\textsf{A}),$$ where this series converges in norm.
\end{cor}

\vspace{.1in}

As we shall see, outer frames for a countably generated Hilbert $\A$-module $X$ are exactly what one should add to the set of all frames for $X$ in order to establish a bijective correspondence with surjections from $\Bbb B(\ell^2(\A),X).$ To do that, we need a unified approach to frames and outer frames, and it turns out that this can be done by using another new concept: strict frames for multiplier Hilbert $C^*$-modules.

\begin{definition}\label{outer frames}
Let $X$ be a Hilbert $\textsf{A}$-module. A sequence $(v_n)_n$ in the multiplier module $M(X)$ is called a \emph{strict frame} for  $M(X)$ if there exist positive constants $A$ and $B$ such that
\begin{equation}\label{outer frame definition}
A \la v,v\ra \le \mbox{(strict)}\sum_{n=1}^{\infty}\la v,v_n\ra \la v_n,v\ra \le B \la v,v\ra,\quad\forall v\in M(X).
\end{equation}
If $A=B=1,$ i.e.,  if
\begin{equation}\label{Parseval drugi puta neoprezno}
\mbox{(strict)}\,\sum_{n=1}^{\infty}\la v,v_n\ra \la v_n,v\ra = \la v,v\ra,\quad\forall v\in M(X),
\end{equation}
the sequence $(v_n)_n$ is called a \emph{strict Parseval frame} for $M(X).$
\end{definition}

\vspace{.1in}

\begin{example}\label{kanonska striktna}
Let $\A$ be a non-unital $C^*$-algebra. Then the sequence $(e^{(n)})_n$ is a strict Parseval frame for the multiplier module $M(\ell^2(\A))$.
This follows immediately from Remark~\ref{multiplier modules}(c).
\end{example}

\begin{prop}\label{outer is fine} Let $X$ be a Hilbert $\A$-module.
Every strict frame for $M(X)$ is a frame or an outer frame for $X.$
\end{prop}
\begin{proof}
Let  $(v_n)_n$ be a strict frame for $M(X).$ By definition of the strict convergence in $M(\textsf{A})$  this implies that the series
$\sum_{n=1}^{\infty}\la v,v_n\ra \la v_n,v\ra a$ is norm convergent in $\textsf{A}$ for all $a\in \textsf{A}.$
Then the series $\sum_{n=1}^{\infty}\la va,v_n\ra \la v_n,va\ra $ is norm convergent for all $v\in M(X)$ and $a\in \textsf{A}.$
Since, by Proposition~2.31 from \cite{RW}, each $x\in X$ can be written in the form $x=va$ for some $v\in X$ and $a\in \textsf{A},$ the preceding discussion shows that
the series $\sum_{n=1}^{\infty}\la x,v_n\ra \la v_n,x\ra$ converges in norm for every $x\in X.$
Now, if each $v_n$ belongs to $X$ then  $(v_n)_n$ is a frame for $X,$ and if some $v_n$ is in $M(X)\setminus X$ then $(v_n)_n$ is an outer frame for $X.$
\end{proof}

\vspace{.1in}

\begin{remark}\label{outer vs standard}
If $X$ is a strictly complete Hilbert $C^*$-module, i.e.,  if $M(X)=X$ (for example, when $\A$ is unital or $X$ is AFG), the preceding proposition implies that strict frames are simply frames for $X.$
\end{remark}

\vspace{.1in}

We begin our study of strict frames by showing that the conditions in the definition of a strict frame can be relaxed in a manner similar to that in Theorem~\ref{frame relaxed}.

\begin{theorem}\label{outer frame relaxed}
Let $X$ be a Hilbert $\textsf{A}$-module and let $(v_n)_n$ be a sequence in $M(X).$ Then the following two conditions are equivalent:
\begin{itemize}
\item[(a)] $(v_n)_n$ is a strict frame for $M(X).$
\item[(b)] The series $\sum_{n=1}^{\infty}\la v,v_n\ra \la v_n,v\ra$  converges strictly for all $v$ in $M(X)$ and there is $A>0$ such that
$A\|v\|^2\le \left\| (\textup{strict})\sum_{n=1}^{\infty}\la v,v_n\ra \la v_n,v\ra \right\|$ for all $v$ in $M(X).$
\end{itemize}
If $(v_n)_n$ is a strict frame for $M(X),$ its analysis operator
\begin{equation}\label{analysis-for-strict}
U: M(X) \rightarrow M(\ell^2(\textsf{A})),\quad U(v)=(\la v_n,v\ra)_n,
\end{equation}
is well defined, adjointable and bounded from below. The synthesis operator $U^*$ is surjective and satisfies
\begin{equation}\label{synthesis-for-strict}
U^*((b_n)_n)=(\textup{strict})\sum_{n=1}^{\infty}v_nb_n,\quad \forall (b_n)_n\in M(\ell^2(\textsf{A})).
\end{equation}
In particular, $v_n=U^*e^{(n)}$ for all $n \in \Bbb N.$
\end{theorem}
\begin{proof}
Let us first make an observation concerning elements of $M(\ell^2(\textsf{A})).$ For each $(b_n)_n\in M(\ell^2(\textsf{A}))$ we know that $
b:=(\mbox{strict})\,\sum_{n=1}^{\infty}b_n^*b_n$ exists, which means that for all $a\in \textsf{A}$ the series
$\sum_{n=1}^{\infty}ab_n^*b_n$ and $\sum_{n=1}^{\infty}b_n^*b_na$ converge in norm to $ab$ and $ba,$ respectively. In particular, if we assume that $\textsf{A}$ is faithfully and non-degenerately represented on some Hilbert space $H,$
then the series
$\sum_{n=1}^{\infty}b_n^*b_n$ also converges to $b$ in the strong operator topology. This, in particular, implies that
$\sum_{n=1}^{N}b_n^*b_n\le b$ and, consequently,
$\|\sum_{n=1}^{N}b_n^*b_n\|\le \|b\|$ for all $N \in \Bbb N.$

\vspace{.1in}

Let us now assume (b).

By the first assumption in (b), the operator $U:M(X) \rightarrow M(\ell^2(\textsf{A})),$ $U(v)=(\la v_n,v\ra)_n,$ is well defined and $M(\textsf{A})$-linear. By applying the closed graph theorem, precisely as in the first part of the proof of Theorem~\ref{Bessel relaxed}, one shows that $U$ is bounded.
Put $\|U\|^2=B.$

As in the proof of Proposition~\ref{outer is fine} we observe that $Ux\in \ell^2(\textsf{A})$ for each $x\in X.$ Thus, the restriction $U_X$ of $U$ to $X$ takes values in
$\ell^2(\textsf{A}).$ Since norms on $M(X)$ and $M(\ell^2(\textsf{A}))$ extend the original norms on $X$ and $\ell^2(\textsf{A}),$ respectively, we also have $\|U_Xx\|\le \sqrt{B} \|x\|$ for all $x\in X.$

We now prove that $U$ is an adjointable operator. Let us first
define $U^*$ on finite sequences by putting $U^*(b_1,\ldots,b_N,0,\ldots)=\sum_{n=1}^Nv_nb_n$ for each $(b_1,\ldots,b_N,0,\ldots)\in c_{00}(M(\textsf{A})).$ In particular, we have $U^*e^{(n)}=v_n, n\in \Bbb N.$ By a routine computation one finds
\begin{equation}\label{adjoint phase1}
\la z,Uv\ra=\la U^*z,v\ra,\quad\forall z\in c_{00}(M(\textsf{A})),\quad \forall v\in M(X).
\end{equation}
We now claim that $U^*$ is bounded on $c_{00}(M(\textsf{A})).$ Indeed, we have for each $z\in c_{00}(M(\textsf{A}))$
\begin{eqnarray*}
\|U^*z\|&=&\sup\{\|\la U^*z,v\ra\|:v\in M(X), \|v\|\le 1\}\\
&\stackrel{(\ref{adjoint phase1})}{=}&\sup\{\|\la z,Uv\ra\|:v\in M(X), \|v\|\le 1\}\\
&\le& \sqrt{B} \|z\|.
\end{eqnarray*}

Next we claim: if we have $z\in M(\ell^2(\textsf{A}))$ and a net $(z_{\lambda})_{\lambda}$ in $c_{00}(M(\textsf{A}))$ such that $z=(\mbox{strict})\lim_{\lambda} z_{\lambda,}$ then there exists $(\mbox{strict}) \lim_{\lambda} U^*z_{\lambda}$ in $M(X).$
By Remark~\ref{multiplier modules}(b) it is enough to prove that $(U^*z_{\lambda})_{\lambda}$ is a strictly Cauchy net. This means that $((U^*z_{\lambda})a)_{\lambda}$
and $(\la U^*z_{\lambda},x\ra)_{\lambda}$ should be Cauchy nets, for all $a\in \textsf{A}$ and $x\in X.$

First, for each $a\in \textsf{A}$ and $\lambda,\mu,$ we have
$$
\|(U^*z_{\mu})a-(U^*z_{\lambda})a\|=\|U^*(z_{\mu}a-z_{\lambda}a)\|\le \sqrt{B} \|z_{\mu}a-z_{\lambda}a\|.
$$
This is enough, since $(z_{\lambda}a)_{\lambda}$ is a norm convergent net.

Secondly, for each $x\in X,$ we have
$\la U^*z_{\lambda},x\ra=\la z_{\lambda},Ux\ra$; but $(\la z_{\lambda},Ux\ra)_{\lambda}$ is a convergent net since $Ux\in \ell^2(\textsf{A}).$

Let us now fix an arbitrary
$z=(b_1,b_2,b_3,\ldots) \in M(\ell^2(\textsf{A})).$ By Remark~\ref{multiplier modules}(c) we have
$z=(\mbox{strict})\lim_{N\rightarrow \infty} z_N$ with
$z_N=\sum_{n=1}^Ne^{(n)}b_n.$ By the preceding paragraph there exists $(\mbox{strict})\lim_{N\rightarrow \infty} U^*z_N$ in $M(X)$ and we denote this limit by $U^*z.$

Let us now prove that $\la z,Uv\ra=\la U^*z,v\ra$ for all $v\in M(X).$

First, for each $a\in \textsf{A}$ we know by (\ref{adjoint phase1}) that $a\la z_N,Uv\ra=a\la U^*z_N,v\ra$ for all $v\in M(X).$ This implies
\begin{equation}\label{adjoint ending}
\la z_Na^*,Uv\ra=\la (U^*z_N)a^*,v\ra,\quad\forall a\in \textsf{A},\quad\forall v\in M(X).
\end{equation}
Since $\|z_Na^*-za^*\| \rightarrow 0$ as $N$ tends to infinity, the left hand side in (\ref{adjoint ending}) converges to $\la za^*,Uv\ra.$
On the other hand, $\|(U^*z_N)a^*-(U^*z)a^*\| \rightarrow 0$ as $N$ tends to infinity; hence the right hand side in (\ref{adjoint ending}) converges to $\la (U^*z)a^*,v\ra.$ Thus, by letting $N\rightarrow \infty$ in (\ref{adjoint ending}), we obtain
$$a\la z,Uv\ra=a\la U^*z,v\ra,\quad\forall a\in \textsf{A},\quad \forall v\in M(X)$$
or, equivalently,
$$
a(\la z,Uv\ra-\la U^*z,v\ra)=0,\quad\forall a\in \textsf{A},\quad \forall v\in M(X).
$$
This is enough to conclude $\la z,Uv\ra=\la U^*z,v\ra$ for each $v\in M(X).$ As $z$ was arbitrary element of
$M(\ell^2(\textsf{A})),$ we have finally proved that $U$ is an adjointable operator.
Recall that $U^*$ is given by
$$
U^*(b_1,b_2,b_3,\ldots)=(\mbox{strict})\lim_{N\rightarrow \infty}U^*(\sum_{n=1}^Ne^{(n)}b_n)=
(\mbox{strict})\lim_{N\rightarrow \infty}\sum_{n=1}^Nv_nb_n$$
for each $(b_1,b_2,b_3,\ldots)\in M(\ell^2(\textsf{A})).$ In other words,
$$
U^*(b_1,b_2,b_3,\ldots)=(\mbox{strict})\sum_{n=1}^{\infty}v_nb_n,\quad\forall (b_1,b_2,b_3,\ldots) \in M(\ell^2(\textsf{A})).
$$

Furthermore, $U^*$ is, being adjointable, norm-continuous. Since for each $(a_1,a_2,a_3,\ldots) \in \ell^2(\textsf{A})$ we have
$(a_1,a_2,a_3,\ldots)=\lim_{N\rightarrow \infty}\sum_{n=1}^Ne^{(n)}a_n$ with convergence in norm, this gives us
$U^*(a_1,a_2,a_3,\ldots)=\sum_{n=1}^{\infty}v_na_n,$ where this series converges with respect to the norm.
Hence, $U_X : X \rightarrow \ell^2(\textsf{A})$ is also an adjointable operator and we have inequalities
$$
A\|v\|^2\le \|Uv\|^2\le B\|v\|^2,\quad\forall v\in M(X),
$$
$$
A\|x\|^2\le \|U_Xx\|^2\le B\|x\|^2,\quad\forall x\in X.
$$
Proposition 2.1 from \cite{A} now implies
$$
A\langle v,v\ra \le (\mbox{strict})\,\sum_{n=1}^{\infty}\la v,v_n\ra \la v_n,v\ra \le B \la v,v\ra ,\quad\forall v\in M(X),
$$
$$
A\langle x,x\ra \le \sum_{n=1}^{\infty}\la x,v_n\ra \la v_n,x\ra \le B \la x,x\ra ,\quad\forall x\in X.
$$
In particular, since $U$ and $U_X$ are bounded from below, $U^*$ and $(U_X)^*$ are surjective.
\end{proof}

\vspace{.1in}

\begin{prop}\label{surjections preserve frames}
Let $X$ and $Y$ be Hilbert $C^*$-modules
and $T\in \Bbb B(M(X),M(Y)).$ Then $T$ maps strict frames for $M(X)$ to strict frames for $M(Y)$ if and only if $T$ is surjective.
\end{prop}
\begin{proof}
Let $(v_n)_n$ be a strict frame for $M(X)$ and $T\in \Bbb B(M(X),M(Y))$ a surjective operator. Denote by $A$ and $B$ the frame bounds of $(v_n)_n.$ Since $T$ is a surjection, $T^*$ is bounded from below, so there exists $m>0$
such that $\|T^*w\|\geq m \|w\|$ for all $w \in M(Y).$
Let $w_n=Tv_n,\,n\in \Bbb N.$
Observe that, for each $w\in M(Y),$ we have $\la  w,w_n\ra \la w_n,w\ra=\la T^*w,v_n\ra \la v_n,T^*w\ra.$ Therefore, there exists
$(\mbox{strict})\,\sum_{n=1}^{\infty}\la  w,w_n\ra \la w_n,w\ra.$ Moreover, for each $w\in M(Y),$
\begin{eqnarray*}
\left\|(\mbox{strict})\,\sum_{n=1}^{\infty}\la  w,w_n\ra \la w_n,w\ra\right\|&=&\left\|(\mbox{strict})\,\sum_{n=1}^{\infty}\la T^*w,v_n\ra \la v_n,T^*w\ra\right\|\\
 &\geq&A\left\|\left\langle T^*w,T^*w\right\rangle\right\|\\
 &=&A\left\|T^*w\right\|^2\\
 &\geq&Am^2\|w\|^2.
\end{eqnarray*}
By Theorem~\ref{outer frame relaxed}, $(w_n)_n$ is a strict frame for $M(Y).$

Conversely, suppose that  $T\in \Bbb B(M(X),M(Y))$ preserves strict frames. So, if $(v_n)_n$ is a strict frame for $M(X),$ then $(Tv_n)_n$ is  a strict frame for $M(Y).$ If $U$ is the analysis operator for $(v_n)_n,$ then $UT^*$ is the analysis operator for $(Tv_n)_n,$ so $TU^*$ is the corresponding synthesis operator.
By Theorem~\ref{outer frame relaxed}, $TU^*$ is surjective; hence, $T$ must be surjective.
\end{proof}

\vspace{.1in}

As a direct consequence we now get an analog of Proposition~\ref{frames vs surjections} for strict frames. Note that here we do not need any assumptions on the underlying $C^*$-algebra $\A.$ In fact, when $\A$ is unital the statement of the following corollary reduces, due to Remark~\ref{outer vs standard} and Remark
\ref{multiplier modules}(a), to Proposition~\ref{frames vs surjections}.

\begin{cor}\label{Msurjekcije}
If $T\in \Bbb B(M(\ell^2(\textsf{A})),M(X))$ is a surjection, then there exists a unique strict frame for $M(X)$ whose synthesis operator is equal to $T.$
\end{cor}
\begin{proof}
Since $(e^{(n)})_n$ is a strict frame for $M(\ell^2(\textsf{A}))$ (see Example~\ref{kanonska striktna}), the preceding proposition implies that $(v_n)_n$ defined by  $v_n=Te^{(n)},\,n\in \Bbb N,$ is a strict frame for $M(X).$ Its synthesis operator $U^*$ also satisfies, by the last assertion of Theorem~\ref{outer frame relaxed}, $U^*e^{(n)}=v_n$ for all $n \in \Bbb N.$ This implies that $U^*$ and $T$ coincide on $c_{00}(\A).$ Since $c_{00}(\A)$ is strictly dense in $M(\ell^2(\A))$ and both $U^*$ and $T$ are strictly continuous (see Remark~\ref{multiplier modules}(d)), this is enough to conclude that $U^*=T.$
Uniqueness is evident.
\end{proof}

\vspace{.1in}

Next we show the reconstruction property of strict frames.

Let $X$ be a Hilbert $\textsf{A}$-module. Suppose that $(v_n)_n$ is a strict frame for  $M(X)$ with the analysis operator $U\in \Bbb B(M(X), M(\ell^2(\textsf{A}))).$ Then $U^*U$ is an invertible operator for which we have
$$
U^*Uy=(\mbox{strict})\sum_{n=1}^{\infty}v_n\la v_n,y\ra,\quad\forall y\in M(X).
$$
If we put $U^*Uy=v,$ this can be rewritten as
$$
v=(\mbox{strict})\sum_{n=1}^{\infty}v_n\la v_n,(U^*U)^{-1}v\ra=(\mbox{strict})\sum_{n=1}^{\infty}v_n\la (U^*U)^{-1}v_n,v\ra$$
for all $v\in M(X).$
Let $w_n=(U^*U)^{-1}v_n,\,n\in \Bbb N.$ By the preceding corollary $(w_n)_n$ is also a strict frame for $M(X)$ that satisfies
\begin{equation}\label{vanjski duali}
v=(\mbox{strict})\sum_{n=1}^{\infty}v_n\la w_n,v\ra,\quad\forall v\in M(X).
\end{equation}
If we denote by $V$ the analysis operator of $(w_n)_n$ then the above equality can we rewritten as $U^*V=I.$ This obviously implies $V^*U=I,$ so we also have
\begin{equation}\label{vanjski duali bis}
v=(\mbox{strict})\sum_{n=1}^{\infty}w_n\la v_n,v\ra,\quad\forall v\in M(X).
\end{equation}

By the last part of the proof of Theorem~\ref{outer frame relaxed} the above two equalities give us, with the respect to the norm topology on $X,$
\begin{equation}\label{vanjski unutra}
x=\sum_{n=1}^{\infty}v_n\la w_n,x\ra=\sum_{n=1}^{\infty}w_n\la v_n,x\ra,\quad\forall x\in X.
\end{equation}

In particular, if $(v_n)_n$ is a strict Parseval frame for $M(X),$ the above equalities reduce to
\begin{equation}\label{vanjski duali parseval}
v=(\mbox{strict})\sum_{n=1}^{\infty}v_n\la v_n,v\ra,\quad\forall v\in M(X),
\end{equation}
and
\begin{equation}\label{vanjski unutra parseval}
x=\sum_{n=1}^{\infty}v_n\la v_n,x\ra,\quad\forall x\in X.
\end{equation}

\vspace{.1in}

\begin{remark}\label{countably generated}
If $(x_n)_n$ is a frame for $X,$ the reconstruction property from Theorem~\ref{full reconstruction} shows that $X$ is countably generated. Similarly, if $(v_n)_n$ is a strict frame for $M(X),$ the reconstruction formula \eqref{vanjski unutra} shows that $X$ is countably generated by $v_n$'s and the only difference is that here the generators (frame members) are elements of $M(X)$ and need not belong to $X.$ This property is introduced and discussed in \cite{RT}. By Definition 2.1 from \cite{RT}, a Hilbert $\A$-module $X$ is countably generated in $M(X)$ if there exists a sequence $(v_n)_n$ in $M(X)$ such that the set $\text{span}\{v_na:n\in \Bbb N, a\in \A\}$ is norm-dense in $X$. It is proved in \cite{RT} that each Hilbert $C^*$-module $X$ that is countably generated in $M(X)$ possesses a Parseval frame or an outer Parseval frame. Finally, the reconstruction property  (\ref{vanjski unutra parseval}) for such frames is derived in Theorem 3.4 in \cite{RT}.
\end{remark}

\vspace{.1in}

\begin{remark}\label{parseval frames are approximate units strict version}
Having obtained the reconstruction formulae for strict frames (in particular, (\ref{vanjski unutra parseval})), we can now generalize the statement of Proposition~\ref{parseval frames are approximate units} in the following way:
Let $X$ be a Hilbert $C^*$-module. Then a sequence $(v_n)_n$ of elements of $M(X)$ is a strict Parseval frame for $X$ if and only if the sequence $(\sum_{n=1}^N\theta_{v_n,v_n})_N$ has the property
$T=\lim_{N\rightarrow \infty}T(\sum_{n=1}^N\theta_{v_n,v_n})$ for each $T$ in $\Bbb K(X).$

The proof is in fact the same as that of Proposition~\ref{parseval frames are approximate units}, so we omit the details.
\end{remark}

\vspace{.1in}

We are now ready to extend Proposition~\ref{frames vs surjections} to the non-unital case. To do that, we first show that the class of strict frames for the multiplier module $M(X)$ consists precisely of all frames and outer frames for $X.$ Let us start with the following technical result.

\begin{lemma}\label{prosirivanje surjekcija}
Let $X$ and $Y$ be Hilbert $\A$-modules, and $T\in \Bbb B(X,Y).$ Let $T_M\in \Bbb B(M(X),M(Y))$ be the strict extension of $T.$
 \begin{itemize}
   \item[(a)] If  $A$ and $B$ are some positive constants, then
\begin{equation}\label{odozdo omedjen U}
    A\|x\|^2\leq \|Tx\|^2\leq B\|x\|^2,\quad\forall x\in X,
\end{equation}
    if and only if it
\begin{equation}\label{odozdo omedjen U_M}
A\|v\|^2\leq \|T_Mv\|^2\leq B\|v\|^2,\quad\forall v\in M(X).
\end{equation}
   \item[(b)]  $T$ is a surjection if and only if $T_M$ is a surjection.
\end{itemize}
\end{lemma}
\begin{proof}
Let us prove (a).
If $M(X)=X,$ there is nothing to prove. Hence, we assume that $M(X)\not =X.$
Obviously, \eqref{odozdo omedjen U_M} implies \eqref{odozdo omedjen U}, so we only need to prove the converse. First we recall a useful result from Remark~\ref{multiplier modules}(a), namely
\begin{equation}\label{surjections generic3}
\|v\|=\sup\{\|va\|: a \in \textsf{A}, \|a\|\le 1\},\quad\forall v\in M(X).
\end{equation}
By applying this formula to $U_Mv\in M(Y)$ for $v\in M(X),$ we get
\begin{eqnarray*}
\|U_Mv\|&=&\sup\left\{\|(U_Mv)a\|: a \in \textsf{A}, \|a\|\le 1\right\}\\
&=& \sup\left\{\|U_M(va)\|: a \in \textsf{A}, \|a\|\le 1\right\}\\
&&\textup{(since $va\in X$)}\\
&=&\sup\left\{\|U(va)\|: a \in \textsf{A}, \|a\|\le 1\right\}.
\end{eqnarray*}
Using the first inequality from the hypothesis we get
$$\|U_Mv\|\ge \sqrt{A}\, \sup\left\{\|va\|: a \in \textsf{A}, \|a\|\le 1\right\}
\stackrel{(\ref{surjections generic3})}{=}\sqrt{A}\,\|v\|,$$
while the second inequality from the hypothesis gives us
$$\|U_Mv\|\le \sqrt{B}\, \sup\left\{\|va\|: a \in \textsf{A}, \|a\|\le 1\right\}
\stackrel{(\ref{surjections generic3})}{=}\sqrt{B}\,\|v\|.$$

To prove (b) suppose first that $T$ is a surjection. If $M(Y)=Y$ then, trivially, $T_M$ is a surjection. So, let us assume that $M(Y)\not =Y$, observe that $T^*$ is bounded from below and hence, by (a), that $(T^*)_M$ is also bounded from below.
Recall from Remark~\ref{multiplier modules}(d) that $(T^*)_M=(T_M)^*.$
By the preceding conclusion we know that  $(T_M)^*$ is bounded from below; thus, $T_M$ is a surjection.
Suppose now that $T_M$ is a surjection. Then $(T_M)^*=(T^*)_M$ is bounded from below. By (a), $T^*$ is also bounded from below which implies that $T$ is a surjection.
\end{proof}

\begin{theorem}\label{outer frames are strict frames}
Let $X$ be a Hilbert $\textsf{A}$-module and $(x_n)_n$ a sequence in $M(X).$ Then  $(x_n)_n$ is a strict frame for $M(X)$ if and only if  $(x_n)_n$ is a frame or an outer frame for $X.$
\end{theorem}
\begin{proof}
If $M(X)=X$ then, by Remarks~\ref{no outer frames} and~\ref{outer vs standard}, there is nothing to prove. So, let us assume that $M(X)\not =X.$
One direction is proved in Proposition~\ref{outer is fine}.

Suppose $(x_n)_n$ is a frame or an outer frame for $X$ with frame bounds $A$ and $B.$
Denote by $U\in \Bbb B(X,\ell^2(\textsf{A}))$ the corresponding analysis operator. Then $U$ is bounded by $\sqrt{B}$ and bounded from below by $\sqrt{A}$ - if $(x_n)_n$ is a frame this is already observed in the introduction, and if $(x_n)_n$ is an outer frame Proposition~\ref{analysis op outer f} applies. So,
$$
A\|x\|^2\le \|Ux\|^2\le B\|x\|^2,\quad\forall x\in X.
$$
Let us now consider the extended operators
$U_M\in \Bbb B(M(X),M(\ell^2(\textsf{A})))$ and
$(U^*)_M\in \Bbb B(M(\ell^2(\textsf{A})),M(X)).$
By Lemma~\ref{prosirivanje surjekcija} we now know that
$$
A\|v\|^2\le \|U_Mv\|^2\le B\|v\|^2,\quad\forall v\in M(X),
$$
and $(U_M)^*=(U^*)_M$ is a surjection. Observe that
$(U^*)_Me^{(n)}=x_n$ for all $n \in \Bbb N.$
By Proposition~\ref{surjections preserve frames}, $(x_n)_n$ is a strict frame for $M(X).$
\end{proof}

\vspace{.1in}

We are now in position to prove the key result of this section.

\begin{theorem}\label{stvarna korespondencija frameova i sinteza}
Let $X$ be a Hilbert $\textsf{A}$-module and let $T\in \Bbb B(\ell^2(\A),X)$ be a surjection. Then there exists a unique frame or outer frame $(x_n)_n$ for $X$ whose synthesis operator coincides with $T.$
\end{theorem}
\begin{proof}
If $\textsf{A}$ is unital, Proposition~\ref{frames vs surjections} applies.

Assume that $\textsf{A}$ is non-unital. By Lemma~\ref{prosirivanje surjekcija}, $T_M\in \Bbb B(M(\ell^2(\A)),M(X))$ is also a surjection.
By Corollary~\ref{Msurjekcije}
the sequence $(v_n)_n$ defined by $v_n=T_Me^{(n)},\,n\in \Bbb N,$ is a strict frame for $M(X)$ whose synthesis operator is equal to $T_M.$
By Proposition~\ref{outer is fine}, $(v_n)_n$ is a frame or an outer frame for $X$ depending on whether all $v_n$'s belong to $X$ or not.
Denote by $U$ the corresponding analysis operator.

By definitions of a frame and an outer frame, we have that $(\langle v_n,x\rangle)_n\in\ell^2(\A)$ for all $x\in X.$ By Proposition~\ref{analysis op outer f} (and the corresponding property of frames observed in the introduction)
we know that $U^*((a_n)_n)=\sum_{n=1}^{\infty}v_na_n,$ where this series converges in norm of $X$ for all $(a_n)_n\in\ell^2(\A).$
It follows from \eqref{synthesis-for-strict} that
$$T((a_n)_n)=T_M((a_n)_n)=\sum_{n=1}^{\infty}v_na_n=U^*((a_n)_n),\quad \forall (a_n)_n\in \ell^2(\textsf{A}).$$
\end{proof}

\vspace{0.15in}

The preceding theorem concludes our description of various classes of frames in terms of corresponding adjointable surjections (i.e., synthesis operators).
The most important statements of this section and their mutual relations are shown in the diagram below.

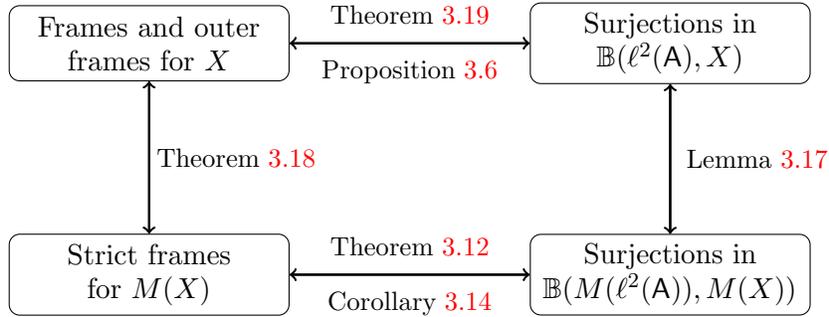
\vspace{0.15in}
\begin{center}
\begin{tikzpicture}
\node[block](fof){Frames and outer frames for $X$};
\node[block, below of =fof](sf){Strict frames for $M(X)$};
\node[block, right of =fof, xshift=10em](surjX){Surjections in $\Bbb B(\ell^2(\textsf{A}),X)$};
\node[block, right of =sf, xshift=10em](surjMX){Surjections in $\Bbb B(M(\ell^2(\textsf{A})),M(X))$};
\path[line](fof)--node[xshift=3em]{\small{Theorem~\ref{outer frames are strict frames}}} (sf);
\path[line](sf)-- (fof);
\path[line](surjX)--node[xshift=3em]{\small{Lemma~\ref{prosirivanje surjekcija}}} (surjMX);
\path[line](fof)--node[yshift=1em]{\small{Theorem~\ref{stvarna korespondencija frameova i sinteza}}}(surjX);
\path[line](fof)--node[yshift=-1em]{\small{Proposition~\ref{analysis op outer f}}}(surjX);
\path[line](sf)--node[yshift=-1em]{\small{Corollary~\ref{Msurjekcije}}} (surjMX);
\path[line](surjMX)--node[yshift=1em]{\small{Theorem~\ref{outer frame relaxed}}}(sf);
\path[line](surjX)--node[xshift=3em]{}(surjMX);
\end{tikzpicture}
\end{center}

\vspace{0.3in}

\begin{remark}\label{more than this}
Observe the bottom row of the above diagram: Theorem~\ref{outer frame relaxed} and Corollary~\ref{Msurjekcije} establish a correspondence of strict frames for $M(X)$ with adjointable surjections from $M(\ell^2(\textsf{A}))$ to $M(X).$
On the other hand, since $M(\A)$ is a unital $C^*$-algebra, frames for $M(X)$ correspond, by Proposition~\ref{frames vs surjections}, to adjointable surjections from $\ell^2(M(\textsf{A}))$ to $M(X).$

It is clear from the definition of a strict frame that the class of strict frames for $M(X)$ contains the class of frames for $M(X).$ This reflects the fact that, in general,
$M(\ell^2(\textsf{A}))$ is larger than  $\ell^2(M(\textsf{A}))$.

As an example of a strict frame which is not a frame for $M(X)$  take again a strict Parseval frame $(s_n)_n$ for $M(\Bbb K(H))=\Bbb B(H)$ from Example~\ref{fail 2 revisited}. To see that $(s_n)_n$ is not a frame for $\Bbb B(H),$ suppose the opposite. Then we would have $b=\sum_{n=1}^{\infty}s_n\la s_n,b\ra= \sum_{n=1}^{\infty}p_nb$ with the norm convergence for all $b\in \Bbb B(H),$ which is obviously impossible.
\end{remark}

\vspace{.1in}

For reader's convenience we include a short overview of the preceding considerations concerning various classes of frames and their interrelations.

Let $X$ ba a countably generated Hilbert $\A$-module.
\begin{itemize}
\item If $X=M(X)$, i.e., if $X$ is strictly complete (e.g., when $\A$ is unital or when $X$ is AFG) then there are no outer frames for $X$ (Remark~\ref{no outer frames}), strict frames coincide with frames (Theorem~\ref{outer frames are strict frames})
and each surjection in $\Bbb B(\ell^2(\A),X)$ is the synthesis operator of some frame for $X$ (Theorem~\ref{stvarna korespondencija frameova i sinteza}).
\item If $X\not = M(X)$ then, in particular, $\A$ is non-unital and $X$ is not AFG. The class of all strict frames for $M(X)$ consists of two disjoint parts which are the classes of all frames for $X$ and all outer frames for $X$ (see the diagram below). Moreover, if the multiplier module $M(X)$ is countably generated itself, then the class of outer frames for $X$ contains as a subset all frames for $M(X)$. Each surjection in $\Bbb B(\ell^2(\A),X)$ is the synthesis operator of some either frame or outer frame for $X$ (Theorem~\ref{stvarna korespondencija frameova i sinteza}).
\end{itemize}

\vspace{0.15in}
\begin{center}
    \begin{tikzpicture}
\draw[thick] ([shift=(180:1.7cm)]4,3) arc (180:305:1.7cm);
      \node [rotate=330] (c1) at (3.7,2.4){Frames};
      \node [rotate=330] (c2) at (3.5,1.9) {for $M(X)$};
      \draw [thick] (2,1.5) ellipse (3cm and 1.5cm);
      \draw [thick] (1.5,0)--(1.5,3);
      \node (a1) at (0.75,1.5) {Frames};
      \node (a2) at (0.75,1) {for $X$};
      \node [opacity=1](b1) at (2.7,1.5) {Outer frames};
      \node [opacity=1](b2) at (2.7,1) {for $X$};
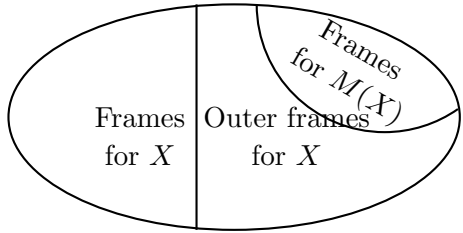
      \node (a2) at (2,-0.5) {Strict frames for $M(X)$ in the case $X\neq M(X).$};
     \end{tikzpicture}
\end{center}
\vspace{0.15in}

\vspace{.2in}

Finally, we include for future reference an easy consequence of the preceding results concerning Bessel sequences.
The following corollary, together with Theorem~\ref{Bessel relaxed} and Corollary~\ref{adjointable operators are Bessel sequences 1}, establishes a bijective correspondence  between
Bessel and outer Bessel sequences in a Hilbert $C^*$-module $X$ and adjointable operators from $\ell^2(\A)$ to $X.$

\begin{cor}\label{adjointable operators are Bessel sequences 2}
Let $X$ be a Hilbert $\textsf{A}$-module and let $T\in \Bbb B(\ell^2(\A),X).$ Then there exists a unique Bessel sequence or outer Bessel sequence $(x_n)_n$ in $X$ whose synthesis operator coincides with $T.$
\end{cor}
\begin{proof}
If $\A$ is unital, put $Te^{(n)}=x_n,\,n\in \Bbb N.$ Then, obviously, $T^*x=(\la x_n,x\ra)_n$ for all $x\in X,$ and, since $T^*$ is bounded, $(x_n)_n$ is a Bessel sequence whose synthesis operator coincides with $T.$

If $\A$ is non-unital, put $T_Me^{(n)}=x_n,\,n\in \Bbb N.$ Again, it follows that $T^*x=(\la x_n,x\ra)_n$ for all $x\in X.$ Moreover, $\sum_{n=1}^{\infty}\la x,x_n\ra \la x_n,x\ra$ converges in norm for all $x \in X.$ From this we conclude that $(x_n)_n$ is a Bessel sequence, which turns out to be outer if at least one $x_n$ belongs to $M(X)\setminus X.$
\end{proof}

\vspace{.1in}

We end the section with some comments on finite frames vs.~adjointable surjections from $\A^N,$ $N\in \Bbb N,$ to the ambient Hilbert $\A$-module $X.$

Let us first extend Definition~\ref{real outer frames} to finite sequences: if $X$ is a Hilbert $\A$-module, we say that a finite sequence $(v_n)_{n=1}^N,$ $N\in \Bbb N,$ in $M(X)$ is an outer frame
for $X$ if  $v_n\in M(X)\setminus X$ for at least one $n\in \{1,\ldots , N\},$ and if there exist positive constants $A$ and $B$ such that
$$
A\la x,x\ra\le \sum_{n=1}^{N}\langle x,v_n\rangle \langle v_n,x\rangle\le B\la x,x\ra,\quad\forall x \in X.
$$
The  analysis operator $U$ is defined as
$$U:X \rightarrow \A^N,\quad Ux=(\la v_n,x\ra)_{n=1}^N.$$
Its adjoint, the synthesis operator $U^*$ is given by
$$U^*(a_1,\ldots , a_N)=\sum_{n=1}^Nv_na_n.$$

\vspace{.1in}

\begin{prop}\label{finite prva napomena}
Let $X$ be a Hilbert $\A$-module.
\begin{itemize}
\item[(a)] If there exists a finite frame for $X,$ then $X$ is AFG and, in particular, there are no outer frames for $X.$
\item[(b)] If there exists a finite outer frame for $X$ then $M(X)\not =X,$ $X$ is not AFG, and $\A$ is a non-unital $C^*$-algebra.
Moreover, then there are no finite frames for $X$ and each finite outer frame for $X$ is a frame for $M(X).$
\end{itemize}
\end{prop}
\begin{proof}
To prove (a), we only need to recall that, by Remark~\ref{multiplier modules}(a), $M(X)=X$ when $X$ is AFG.

Similarly, if there exists an outer frame for $X$ then, by definition, $M(X) \not =X$ and again Remark~\ref{multiplier modules}(a) implies that then
$X$ is not AFG and $\A$ is non-unital. To prove the last statement in (b) we can argue as follows.

First, observe that $M(\A^N)=M(\A)^N.$ Suppose that  $(v_n)_{n=1}^N$ is an outer frame for $X,$ consider the analysis operator $U\in \Bbb B(X,\A^N)$ and its extension $U_M\in \Bbb B(M(X), M(\A)^N).$ By Lemma~\ref{prosirivanje surjekcija}, $(U_M)^* \in \Bbb B(M(\A)^N, M(X))$ is a surjection. Let $w_n=(U_M)^*e^{(n)},$ $n=1,\ldots , N.$ As in the proof of Proposition~\ref{frames vs surjections}, we easily conclude that $(w_n)_{n=1}^N$ is a frame for $M(X).$ We now claim that $w_n=v_n$ for all $n=1,\ldots , N.$ To see this, take arbitrary $n\in \{1,\ldots , N\}$ and $a\in \A.$
Then
$$
(U_M)^*a^{(n)}=(U_M)^*(e^{(n)}a)=\left((U_M)^*e^{(n)}\right)a=w_na,
$$
and
$$
(U_M)^*a^{(n)}=(U^*)_Ma^{(n)}=U^*a^{(n)}=v_na.
$$
Thus, $w_na=v_na$ and, since $a$ was arbitrary, this is enough to conclude $w_n=v_n.$
\end{proof}

\begin{prop}\label{finite druga napomena}
Let $X$ be a Hilbert $\A$-module and $T\in \Bbb B(\A^N, X),$ $N\in \Bbb N,$ a surjection. Then there exists a unique frame or outer frame $(x_n)_{n=1}^N$ for $X$ whose synthesis operator coincides with $T.$
\end{prop}
\begin{proof}
If $\A$ is unital let $x_n=Te^{(n)},$ $n=1,\ldots N.$ Then, clearly, $(x_n)_{n=1}^N$ is a frame for $X$ whose synthesis operator is $T.$

Suppose now that $\A$ is non-unital. Consider $T_M \in \Bbb B(M(\A)^N, M(X))$ and put $x_n=T_Me^{(n)},$ $n=1,\ldots, N.$ Since, by Lemma~\ref{prosirivanje surjekcija}, $T_M$ is a surjection, $(x_n)_{n=1}^N$ is a frame for $M(X)$ by \cite[Theorem~2.5]{A}.
There are now two possibilities: either each $x_n$ belongs to $X,$ or $x_n\in M(X)\setminus X$ for at least one $n.$

Assume first $x_n\in X$ for all $n=1,\ldots , N.$ By the reconstruction property it follows immediately that $M(X)\subseteq X$; thus, in fact $(x_n)_{n=1}^N$ is a frame for $X$ and, in particular, $X$ is AFG.

In the remaining possibility, if there exists $n$ such that $x_n\in M(X)\setminus X,$ $(x_n)_{n=1}^N$ is an outer frame for $X$ and, in particular, Proposition~\ref{finite prva napomena}(b) applies.

In both cases the corresponding synthesis operator coincides with $T$.
\end{proof}

\vspace{.1in}

Note that the situation described in Proposition~\ref{finite prva napomena}(b) means that $X,$ although not algebraically generated by finitely many elements, admits finite outer frames. It is not difficult to find examples of such Hilbert $C^*$-modules. In fact, every non-unital $C^*$-algebra $\A$ serves as a simple example of this kind.

To see this, take any non-unital $C^*$-algebra $\A$ and regard it as a Hilbert $\A$-module over itself. Since $\Bbb K(\A)=\A$ is non-unital, $\A$ is not AFG as a Hilbert $C^*$-module, and therefore there are no finite frames for $\A.$
On the other hand, here the multiplier algebra $M(\A)$ plays the role of the multiplier module, so the unit element $e\in M(\A)$ serves as a frame for $M(\A)$ and an outer frame for $\A.$ This is, indeed, obvious from the equality $a=e\la e,a\ra$ that is trivially satisfied for all $a\in \A.$

\vspace{.1in}

Having obtained necessary results on outer frames we are now ready for a detailed study of various questions (such as dual frames, perturbations, tight approximations, etc) that are prominent for the frame theory.
This is the purpose of the second part of the paper. In our study we shall be interested primarily in countably generated Hilbert $C^*$-modules and their frames, but as we shall see, outer frames will naturally appear into the picture.
Hence the results that follow will be concerned with both frames and outer frames.

It should be noted that some of that results are valid even for Hilbert $C^*$-modules that are not countably generated (in the usual sense), but which are countably generated in $M(X)$.

\vspace{0.2in}


\section{Dual frames}

\vspace{0.2in}

Dual frames for Hilbert $C^*$-modules were introduced and discussed in \cite{FL2}, Section 6, where the existence of canonical and alternate dual frames was established, and some of their fundamental properties were proven.

Suppose that $(x_n)_n$ is a frame for a Hilbert $\A$-module $X$ with the analysis operator $U.$ Then we know that the sequence $(y_n)_n= ((U^*U)^{-1}x_n)_n$ is also a frame for $X$ which satisfies
$x=\sum_{n=1}^{\infty}y_n\la x_n,x\ra$ for all $x \in X.$ The frame $(y_n)_n$ is called the canonical dual of $(x_n)_n.$ If we denote its analysis operator by $V$ then the preceding equality can be rewritten in the form $V^*U=I.$ It is now natural to try to describe all frames $(z_n)_n$ for $X$ that are dual to $(x_n)_n$ in the sense that the equality
$x=\sum_{n=1}^{\infty}z_n\la x_n,x\ra$ is satisfied for each $x$ in $X.$ If $W$ denotes the analysis operator of $(z_n)_n,$ this simply means $W^*U=I.$

Hence, the problem of finding dual frames of $\xn$ is closely related to the problem of finding solutions of the equation $TU=I$ with  $T \in \Bbb B(\ell^2(\A),X).$ Obviously, each $T$ such that $TU=I$ is surjective.
When $\A$ is unital, we know by Proposition~\ref{frames vs surjections} that such $T$ is the synthesis operator of some frame for $X,$ and one immediately concludes (see Lemma~\ref{dual relaxed} below) that the obtained frame is dual to $\xn.$

The non-unital case is more complicated because among solutions of $TU=I$ there might be adjointable surjections which are not synthesis operators of frames for $X.$ However, by Theorem~\ref{stvarna korespondencija frameova i sinteza}, such surjections are synthesis operators of outer frames for $X,$ and it will turn out that each outer frame $\yn$ obtained in that way  also satisfies $x=\sum_{n=1}^{\infty}y_n\la x_n,x\ra$ for all $x \in X.$
(Indeed, outer duals do exist; see Examples~\ref{an outer dual} and \ref{an outer dual 1} below.)

Therefore, by solving the equation  $TU=I$ in $\Bbb B(\ell^2(\A),X)$ we shall get synthesis operators of both frames and outer frames for $X$ dual to a given frame.

This suggests a need for a unified treatment of dual frames, without a priori distinguishing between frames and dual frames.

\vspace{0.1in}

Before embarking into our study, let us point out that here we shall restrict ourselves to (outer) frames and (outer) Bessel sequences. That is, we are not going to discuss general sequences that behave like duals to a given frame. Recall that even in a Hilbert space in some situations there are sequences that are not even Bessel, but which are dual to a given frame.

\vspace{.1in}

Throughout this section all our statements are concerned only with infinite frames and outer frames. A short remark about the finite case is included at the end of the section.

Further, if $Y$ is a complementable closed Hilbert $C^*$-submodule of $X$ (i.e.,  $X=Y\oplus Y^{\perp}$), we denote by $P_Y\in \Bbb B(X)$ the orthogonal projection to $Y.$ Recall that a closed Hilbert $C^*$-submodule $Y$ of $X$ is complementable in $X$ if and only if $Y$ is the range of an adjointable operator (see e.g. \cite[Corollary~15.3.9]{W-O}).

Let us start with a definition.

\begin{definition} \label{dual definition}
Let $X$ be a Hilbert $\A$-module and $\xn$ a frame or an outer frame for $X.$ A frame or an outer frame $(y_n)_n$ for $X$ is said
to be a \emph{dual to $\xn$} if
\begin{equation}\label{dual formula a}
\sum_{n=1}^\infty y_n\la x_n,x\ra=x,\quad \forall x\in X.
\end{equation}
\end{definition}

\begin{remark}
Let $(x_n)_n$ and $(y_n)_n$ be as in the above definition; denote by $U$ and $V$ the analysis operators of $(x_n)_n$ and $(y_n)_n,$ respectively. Then, obviously, (\ref{dual formula a}) can be rewritten as
\begin{equation}\label{dual formula short a}
V^*U=I,
\end{equation}
which is equivalent to
\begin{equation}\label{dual formula short}
U^*V=I,
\end{equation}
or
\begin{equation}\label{dual formula}
\sum_{n=1}^\infty x_n\la y_n,x\ra=x,\quad \forall x\in X.
\end{equation}
Hence, as long as we work with frames and outer frames ({\em not mere sequences}), equalities (\ref{dual formula a}) to (\ref{dual formula}) are mutually equivalent,
duality is a symmetric relation, and we can say that $\xn$ and $\yn$ are dual to each other.

Moreover, our first lemma will show, generalizing \cite[Proposition~3.8]{HJLM}, that the same is true for Bessel sequences and outer Bessel sequences.
\end{remark}

\begin{lemma}\label{dual relaxed}
Let $X$ be a Hilbert $\A$-module. If  $\xn¸$ and $\yn$ are Bessel or outer Bessel sequences in $X$ with the analysis operators $U$ and $V,$ respectively, satisfying at least one of equalities (\ref{dual formula a}) to (\ref{dual formula}), then $\xn$ and $\yn$ are frames or outer frames for $X$, they satisfy all equalities (\ref{dual formula a}) to (\ref{dual formula}), and are dual to each other.
\end{lemma}
\begin{proof}
First note, by Theorem~\ref{Bessel relaxed} and Corollaries~\ref{adjointable operators are Bessel sequences 1} and \ref{adjointable operators are Bessel sequences 2}, that  Bessel sequences and outer Bessel sequences correspond to adjointable operators from $X$ to $\ell^2(\A),$ so that not only all four above equalities make sense, but also each of them implies the remaining three.

So, suppose that equalities (\ref{dual formula a}) to (\ref{dual formula}) hold. From (\ref{dual formula short a}) we get that $V^*$ is a surjection. By Theorem~\ref{stvarna korespondencija frameova i sinteza}, $(y_n)_n$ is a frame or an outer frame for $X.$
By invoking (\ref{dual formula short}), the same argument applies to $(x_n)_n.$
\end{proof}

\vspace{.1in}

If $X$ is a strictly complete Hilbert $C^*$-module (i.e.,  if $M(X)=X$) our discussion on duality reduces to frames since then there are no outer frames. When $M(X)\not =X,$ the situation is more complicated. If $(x_n)_n$ is a frame for $X,$ its canonical dual is also a frame. On the other hand, if $(x_n)_n$ is an outer frame for $X,$ its canonical dual is also outer. Both statements follow from the fact that the $(U^*U)^{-1}$ acts bijectively on $X.$
In general, a frame $(x_n)_n$ for $X$ may have outer dual frames. Indeed, in two examples that follow we demonstrate that:
\begin{itemize}
\item each countably generated Hilbert $C^*$-module $X$ such that $M(X)\not =X$ has a frame which possesses an outer dual frame,
\item there exists a frame for a countably generated Hilbert $C^*$-module $X$ possessing an outer dual frame whose all elements are in $M(X)\setminus X.$
\end{itemize}

\begin{example}\label{an outer dual}
Let $X$ be a countably generated Hilbert $\textsf{A}$-module that is not strictly complete, i.e.,  $X \not = M(X).$ Take any Parseval frame $(x_n)_{n=1}^{\infty}$ for $X.$ Let $x_0\in X$ and $y_0\in M(X)\setminus X$ be such that $\theta_{y_0,x_0}=0$ (for example, we can take $x_0=0$ and arbitrary $y_0\in M(X)\setminus X$). Let $y_n=x_n$ for $n\in \Bbb N.$

Then $(x_n)_{n=0}^{\infty}$ is a frame for $X,$ $(y_n)_{n=0}^{\infty}$ is an outer frame for $X,$ and they are dual to each other since
$$\sum_{n=0}^{\infty}y_n\la x_n,x\ra=y_0\la x_0,x\ra+  \sum_{n=1}^{\infty}x_n\la x_n,x\ra =x,\quad \forall x \in X.$$

\end{example}

\vspace{.1in}

\begin{example}\label{an outer dual 1}
Let $(\epsilon_n)_n$ be an orthonormal basis of a separable Hilbert space $H.$ Denote by $(\cdot | \cdot )$ the inner product in $H.$
Consider $\Bbb K=\Bbb K(H)$ as a Hilbert $\Bbb K$-module in the standard way.

For $i,j\in \Bbb N$ let $e_{i,j}\in\Bbb B(H)$ be the $1$-dimensional partial isometry defined by $e_{i,j}(\xi)=(\xi|\epsilon_j)\epsilon_i,\,\xi \in H.$ In particular, for each $n\in\Bbb N,$  $e_{n,n}$ is the orthogonal projection to $\text{span}\{\epsilon_n\}.$ One easily verifies that $e_{i,j}e_{k,l}=\delta_{j,k}e_{i,l}$ and $e_{i,j}^*=e_{j,i}$ for all $i,j,k,l\in \Bbb N.$

Since $(\sum_{n=1}^N\theta_{e_{n,1},e_{n,1}})_N=(\sum_{n=1}^N e_{n,1}e_{n,1}^*)_N=(\sum_{n=1}^Ne_{n,n})_N$ is an approximate unit for $\Bbb K,$  Proposition~\ref{parseval frames are approximate units} implies that $(e_{n,1})_n$ is a Parseval frame for $\Bbb K.$

Let $(H_n)_n$ be a sequence of closed infinite dimensional subspaces of $H$ such that $\bigoplus_{n=1}^{\infty}H_n=H.$ For each $n\in \Bbb N$ consider a partial isometry $t_n\in \Bbb B(H)$ such that $\N(t_n)=
\textup{span}\{\epsilon_1\}$ and $\R (t_n)=H_n.$ Thus, $t_nt_n^*$ is the orthogonal projection to $H_n$ for all $n \in \Bbb N.$ By construction, $t_ne_{1,n}=0$ for all $n \in \Bbb N.$
As in Example~\ref{fail 2} one verifies that the series $\sum_{n=1}^{\infty}t_nt_n^*a$ converges in norm to $a$, for each $a$ in $\Bbb K.$

Let $y_n=e_{n,1}+t_n,\,n\in \Bbb N.$ Then for all $a \in \Bbb K$ we have
$$
\sum_{n=1}^{\infty}y_n\la y_n,a\ra=\sum_{n=1}^{\infty}(e_{n,1}+t_n)(e_{1,n}+t_n^*)a=\sum_{n=1}^{\infty}e_{n,n}a + \sum_{n=1}^{\infty}t_nt_n^*a=2a,
$$
and since $y_n\not \in \Bbb K$ for every $n\in \Bbb N,$ we conclude that $(y_n)_n$ is an outer 2-tight frame for $\Bbb K.$

Finally, $(y_n)_n$ and $(e_{n,1})_n$ are dual to each other, since for all $a\in \Bbb K$
$$
\sum_{n=1}^{\infty}y_n\la e_{n,1},a\ra=\sum_{n=1}^{\infty}(e_{n,1}+t_n) e_{1,n}a=\sum_{n=1}^{\infty}e_{n,n}a=a.
$$
\end{example}

\vspace{.1in}

We now state our first result which describes all frames and outer frames that are dual to a given one.

\begin{theorem}\label{general dual}
Let $(x_n)_n$ be a frame or an outer frame for a Hilbert $\A$-module $X$ with the analysis operator $U.$
An operator $V\in \Bbb B(X,\ell^2(\A))$ is the analysis operator of a frame or an outer frame dual to $(x_n)_n$ if and only if $V$ is of the form
\begin{equation}\label{dual analysis}
V=U(U^*U)^{-1}+ \left( I-U(U^*U)^{-1}U^*\right)L
\end{equation}
for some $L\in \Bbb B(X,\ell^2(\textsf{A})).$
\end{theorem}
\begin{proof}
If $\yn$ is a frame or an outer frame dual to $\xn,$ i.e., if $U^*V=I,$ then (\ref{dual analysis}) is fulfilled  if we choose $L=V.$

Conversely, if $V$ is as in (\ref{dual analysis}), then a straightforward verification shows that $U^*V=I$ and Lemma~\ref{dual relaxed} applies.
\end{proof}

\begin{remark}\label{pseudoinverz operatora analize}
Suppose we are given a frame or an outer frame $(x_n)_n$ for a Hilbert $C^*$-module $X.$ Denote the corresponding analysis operator by $U.$

Let $P= I-U(U^*U)^{-1}U^*\in\Bbb B(\ell^2(\A)).$ It is easy to verify that $P=P^*=P^2,$ and $P((a_n)_n)=(a_n)_n$ if and only if $(a_n)_n\in \N(U^*).$ Since $U$ has a closed range, $\R(U)$ is complementable in $\ell^2(\A)$ and $\N(U^*)=\R(U)^\perp.$ Therefore,
\begin{equation}\label{second term duals}
I-U(U^*U)^{-1}U^*=P_{\R(U)^\perp}.
\end{equation}

Observe that each $V$ as in  (\ref{dual analysis}) consists of two terms. The first one is just the analysis operator of the canonical dual of $(x_n)_n.$ The second term comes from an arbitrary adjointable operator $L: X \rightarrow \ell^2(\textsf{A})$ compressed to the submodule $\textup{R}(U)^{\perp}=\textup{N}(U^*)$ which is a part of $\ell^2(\textsf{A})$ of no relevance as far as the right inverse of $U^*$ is concerned.
\end{remark}

\vspace{.1in}

\begin{cor} \label{kao 6.5 kod FL}
Let $(x_n)_n$ be a frame or an outer frame for a Hilbert $\A$-module $X$ with the analysis operator $U.$
If $\yn$ is a frame or an outer frame dual to $(x_n)_n$ then
\begin{equation}\label{minimalni dualni}
\sum_{n=1}^\infty\la x,(U^*U)^{-1}x_n\ra\la (U^*U)^{-1}x_n,x\ra\leq\sum_{n=1}^\infty\la x,y_n\ra\la y_n,x\ra,\quad \forall x\in X.
\end{equation}
\end{cor}
\begin{proof} Let $V$ be the analysis operator for $\yn.$
For every $x\in X$ the left hand side of \eqref{minimalni dualni} is equal to
\begin{eqnarray*}
  \sum_{n=1}^\infty\la x,(U^*U)^{-1}x_n\ra\la (U^*U)^{-1}x_n,x\ra&=&\sum_{n=1}^\infty\la (U^*U)^{-1}x,x_n\ra\la x_n,(U^*U)^{-1}x\ra\\
&=&\la U(U^*U)^{-1}x,U(U^*U)^{-1}x\ra\\
&=&\la (U^*U)^{-1}x,x\ra,
\end{eqnarray*}
while the right hand side is $\la Vx,Vx\ra=\la V^*Vx,x\ra.$  Therefore, \eqref{minimalni dualni} reads as
\begin{equation}\label{minimalni dualni short}
(U^*U)^{-1}\leq V^*V.
\end{equation}
By Theorem~\ref{general dual} and Remark~\ref{pseudoinverz operatora analize}, there is $L\in \Bbb B(X,\ell^2(\textsf{A}))$ such that $V=U(U^*U)^{-1}+ P_{\R(U)^\perp}L.$ Then a straightforward calculation shows that  $V^*V=(U^*U)^{-1}+L^*P_{\R(U)^\perp}L,$ which obviously implies $V^*V\ge (U^*U)^{-1}$.
 \end{proof}

\vspace{.1in}

We note that the above corollary sharpens Proposition 6.5 from \cite{FL2}. It shows that the frame coefficients of the canonical dual retain the minimality property even when outer frames are included into consideration.

Theorem~\ref{general dual} enables us also to describe those frames and outer frames  that possess Parseval duals.

\begin{cor}\label{kad postoji Pdual}
Let $(x_n)_n$ be a frame or an outer frame for a Hilbert $\A$-module $X$ with the analysis operator $U.$
Then $(x_n)_n$ admits a Parseval dual or an outer Parseval dual $(y_n)_n$ if and only if there is $T\in \Bbb B(X,\ell^2(\textsf{A}))$ such that
$U^*U-I=T^*P_{\textup{R}(U)^{\perp}}T.$
\end{cor}
\begin{proof}
Suppose $(x_n)_n$ admits a Parseval dual frame or outer frame $(y_n)_n.$ If $V$ is the analysis operator of $(y_n)_n$ then, by Theorem~\ref{general dual}, there exists $L\in \Bbb B(X,\ell^2(\A))$ such that $V=U(U^*U)^{-1}+ P_{\textup{R}(U)^{\perp}}L.$
Then $I=V^*V=(U^*U)^{-1}+L^*P_{\textup{R}(U)^{\perp}}L,$
so, denoting $T= L(U^*U)^{\frac{1}{2}},$ we get
\begin{eqnarray*}
U^*U-I&=&(U^*U)^\frac{1}{2}(I-(U^*U)^{-1})(U^*U)^\frac{1}{2}\\
&=&(U^*U)^\frac{1}{2}L^*P_{\textup{R}(U)^{\perp}}L(U^*U)^\frac{1}{2}\\
&=&T^*P_{\textup{R}(U)^{\perp}}T.
\end{eqnarray*}

Conversely, suppose $U^*U-I=T^*P_{\textup{R}(U)^{\perp}}T$ for some  $T\in \Bbb B(X,\ell^2(\textsf{A})).$
Let $V=U(U^*U)^{-1}+ P_{\textup{R}(U)^{\perp}}T(U^*U)^{-\frac{1}{2}}.$ By Theorem~\ref{general dual}, $V$ is the analysis operator of some frame or outer frame $(y_n)_n$ for $X$ which is dual to $(x_n)_n.$ Since
\begin{eqnarray*}
  V^*V&=& (U^*U)^{-1}+(U^*U)^{-\frac{1}{2}}T^*P_{\textup{R}(U)^{\perp}}T(U^*U)^{-\frac{1}{2}}\\
&=&(U^*U)^{-\frac{1}{2}}(I+T^*P_{\textup{R}(U)^{\perp}}T)(U^*U)^{-\frac{1}{2}}\\
&=&(U^*U)^{-\frac{1}{2}}U^*U(U^*U)^{-\frac{1}{2}}=I,
\end{eqnarray*}
$(y_n)_n$ is a Parseval frame or an outer Parseval frame for $X.$
\end{proof}
\vspace{.1in}


\begin{remark}
Hilbert space frames that possess Parseval duals are described in \cite{H} and \cite{ACRS}, see also \cite{BB1}. It turns out that a frame $(x_n)_n$ for a Hilbert space possesses a Parseval dual if and only if $A\geq 1$ and $\text{dim}\,\textup{R}(U^*U-I)\leq \text{dim}\, \textup{R}(U)^{\perp}.$ (Here, as usual, $A$ denotes a lower frame bound and $U$ is the analysis operator.) Note that the later condition means that $\textup{R}(U^*U-I)$ can be isometrically embedded into $\textup{R}(U)^{\perp}.$ We note that these conditions are implied by Corollary~\ref{kad postoji Pdual}.

Indeed, if $U^*U-I=T^*P_{\textup{R}(U)^{\perp}}T$ then, obviously, $U^*U-I\geq 0$ which implies $A\geq 1.$ On the other hand, the equality
$U^*U-I=T^*P_{\textup{R}(U)^{\perp}}T$ can be rewritten as
$(U^*U-I)^{\frac{1}{2}}(U^*U-I)^{\frac{1}{2}}=T^*P_{\textup{R}(U)^{\perp}}T$ which gives us
$$
\left\langle (U^*U-I)^{\frac{1}{2}}x,(U^*U-I)^{\frac{1}{2}}x\right\rangle = \left\langle P_{\textup{R}(U)^{\perp}}Tx,P_{\textup{R}(U)^{\perp}}Tx\right\rangle,\quad\forall x\in X.
$$
Using this equality, we can define a map $\phi : \textup{R}((U^*U-I)^{\frac{1}{2}}) \rightarrow \textup{R}(U)^{\perp}$ by $\phi((U^*U-I)^{\frac{1}{2}}x)=
P_{\textup{R}(U)^{\perp}}Tx.$ Clearly, $\phi$ is a well-defined isometry. Since, obviously, $\textup{R}(U^*U-I)\subseteq \textup{R}((U^*U-I)^{\frac{1}{2}}),$ we conclude that the above map $\phi$ provides an isometrical embedding of $\textup{R}(U^*U-I)$ into $\textup{R}(U)^{\perp}.$
\end{remark}


\vspace{.1in}

Next we provide another description of all frames and outer frames that are dual to a given one. We shall use \cite{BB1} as a blueprint.

\begin{prop}\label{left inverse 1}
Let $X$ and $Y$ be Hilbert $\textsf{A}$-modules. Let $U\in \Bbb B(X,Y)$ and $T\in \Bbb B(Y,X)$ be such that $TU=I.$ Then
\begin{enumerate}
\item[(a)] $\N(T)=\textup{R}(I-UT)=(I-UT)(\N(U^*))$
\item[(b)] $Y=\textup{R}(U) \dotplus \N(T)$ (a direct sum),
\item[(c)] $UT \in \Bbb B(Y)$ is the oblique projection to $\textup{R}(U)$ along $\N(T).$
\end{enumerate}
\end{prop}
\begin{proof}
(a) From $T(I-UT)=0$ we have $\textup{R}(I-UT)\subseteq \N(T).$ Conversely, $y\in \N(T) \Rightarrow (I-UT)y=y  \Rightarrow y \in \textup{R}(I-UT).$ This gives the first equality.

To prove $\textup{R}(I-UT)\subseteq (I-UT)(\N(U^*))$ (the opposite inclusion is obvious), first observe that our assumption $TU=I$ implies that $U$ is bounded from below. Hence $\textup{R}(U)$ is a closed submodule of $Y$ and $Y=\textup{R}(U)\oplus \N(U^*).$
Let us now take arbitrary $(I-UT)y \in \textup{R}(I-UT),$ $y\in Y.$ Then $y=Ux+z$ for some $x\in X$ and $z\in \N(U^*),$ so we have $$(I-UT)y=(I-UT)Ux+(I-UT)z=(I-UT)z\in (I-UT)(\N(U^*)).$$

(b) Let $y\in \textup{R}(U) \cap \N(T).$ Then $y=Ux$ for some $x\in X$ and $Ty=0.$ Putting this together we get $TUx=0$; thus, by assumption, $x=0.$ Hence, $y=0$ and this shows that the intersection $\textup{R}(U) \cap \N(T)$ is trivial.

Let us now take arbitrary $y\in Y$ and write it, as in the preceding paragraph, in the form $y=Ux+z$ with $x\in X$ and $z\in \N(U^*).$ Then we have (again as before) $(I-UT)y=(I-UT)z.$
This can be rewritten as
\begin{equation}\label{direct sum}
y=UTy+(I-UT)z.
\end{equation}
Since $UTy\in\textup{R}(U)$ and $(I-UT)z\in (I-UT)(\N(U^*))\stackrel{\textup{(a)}}{=}\N(T),$ the proof is completed.

(c) Evidently, $UT\in \Bbb B(Y)$ satisfies $UTUx=Ux$ for all $Ux\in \textup{R}(U),$ and $UTy=0$ for all $y\in \N(T).$
\end{proof}

\begin{prop}\label{dual frames again}
Let $(x_n)_n$ and $(y_n)_n$ be frames or outer frames for a Hilbert $\textsf{A}$-module $X$ that are dual to each other. Denote by $U$ and $V$ the corresponding analysis operators. Then
\begin{enumerate}
\item[(a)] $\ell^2(\textsf{A})=\textup{R}(U) \dotplus \N(V^*),$
\item[(b)] $UV^*$ is the oblique projection to $\textup{R}(U)$ along $\N(V^*),$
\item[(c)] $\ell^2(\textsf{A})=\textup{R}(V) \dotplus \N(U^*),$
\item[(d)] $VU^*$ is the oblique projection to $\textup{R}(V)$ along $\N(U^*).$
\end{enumerate}
\end{prop}
\begin{proof}
(a) and (b) follow from the preceding proposition and the equality $V^*U=I,$ while (c) and (d) are obtained in the same way using the equality $U^*V=I.$
\end{proof}

\vspace{.1in}

\begin{remark}\label{the only orthogonal projection}
Consider a frame or an outer frame $(x_n)_n$ for a Hilbert $\textsf{A}$-module $X$ with the analysis operator $U.$ Let $(y_n)_n$ be a frame or an outer frame dual to $(x_n)_n$; let $V$ denotes the corresponding analysis operator. Then, by the preceding proposition, $UV^*$ is an obligue projection to $\textup{R}(U).$
In the special case when $(y_n)_n$ is the canonical dual of $(x_n)_n$ we have $V=U(U^*U)^{-1}$ and $UV^*=U(U^*U)^{-1}U^*$ which is by Remark~\ref{pseudoinverz operatora analize} the {\em orthogonal} projection to $\textup{R}(U).$
So, in the light of the preceding proposition, this orthogonality is the exclusive property of the canonical dual among all frames and outer frames that are dual to $(x_n)_n.$
\end{remark}

\vspace{.1in}

The following theorem is a result similar to Theorem~\ref{general dual}. It provides another characterization of analysis operators of dual frames and outer frames.

\begin{theorem}\label{general dual 1}
Let $(x_n)_n$ be a frame or an outer frame for Hilbert $\textsf{A}$-module $X$ with the analysis operator $U.$
An operator $V\in \Bbb B(X,\ell^2(\A))$ is the analysis operator of a frame or an outer frame dual to $(x_n)_n$ if and only if $V$ is of the form
\begin{equation}\label{dual analysis 1}
V=F^*U(U^*U)^{-1},
\end{equation}
where $F\in \Bbb B(\ell^2(\textsf{A}))$ is an oblique projection to $\textup{R}(U)$ along some closed direct complement of $\textup{R}(U)$ in $\ell^2(\textsf{A}).$
\end{theorem}
\begin{proof}
If $\yn$ is a frame or an outer frame dual to $\xn$ then its analysis operator $V \in \Bbb B(X,\ell^2(\textsf{A}))$ satisfies $V^*U=I$ so, by Proposition~\ref{dual frames again},
$\ell^2(\textsf{A})=\textup{R}(U) \dotplus \N(V^*)$
and $UV^*$ is the oblique projection to $\textup{R}(U)$ along $\N(V^*).$ Let $F=UV^*.$ Then $F^*U(U^*U)^{-1}=VU^*U(U^*U)^{-1}=V.$

To prove the converse, take $V$ as in (\ref{dual analysis 1}) and observe that $FU=U.$ Then we have $V^*U=(U^*U)^{-1}U^*FU=I.$
\end{proof}

We conclude this section with a discussion about frames and outer frames that have a unique dual.

\begin{theorem}\label{unique dual}
Let $(x_n)_n$ be a frame or an outer frame for a Hilbert $\A$-module $X$ with the analysis operator $U.$
Consider the following conditions:
\begin{enumerate}
\item[(a)] $\N(U^*)=\{0\}.$
\item[(b)] $\R(U)=\ell^2(\textsf{A}).$
\item[(c)] The canonical dual is the only dual (including both frames and outer frames) of $(x_n)_n.$
\end{enumerate}
Then (a) and (b) are mutually equivalent and imply (c). If $X$ is full, (c) is equivalent to (a) and (b).\\
\end{theorem}
\begin{proof}
Since $\ell^2(\A)=\R(U) \oplus \N(U^*),$ (a) and (b) are equivalent. Also, (b) together with Theorem~\ref{general dual 1} immediately implies (c).

To prove the last statement, suppose that $X$ is full and that (c) is satisfied. Recall from Theorem~\ref{general dual} that each adjointable operator $L:X \to \ell^2(\A)$ gives rise to a frame or an outer frame for $X$ that is dual to $(x_n)_n$ and whose analysis operator is given by
$$
V=U(U^*U)^{-1}+ \left( I-U(U^*U)^{-1}U^*\right)L.
$$
By (c), we now have $\left( I-U(U^*U)^{-1}U^*\right)L=0$ for all $L\in \Bbb B(X,\ell^2(\A)),$ or equivalently, $L=U(U^*U)^{-1}U^*L$ for all $L\in \Bbb B(X,\ell^2(\A)).$ Recall from Remark~\ref{pseudoinverz operatora analize} that $U(U^*U)^{-1}U^*$ is the orthogonal projection to $\R(U).$ Hence, the above conclusion means that each operator $L\in \Bbb B(X,\ell^2(\A))$ takes values in $\R(U).$

Let us now take arbitrary $x\in X$ and $j\in \Bbb N.$ Define $L_{x,j}:X \to \ell^2(\A)$ by $L_{x,j}(y)=(0,\ldots,0,\la x,y\ra,0,\ldots)$ (with $\la x,y\ra$ on $j$-th position). Obviously, $L_{x,j}$ is an adjointable operator whose adjoint is given by $L_{x,j}^*((a_n)_n)=xa_j.$ By the preceding conclusion, all $L_{x,j}$ take values in $\R(U).$ Since $X$ is by our assumption full, this immediately implies that $c_{00}(\A)\subseteq \R(U).$ Since $\R(U)$ is closed, this gives us $\ell^2(\textsf{A}) \subseteq \R(U)$ and hence
$\R(U)=\ell^2(\textsf{A}).$
\end{proof}

\vspace{.1in}

Here we need to make a comment on Theorem 3.10 from \cite{HJLM}. Namely, that theorem states that all three above conditions are equivalent
without assuming that the ambient Hilbert module $X$ is full over $\A.$ However, there is a gap in the proof of Theorem 3.10 from \cite{HJLM} and this is the reason why we decided to include the preceding theorem in the present paper.

To show that the fullness assumption is really necessary in the proof of the implication (c) $\Rightarrow$ (b) from Theorem~\ref{unique dual}, we provide the following example.

\vspace{.1in}

Let $\textsf{B}$ be a unital $C^*$-algebra that is contained as a non-essential ideal in a unital $C^*$-algebra $\A.$ This means that $\textsf{B}^{\perp}=\{a\in \A: a\textsf{B}=\{0\}\} \not =\{0\}.$ Consider $X=\ell^2(\textsf{B})$ as a Hilbert $C^*$-module over $\A.$ Clearly, $X$ is not full as a Hilbert $\A$-module.
Denote by $e$ the unit element of $\textsf{B}.$ Obviously, the sequence $(e^{(n)})_n$ is a Parseval frame for $X.$ One easily concludes that the corresponding analysis operator $U:X \rightarrow \ell^2(\A)$ acts as the inclusion; hence $\R(U)=\ell^2(\textsf{B}).$ This means that $U$ is not a surjection and that $\R(U)^{\perp}=\N(U^*)$ is a non-trivial submodule of $\ell^2(\A).$

However, $(e^{(n)})_n$ has a unique dual frame (in fact, $(e^{(n)})_n$ is, being Parseval, self-dual). To prove this, recall that the analysis operator of each frame  dual to $(e^{(n)})_n$ (here there are no outer frame since $\A$ is unital) is given by
$$
V=U(U^*U)^{-1}+ \left( I-U(U^*U)^{-1}U^*\right)L,
$$
where $L:X \rightarrow \ell^2(\A)$ is an adjointable operator. We now observe that the Hewitt-Cohen factorization (Proposition 2.31 from \cite{RW}) forces each $L$ to take values in $\R(U)=\ell^2(\textsf{B}).$ Since, by Remark~\ref{pseudoinverz operatora analize}, $I-U(U^*U)^{-1}U^*$ is the orthogonal projection to $\R(U)^{\perp},$ we have
$\left( I-U(U^*U)^{-1}U^*\right)L=0$ for each $L\in \Bbb B(X,\ell^2(\A)).$
This together with the fact $U^*U=I$ shows that the above equality reduces, for all $L$, to $V=U$. Hence, there is only one dual frame, namely $(e^{(n)})_n$ itself.

\vspace{.1in}

Let us now state several consequences of Theorem~\ref{unique dual}.

\begin{cor}\label{module with unique dual}
A full Hilbert $\A$-module $X$ which possesses a frame or an outer frame $\xn$ with a unique dual is unitarilly equivalent to $\ell^2(\A).$
\end{cor}
\begin{proof}
If $\xn$ is a frame or an outer frame for $X$ which has a unique dual, then its analysis operator $U\in\B(X,\ell^2(\A))$ is invertible by Theorem~\ref{unique dual}, so the operator $U(U^*U)^{-\frac{1}{2}}\in\B(X,\ell^2(\A))$ is unitary.
\end{proof}

\vspace{.1in}

\begin{cor}\label{unique dual-never in non-unital case}
Let  $X$ be a full Hilbert $C^*$-module over a non-unital $C^*$-algebra $\A.$ Then every frame for $X$ has at least two duals.
\end{cor}
\begin{proof}
Suppose there is a frame $\xn$ for $X$ with the unique dual. Let $U\in \Bbb B(X,\ell^2(\A))$ be the corresponding analysis operator. By Theorem~\ref{unique dual}, $U$ is a bijection.

Regarding $X$ as a Hilbert $\tilde{\A}$-module, it is easy to verify that $U$ can be regarded as an adjointable operator $\tilde{U}\in \Bbb B(X,\ell^2(\tilde{\textsf{A}}))$ given by $\tilde{U}x=Ux,$ $x\in X.$
Then $\text{R}(\tilde{U})= \text{R}(U)= \ell^2(\textsf{A}).$
Since $\tilde{U}$ is bounded from below, its range  $\text{R}(\tilde{U})$ is closed in $\ell^2(\tilde{\textsf{A}}).$ So, being the range of an adjointable operator, a closed submodule $\ell^2(\textsf{A})$ of $\ell^2(\tilde{\textsf{A}})$ must be complementable in $\ell^2(\tilde{\textsf{A}}).$ But this is a contradiction since $\ell^2(\textsf{A})^{\perp}=\{0\}.$
(Namely,  if $(b_n)_n\in \ell^2(\tilde{\textsf{A}})$ belongs to $\ell^2(\A)^\perp,$ then for each $m$ it holds
$b_ma=\la (b_n)_n,(a^{(m)})_n\ra =0$ for all $a\in \A.$ Since $\A$ is an essential ideal of $\tilde{\A},$ it follows that $b_m=0$ for all $m.$)
\end{proof}

\vspace{.1in}

\begin{remark}
Corollary~\ref{unique dual-never in non-unital case} does not hold for outer frames. Indeed, if $\A$ is a non-unital $C^*$-algebra and $X=\ell^2(\A),$ then $(e^{(n)})_n$ is an outer frame for $X$ whose analysis operator $U$ is the identity, so by Theorem~\ref{unique dual}, $(e^{(n)})$ has a unique dual.
\end{remark}

\vspace{.1in}

\begin{remark}\label{el-2 je jedini}
By Corollary~\ref{module with unique dual} generalized Hilbert spaces $\ell^2(\A)$ with $\A$ $\sigma$-unital are, up to unitary equivalence, only countably generated Hilbert $C^*$-modules that possess frames or outer frames with unique duals. If $\A$ is unital, $(e^{(n)})_n$ is a Parseval frame for $\ell^2(\A)$ with this property. If $\A$ is non-unital, Corollary~\ref{unique dual-never in non-unital case} tells us that such frames in $\ell^2(\A)$ do not exist, so we only have outer frames with unique duals. As the example from the preceding remark shows, $(e^{(n)})_n$ is such an outer frame.
\end{remark}

\vspace{.1in}

For our last result of this section recall that each frame or outer frame $(x_n)_n$ for a Hilbert $C^*$-module $X$ possesses canonically associated Parseval frame or outer Parseval frame $(y_n)_n$. If $U$ denotes the analysis operator of $(x_n)_n$, $y_n$'s are given by $y_n=(U_M^*U_M)^{-\frac{1}{2}}x_n$, $n\in \Bbb N$. Here we must work with the extended operator $U_M$ if $(x_n)_n$ is outer. If, on the other hand, $x_n\in X$, for all $n\in \Bbb N$, then the preceding equality reduces to $y_n=(U^*U)^{-\frac{1}{2}}x_n\in X$, $n\in \Bbb N$. Observe that in both cases the analysis operator of $(y_n)_n$ is given by $U(U^*U)^{-\frac{1}{2}}$.

In the following corollary we consider a Hilbert $C^*$-module over a unital $C^*$-algebra (because of Corollary~\ref{unique dual-never in non-unital case}), so there are no outer frames.

\begin{cor}\label{unique dual-unital case}
Let  $X$ be a full Hilbert $C^*$-module over a unital $C^*$-algebra $\A.$ Suppose there exists a frame $\xn$ with a unique dual. Let $U$ be the analysis operator for $\xn.$  Then the following statements hold:
 \begin{enumerate}
   \item[(a)] The Parseval frame $(y_n)_n$ canonically associated with $(x_n)_n$ has a unique dual and
   $$\la y_n,y_m\ra=\delta_{nm}e,\quad \forall m,n\in \Bbb N.$$
   \item[(b)] $\la x_n,x_n\ra$ is invertible for every $n.$
   \item[(c)] If $\sum_{n=1}^\infty x_na_n=0$ for some $a_n\in\A,n\in\Bbb N,$ then $a_n=0$ for all $n\in \Bbb N.$
\end{enumerate}
\end{cor}
\begin{proof}
By Theorem~\ref{unique dual}, $U$ is a bijection. Since $\A$ is unital, $x_n=U^*(e^{(n)})$ for every $n\in \Bbb N.$
The analysis operator $V=U(U^*U)^{-\frac{1}{2}}$ for $\yn$ is an isometry and a bijection, hence unitary, so
$$\la y_n,y_m\ra=\la V^*(e^{(n)}),V^*(e^{(m)})\ra=\la (e^{(n)}),(e^{(m)})\ra=\delta_{nm}e,\quad \forall m,n\in\Bbb N.$$

Further,
$$e=\la y_n,y_n\ra= \la (U^*U)^{-1}x_n,x_n\ra\le \|(U^*U)^{-1}\|\la x_n,x_n\ra,\quad \forall n\in \Bbb N,$$
so $\la x_n,x_n\ra$ is invertible for all $n\in \Bbb N.$

Finally, if $\sum_{n=1}^\infty x_na_n=0$ for some sequence $(a_n)_n$ in $\A,$ then we also have  $(U^*U)^{-\frac{1}{2}}(\sum_{n=1}^\infty x_na_n)=0,$ i.e.,
$\sum_{n=1}^\infty y_na_n=0.$ Then for all $m\in \Bbb N$ we have
$$0=\la y_m,\sum_{n=1}^\infty y_na_n\ra=\sum_{n=1}^\infty\la y_m, y_na_n\ra=\sum_{n=1}^\infty\delta_{mn}a_n= a_m,$$
and (c) is proved.
\end{proof}

\vspace{.1in}

We conclude this section with a remark concerning finite frames and outer frames and their duals.

\begin{remark}\label{finite duals}
Let $(x_n)_{n=1}^N$ be a frame or an outer frame for a Hilbert $C^*$-module $X.$ A frame or an outer frame  $(y_n)_{n=1}^N$ is said to be dual to $(x_n)_{n=1}^N$ if $\sum_{n=1}^N y_n\la x_n,x\ra=x$ for all $x\in X.$

Observe that the analysis operator $U$ of $(x_n)_{n=1}^N$ takes values in $\textsf{A}^N.$ It is easy to see that, with this difference, i.e., with $\textsf{A}^N$ playing the role of $\ell^2(\textsf{A}),$ all the preceding results from this section survive. In particular, one can show that, for $N\in \Bbb N$, Hilbert $C^*$-modules $\A^N$ have properties analogous to those of $\ell^2(\A)$ discussed in Remark~\ref{el-2 je jedini}. We omit the details.
\end{remark}

\vspace{0.2in}

\section{Perturbations and tight approximations of frames}

\vspace{0.2in}

In this section we study neighborhoods of frames and outer frames. In fact, our results will be stated in terms of neighborhoods of the corresponding analysis operators.

There are several important results concerning perturbations of frames for Hilbert spaces (see \cite{CC} and references therein). Perturbations of frames for Hilbert $C^*$-modules are considered in \cite{HJLM1}.
A remarkable property of any frame $(x_n)_n$ for a Hilbert space is that one can always find a neighborhood of $(x_n)_n$ (defined in terms of $\ell^2$-distance of sequences or in terms of the distance of analysis/synthesis operators) such that each sequence belonging to that neighborhood (i.e.,  sufficiently close to $(x_n)_n$) is also a frame. As usual, the situation is more complicated in the modular context.

We begin with an example which shows that in any Hilbert $C^*$-module $X$ such that $M(X)\neq X$ we can find a frame for $X$ such that any neighborhood of its analysis operator contains an operator that is not the analysis operator of any frame for $X.$

\begin{example}\label{okolina od U primjer}
Let $X$ be a Hilbert $\A$-module such that $M(X)\neq X$ and $v\in M(X)\setminus X$ such that $\|v\|=1.$
Take arbitrary $\varepsilon>0.$

Let $\xn$ be a frame for $X.$ Then the sequence $0,x_1,x_2,x_3,\ldots$ is also a frame for $X,$ and its analysis operator $U\in\Bbb B(X,\ell^2(\A))$ is given by $Ux=(0, \la x_1,x\ra, \la x_2,x\ra,\ldots).$ Further, the sequence $\varepsilon v,x_1,x_2,x_3,\ldots$ is an outer frame for $X$ and its analysis operator $V\in\Bbb B(X,\ell^2(\A))$ is given by  $Vx=(\varepsilon\la v,x\ra, \la x_1,x\ra, \la x_2,x\ra,\ldots).$
Observe that the operator $V,$ being the analysis operator of an outer frame for $X,$ is not the analysis operator of any frame for $X.$
On the other hand,
$$\| U-V\|=\sup\{\varepsilon \|\la v,x\ra\|:x\in X,\|x\|\le 1\}
 \le\varepsilon \|v\|=\varepsilon.$$
\end{example}

\vspace{0.1in}

The above example suggests that, as in the preceding section, in order to obtain analogues of the classical results, one should include  outer frames into the consideration.

\vspace{0.1in}

We restrict our discussion to {\em infinite sequences}. Thereby, we shall understand that finite frames $(x_n)_{n=1}^N$ are extended to infinite sequences by adding infinitely many zero vectors.

\vspace{0.1in}

\begin{theorem}\label{okolina1}
Let $(x_n)_n$ be a frame or an outer frame for a Hilbert $\textsf{A}$-module $X$ with the analysis operator $U$ and the optimal lower frame bound $A.$ Suppose that $V\in \Bbb B(X,\ell^2(\textsf{A}))$ satisfies $\|U-V\|<\sqrt{A}.$ Then $V$ is the analysis operator of a frame or an outer frame $(y_n)_n$ for $X$ such that
$$\|x_n-y_n\|\leq \|U-V\|<\sqrt{A},\quad\forall n\in \Bbb N.$$
\end{theorem}
\begin{proof}
Let $\|U-V\|=m<\sqrt{A}.$ Then
$$
\|Vx\|\ge \|Ux\|-\|Ux-Vx\|\ge \sqrt{A}\|x\|-m\|x\|=(\sqrt{A}-m)\|x\|
$$
 for all $x\in X.$ Thus, $V$ is bounded from below and consequently, $V^*\in \Bbb B(\ell^2(\textsf{A}),X)$ is a surjection. By Theorem~\ref{stvarna korespondencija frameova i sinteza}, $V^*$ is the synthesis operator of a frame or an outer frame $(y_n)_n$ for $X$ defined by  $y_n=(V_M)^*e^{(n)},\, n \in \Bbb N.$ Then, using Remark~\ref{multiplier modules}(d),  for each $n\in \Bbb N$ we have
\begin{eqnarray*}
\|x_n-y_n\|&=&\|(U_M-V_M)^*e^{(n)}\|=\|(U^*-V^*)_Me^{(n)}\|\\
&\le & \|(U^*-V^*)_M\|=\|U^*-V^*\|\\
&=&\|U-V\|<\sqrt{A}.
\end{eqnarray*}
Observe that the extended operators $U_M$ and $V_M$ coincide with $U$ and $V,$ respectively, when $\A$ is unital. On the other hand, if $\A$ is non-unital and $x_n$ or $y_n$ belongs to $M(X)\setminus X$ for some $n,$ then the expression $\|x_n-y_n\|$ is computed in the multiplier module $M(X).$
\end{proof}

Let us first note an easy consequence of this result. A similar result appeared in Theorem~3.16. of \cite{Jing}.

\begin{cor}\label{okolina1-corollary}
Let $(x_n)_n$ be a frame or an outer frame for a Hilbert $\textsf{A}$-module $X$ with the analysis operator $U$ and the optimal lower frame bound $A.$ If $\|x_j\|<\sqrt{A}$ for some $j,$ then $(x_n)_{n\neq j}$ is a frame or an outer frame for $X.$
\end{cor}
\begin{proof} Let us define a sequence $\yn$ as $y_j=0$ and $y_n=x_n$ for $n\neq j.$ Since $\xn$ is a frame or an outer frame for $X,$ $\yn$ is a Bessel sequence or an outer Bessel sequence. Let $V\in \Bbb B(X,\ell^2(\A))$ be the analysis operator associated to $\yn.$ Since $\|(U-V)x\|=\|\la x_j,x\ra\|$ for all $x\in X,$ we have $\|U-V\|=\|x_j\|<\sqrt{A},$ so by Theorem~\ref{okolina1}, $\yn$ is a frame or an outer frame for $X.$ Then obviously, $(x_n)_{n\neq j}$ is also a frame or an outer frame for $X.$
\end{proof}

\begin{remark}\label{konus}
The open ball from Theorem~\ref{okolina1} is the largest open ball around $U$ with that property. Indeed, let us consider  an orthonormal basis $(\epsilon_n)_n$ for a Hilbert space $H$ as a frame for $H;$ the analysis operator $U$ is then an isometry and the optimal lower bound is $A=1.$ If we denote by $V$ the analysis operator of the Bessel sequence $\{0\}\cup(\epsilon_n)_{n\ge 2}$ (which is not a frame for $H$), then $\|V-U\|=1=\sqrt{A},$ so the boundary of the  open ball around $U$ with the radius $\sqrt{A}$ contains an operator which is not the analysis operator of any frame (or outer frame) for $H.$
\end{remark}

\vspace{0.1in}

At this point we need to make a comment on Theorem~3.2 from \cite{HJLM1}. The second statement of that theorem may be rephrased as follows:

{\em Let $(x_n)_n$ be a frame for a Hilbert $C^*$-module $X$ over a unital $C^*$-algebra $\textsf{A}$ with the analysis operator $U$ and the frame bounds $A$ and $B.$ Suppose that $(y_n)_n$ is a sequence in $X$ for which there exist constants $\lambda_1,\lambda_2,\mu\geq 0$ with the properties
\begin{equation}\label{han et al prva}
\max\{\lambda_1+\frac{\mu}{\sqrt{A}},\lambda_2\}<1,
\end{equation}
and
\begin{equation}\label{han et al druga}
\|\sum_{n=1}^N(x_n-y_n)a_n\|\leq \lambda_1\|\sum_{n=1}^Nx_na_n\|+\lambda_2\|\sum_{n=1}^Ny_na_n\|+\mu\|\sum_{n=1}^Na_n^*a_n\|^{\frac{1}{2}},
\end{equation}
for all finite sequences $(a_1,\ldots,a_N,0,0,\ldots)\in c_{00}(\A).$ Then $(y_n)_n$ is also a frame for $X.$
}

Clearly, one could easily deduce our Proposition 5.2 from this statement, at least in the unital case. Indeed, suppose we are given an operator $V\in \Bbb B(X,\ell^2(\textsf{A}))$ such that $\|U-V\|=\mu<\sqrt{A}.$ Put $y_n=V^*e^{(n)},$ $n \in \Bbb N.$ Then, obviously, we have for each $(a_1,\ldots,a_N,0,0,\ldots)\in c_{00}(\textsf{A}),$
$$
\|\sum_{n=1}^N(x_n-y_n)a_n\|=\|(U^*-V^*)(a_1,\ldots,a_N,0,0,\ldots)\|\leq \mu\|\sum_{n=1}^Na_n^*a_n\|^{\frac{1}{2}},
$$
which means that the sequence $(y_n)_n$ satisfies (\ref{han et al druga}) with $\lambda_1=\lambda_2=0.$ Since $\mu<\sqrt{A},$ we also have (\ref{han et al prva}); thus, by applying the above statement one could conclude that $(y_n)_n$ is a frame for $X.$

\vspace{.1in}

However, there is a gap in the proof of the above statement (i.e.,  the second part of Theorem 3.2. from \cite{HJLM1}) and it is not clear how one can fix the proof presented there.
Namely, that proof uses Lemma 2.7 and Proposition 2.8. from \cite{HJLM1} which, as we have seen in our Example~\ref{fail 1} and Remark~\ref{Heuser}, fail to be generally true. It seems that in order to obtain a result as in aforementioned Theorem 3.2 from \cite{HJLM1}, one should additionally include in the hypothesis that the sequence $(y_n)_n$ is Bessel.


\vspace{.1in}

We proceed with a remark that is known, but which we include for convenience of the reader.

\begin{remark}\label{optimalne}
Let $\xn$ be a frame or an outer frame for a Hilbert $\A$-module $X$ with optimal frame bounds $A$ and $B.$  Let us describe $A$ and $B$ in terms of the associated analysis operator $U.$

First, by Theorem~2.8 and Remark~2.9 from \cite{pas} we conclude that the optimal upper frame bound $B$ satisfies
\begin{equation}\label{B}
\sqrt{B}=\|U\|=\min\{M\ge 0: \|Ux\|\le M\|x\|,x\in X \}.
\end{equation}

Further, writing the relation  $\la Ux,Ux\ra\ge A\la x,x\ra, x\in X,$ in an equivalent form
$\la (U^*U)^{-\frac{1}{2}}x,(U^*U)^{-\frac{1}{2}}x\ra\le \frac{1}{A}\la x,x\ra, x\in X$ (obtained by replacing $x$ with  $(U^*U)^{-\frac{1}{2}}x$), and then applying \eqref{B} we get
\begin{eqnarray*}
\frac{1}{\sqrt{A}}&=&\|(U^*U)^{-\frac{1}{2}}\|\\
&=&\min\{M\ge 0: \|(U^*U)^{-\frac{1}{2}}x\|\le M\|x\|,x\in X \}\\
&&\textup{(replace $x$ with  $(U^*U)^{\frac{1}{2}}x$ and apply $\|Ux\|=\|(U^*U)^{\frac{1}{2}}x\|)$)}\\
&=&\min\{M\ge 0: \|Ux\|\ge \frac{1}{M}\|x\|,x\in X \}\\
&=&(\max\{m\ge 0: \|Ux\|\ge m\|x\|,x\in X \})^{-1}.
\end{eqnarray*}
Therefore,
\begin{equation}\label{A}
\sqrt{A}=\|(U^*U)^{-\frac{1}{2}}\|^{-1}=\max\{ m\ge 0: \|Ux\|\ge m\|x\|\}.
\end{equation}
\end{remark}
\vspace{0.1in}

The following corollary provides another useful property of the open ball with the center in $U$ that is considered in Theorem~\ref{okolina1}.

\begin{cor}\label{bliski su pseudodualni}
Let $(x_n)_n$ be a frame or an outer frame for a Hilbert $\textsf{A}$-module $X$ with the analysis operator $U$ and the optimal lower frame bound $A.$ Suppose that $V\in \Bbb B(X,\ell^2(\textsf{A}))$ satisfies $\|U-V\|<\sqrt{A}.$ Then $V^*U\in\Bbb B(X)$ is an invertible operator.
\end{cor}
\begin{proof}
Put again $\|U-V\|=m<\sqrt{A}.$ Then
\begin{eqnarray*}
\|x-V^*U(U^*U)^{-1}x\|&=&\|U^*U(U^*U)^{-1}x-V^*U(U^*U)^{-1}x\|\\
&=&\|(U^*-V^*)(U(U^*U)^{-1}x)\|\\
&\le & m \|U(U^*U)^{-1}x\|
\end{eqnarray*}
for all $x\in X.$ By taking the supremum over the unit ball in $X$ we get
$$\|I-V^*U(U^*U)^{-1}\|\le m\|U(U^*U)^{-1}\|.$$
Since $\|U(U^*U)^{-1}\|^2=\|(U(U^*U)^{-1})^*(U(U^*U)^{-1})\|=\|(U^*U)^{-1}\|=\frac{1}{A},$ we have
$$\|I-V^*U(U^*U)^{-1}\|\le \frac{m}{\sqrt{A}}<1.$$
This shows that $V^*U(U^*U)^{-1}$ is an invertible operator. In particular, $V^*U$ is invertible as well.
\end{proof}

\vspace{.1in}

\begin{remark}
We say that frames or outer frames $(x_n)_n$ and $(y_n)_n$ for a Hilbert $C^*$-module $X$ with the analysis operators $U$ and $V$ are {\em pseudodual} if $V^*U$ is an invertible operator. When this is the case, we have, for each $x\in X,$
$$
x=U^*V((U^*V)^{-1}x)=\sum_{n=1}^{\infty}x_n\la y_n, (U^*V)^{-1}x \ra= \sum_{n=1}^{\infty}x_n\la (V_M^*U_M)^{-1}y_n, x \ra.
$$
This shows that $(x_n)_n$ and $((V_M^*U_M)^{-1}y_n)_n$ are dual to each other. In an analogous way we conclude that
$(y_n)_n$ and $((U_M^*V_M)^{-1}x_n)_n$ are also dual to each other.
\end{remark}

\vspace{.1in}

Given a frame or an outer frame $(x_n)_n$ for a Hilbert $\textsf{A}$-module $X,$ we now want to find a Parseval frame for $X$ closest to $(x_n)_n,$ again measured in terms of distance of the corresponding analysis operators.
As one might expect, a solution is the Parseval frame canonically associated with $(x_n)_n,$ i.e.,  $(y_n)_n,$ where $y_n=(U_M^*U_M)^{-\frac{1}{2}}x_n,\,n\in \Bbb N,$ and $U$ is the analysis operator of $(x_n)_n.$ Recall that $(y_n)_n$ is outer if and only if $(x_n)_n$ is outer; nevertheless, its analysis operator is always equal to $U(U^*U)^{-\frac{1}{2}}.$

\begin{prop}\label{Parsevalov frame u okolini}
Let $(x_n)_n$ be a frame or an outer frame for a Hilbert $\textsf{A}$-module $X$ with the analysis operator $U$ and the optimal frame bounds $A$ and $B.$
If $(y_n)_n$ is the Parseval frame canonically associated with $(x_n)_n,$ then its analysis operator $U(U^*U)^{-\frac{1}{2}}$ satisfies
 $$\left\| U-U(U^*U)^{-\frac{1}{2}}\right\|=\max\left\{1-\sqrt{A},\sqrt{B}-1\right\}.$$
If $(y_n)_n$ is any Parseval  frame or outer Parseval frame for $X,$ then its analysis operator $V$ satisfies
$$\|U-V\|\geq \max\left\{1-\sqrt{A},\sqrt{B}-1\right\}.$$
\end{prop}
\begin{proof}
First, we have
\begin{eqnarray*}
  \left\|U-U(U^*U)^{-\frac{1}{2}}\right\|&=&\left\|\Big(U-U(U^*U)^{-\frac{1}{2}}\Big)^*\Big(U-U(U^*U)^{-\frac{1}{2}}\Big)\right\|^{\frac{1}{2}}\\
  &=& \left\|\Big((U^*U)^{-\frac{1}{2}}-I\Big)U^*U\Big((U^*U)^{-\frac{1}{2}}-I\Big)\right\|^\frac{1}{2}\\
&=&\left\|\Big(I-(U^*U)^\frac{1}{2}\Big)^2\right\|^\frac{1}{2}\\
&=&\left\|I-(U^*U)^\frac{1}{2}\right\|\\
&=&\max\left\{|1-\sqrt{A}|,|1-\sqrt{B}|\right\}\\
&=&\max\left\{1-\sqrt{A},\sqrt{B}-1\right\}.
\end{eqnarray*}

To prove the second assertion, suppose that $(y_n)_n$ is a Parseval frame or an outer Parseval frame for $X.$  Then its analysis operator $V\in \Bbb B(X,\ell^2(\textsf{A}))$ is an isometry, so  we have
$$ \|Ux\|\ge \|Vx\|-\|Vx-Ux\|=\|x\|-\|Vx-Ux\|\ge (1-\|U-V\|)\|x\|$$
for all $x\in X.$ By  \eqref{A} we get
$\sqrt{A}\ge 1-\|U-V\|,$ that is,  $\|U-V\|\geq 1-\sqrt{A}.$

On the other hand, $\sqrt{B}=\|U\|\le \|U-V\|+\|V\|=\|U-V\|+1,$ wherefrom $\|U-V\|\ge \sqrt{B}-1.$ Therefore, $\|U-V\|\ge \max\left\{1-\sqrt{A},\sqrt{B}-1\right\}.$
\end{proof}

\vspace{.1in}

In a similar fashion we can find the distance of a given frame or outer frame $(x_n)_n$ for $X$  with the optimal bounds $A$ and $B$ to the set of all tight frames and outer tight frames for $X.$ It turns out that this distance is equal to $\frac{\sqrt{B}-\sqrt{A}}{2}.$ For Hilbert space frames this question was discussed in \cite[Proposition 5.4]{FL3}.

\begin{prop}\label{napeti frame u okolini}
Let $(x_n)_n$ be a frame or an outer frame for a Hilbert $\textsf{A}$-module $X$ with the analysis operator $U$ and the optimal frame bounds $A$ and $B.$
Let $V_0=\frac{\sqrt{A}+\sqrt{B}}{2}U(U^*U)^{-\frac{1}{2}}.$ Then $V_0$ is the analysis operator of a $\left( \frac{\sqrt{A}+\sqrt{B}}{2}\right)^2$-tight frame or outer frame for $X$ for which
 $$\|U-V_0\|=\frac{\sqrt{B}-\sqrt{A}}{2}.$$
If $(y_n)_n$ is any tight frame or outer frame for $X$ with the analysis operator $V,$ then
$$\|U-V\|\geq\frac{\sqrt{B}-\sqrt{A}}{2}.$$
\end{prop}
\begin{proof}
Suppose first that $(y_n)_n$ is a frame or an outer frame for $X$ whose analysis operator $V$ is of the form $V^*V=\lambda^2 I$ for some scalar $\lambda>0.$ Then, as in the preceding proof, we have
$$
\|Ux\|\ge \|Vx\|-\|Vx-Ux\|=\lambda \|x\|-\|Vx-Ux\|\ge (\lambda-\|U-V\|)\|x\|
$$
for all $x\in X,$ so by \eqref{A} we get $\sqrt{A}\ge \lambda-\|U-V\|,$ that is,
\begin{equation}\label{lambda1}
\|U-V\|\geq \lambda- \sqrt{A}.
\end{equation}
On the other side, $\sqrt{B}=\|U\|\le \|U-V\|+\|V\|=\|U-V\|+\lambda$; thus,
\begin{equation}\label{lambda2}
\|U-V\|\geq \sqrt{B}-\lambda.
\end{equation}
Adding (\ref{lambda1}) and (\ref{lambda2}) we get $\|U-V\|\geq \frac{\sqrt{B}-\sqrt{A}}{2}.$

Consider now $V_0=\frac{\sqrt{A}+\sqrt{B}}{2}U(U^*U)^{-\frac{1}{2}}.$ An easy verification shows that $V_0$ is the analysis operator of a frame or an outer frame $(y_n)_n$ given by  $y_n=\frac{\sqrt{A}+\sqrt{B}}{2}(U_M^*U_M)^{-\frac{1}{2}}x_n$ for $n\in \Bbb N.$ Since $V_0^*V_0=\left( \frac{\sqrt{A}+\sqrt{B}}{2}\right)^2I,$ this is a tight frame or outer frame. Then, repeating (each particular step of) the computation from the beginning of the preceding proof we get
\begin{eqnarray*}
\|U-V_0\|&=&\left\|U-\frac{\sqrt{A}+\sqrt{B}}{2}U(U^*U)^{-\frac{1}{2}}\right\|\\
&=&\max\left\{\frac{\sqrt{A}+\sqrt{B}}{2} - \sqrt{A}, \sqrt{B} - \frac{\sqrt{A}+\sqrt{B}}{2}\right\}\\
&=&\frac{\sqrt{B}-\sqrt{A}}{2}.
\end{eqnarray*}
\end{proof}

\vspace{.1in}

\begin{remark}
Observe that "the best tight approximation", namely the frame or an outer frame from the preceding proposition, is actually Parseval if and only if $\sqrt{A}+\sqrt{B}=2.$ If this is the case the resulting distance is equal to $1-\sqrt{A}=\sqrt{B}-1$ and this result is, for $\sqrt{A}+\sqrt{B}=2,$ in accordance with Proposition~\ref{Parsevalov frame u okolini}.
\end{remark}

\vspace{0.2in}


\section{Finite extensions of Bessel sequences}

\vspace{0.2in}

Finite extensions of Bessel sequences to frames in Hilbert spaces are recently discussed in \cite{BB}.
In this section we discuss the same problem in modular context.

First note that each Bessel sequence in an AFG Hilbert $C^*$-module $X$ admits a finite extension to a frame: given a Bessel sequence (finite or infinite) in $X$ it suffices to extend it by any finite set of generators for $X.$
Thus, here we are interested primarily in countably generated Hilbert $C^*$-modules which are not AFG.

As before, our discussion will include both frames and outer frames. We shall first characterize (again in terms of analysis operators) those Bessel sequences and outer Bessel sequences in a Hilbert $C^*$-module $X$ that admit finite extensions to frames or outer frames for $X.$ After that, more specifically, we shall describe Bessel sequences and outer Bessel sequences which allow finite extensions to Parseval frames or outer Parseval frames.

A related, but more restrictive question we address is the following: given a Bessel sequence in $X$, does there exist its finite extension to a frame for $X$? We shall find necessary and sufficient conditions under which one can extend a given Bessel sequence to a frame by adding finitely many elements of $X$. This is, indeed, a stronger property;
 we shall see in Example~\ref{pa naravno}, that there are Bessel sequences that do not admit finite extensions to frames, but which do admit finite extensions (by elements of $M(X)\setminus X$) to outer frames.

We begin with our most general result on finite extensions.

\begin{theorem}\label{fin_ext_outer}
Let $(x_n)_{n=1}^{\infty}$ be a Bessel or an outer Bessel sequence in a Hilbert $\textsf{A}$-module $X$ with the analysis operator $U.$ Then there is a finite extension of $\xn$ to a frame or an outer frame for $X$ if and only if there exist $V\in \Bbb B(X,\ell^2(\textsf{A}))$ and $\theta \in \Bbb F(M(X))$ such that $I-V^*U=\theta|_X$.
\end{theorem}
\begin{proof} Suppose that a Bessel sequence $\xn$ admits a finite extension to a frame or an outer frame for $X.$
Let $f_1,\ldots,f_N\in M(X)$, $N\in \Bbb N$, be such that $(f_n)_{n=1}^{N} \cup (x_n)_{n=1}^{\infty}$ is a frame or an outer frame for $X.$
Let $U_1,F\in\Bbb B(X,\ell^2(\textsf{A}))$ be analysis operators of Bessel or outer Bessel sequences $(f_n)_{n=1}^N\cup (x_n)_{n=1}^{\infty}$ and $(f_n)_{n=1}^N\cup (0)_{n=1}^{\infty}$ respectively (so, in the later case we have $f_n$'s followed by infinitely many zeros).
Obviously, $U_1=F+S^NU,$ where $S$ denotes the unilateral shift on $\ell^2(\textsf{A}).$

Let us take any frame or outer frame dual to $(f_n)_{n=1}^N \cup (x_n)_{n=1}^{\infty},$ and write it, for convenience, as $(g_n)_{n=1}^{N} \cup (y_n)_{n=1}^{\infty}.$ Let $G, V,V_1\in\Bbb B(X,\ell^2(\textsf{A}))$ be the analysis operators of the Bessel sequences or outer Bessel sequences  $(g_n)_{n=1}^N\cup (0)_{n=1}^{\infty},$ $(y_n)_{n=1}^{\infty},$ and $(g_n)_{n=1}^N\cup (y_n)_{n=1}^{\infty},$ respectively. Again, $V_1=G+S^NV.$

Since $(g_n)_{n=1}^{N} \cup (y_n)_{n=1}^{\infty}$ and $(f_n)_{n=1}^N \cup (x_n)_{n=1}^{\infty}$ are dual to each other, it follows $V_1^*U_1=I.$ Since, obviously, $G^*S^N=0$ and $(S^N)^*F=0,$  we have
\begin{eqnarray*}
  I&=& (G+S^NV)^*(F+S^NU) \\
   &=& G^*F+V^*(S^N)^*F+G^*S^NU+V^*(S^N)^*S^NU\\
   &=&G^*F+V^*U,
\end{eqnarray*}
that is, $I-V^*U=G^*F.$ Let  $\theta\in\Bbb F(M(X))$ be defined as $\theta=\sum_{n=1}^N\theta_{g_n,f_n}.$ Then
$$G^*F(x)=\sum_{n=1}^Ng_n\la f_n, x\ra=\sum_{n=1}^N\theta_{g_n,f_n}(x)=\theta (x),\quad \forall x\in X,$$
so we conclude that $I-V^*U=\theta|_X.$

Conversely, suppose there is $V\in \Bbb B(X,\ell^2(\textsf{A}))$ and $\theta \in \Bbb F(M(X))$ such that $I-V^*U=\theta|_X$.
Let $f_1,\ldots,f_N,g_1,\ldots,g_N\in M(X)$ be such that $\theta=\sum_{n=1}^N\theta_{g_n,f_n}.$
By Corollary~\ref{adjointable operators are Bessel sequences 2} there is a Bessel sequence or an outer Bessel sequence $\yn$ such that $V$ is its analysis operator.
Then the sequences  $(f_n)_{n=1}^N \cup (x_n)_{n=1}^{\infty}$ and  $(g_n)_{n=1}^N \cup (y_n)_{n=1}^{\infty}$ are also Bessel or outer Bessel sequences. Let $F,G,U_1,V_1$ be as before. The same computation shows that
$$V_1^*U_1=G^*F+V^*U=\theta|_X+V^*U=I,$$
so, by Lemma~\ref{dual relaxed},
$(f_n)_{n=1}^N \cup (x_n)_{n=1}^{\infty}$ and  $(g_n)_{n=1}^N \cup (y_n)_{n=1}^{\infty}$ are frames or outer  frames for $X.$
\end{proof}

In the same way one proves the following corollary which concerns finite extensions of a Bessel sequence by elements of the original Hilbert $C^*$-module $X$ (i.e.,  without using elements from $M(X)\setminus X$) in which case we end up with a frame for $X$.

\begin{cor}\label{fin_ext}
Let $(x_n)_{n=1}^{\infty}$ be a Bessel sequence in a Hilbert $\textsf{A}$-module $X$ with the analysis operator $U.$ Then there is a finite extension of $\xn$ to a frame for $X$ if and only if there exists $V\in \Bbb B(X,\ell^2(\A))$ such that $I-V^*U\in \Bbb F(X).$
\end{cor}

The above property of a Bessel sequence is more restrictive than that from Theorem~\ref{fin_ext_outer}. To see this, we demonstrate an example of
a Bessel sequence that does not allow a finite extension to a frame, but  which does have (many) finite extensions to outer frames.

\begin{example}\label{pa naravno}
Take a separable infinite-dimensional  Hilbert space $H$ and consider $X=\Bbb K(H)$ as a Hilbert $\Bbb K(H)$-module in the standard way.
Let $(\epsilon_n)_n$ be an orthonormal basis for $H.$ For each $n\in \Bbb N$ denote by $e_n$ the orthogonal projection to $\text{span}\{\epsilon_n\}.$ Further, put $H_1=\overline{\text{span}}\{\epsilon_{2n-1}:n\in \Bbb N\}$ and
$H_2=\overline{\text{span}}\{\epsilon_{2n}:n\in \Bbb N\}.$ If we denote by $p_1$ and $p_2$ the corresponding orthogonal projections then, obviously, $p_1+p_2=e,$ where $e$ denotes the identity operator on $H.$

Consider now the sequence $(x_n)_n$ in $X$ defined by $x_n=e_{2n}$ for all $n \in \Bbb N.$ For each $a\in X$ we have
$$
\sum_{n=1}^{\infty}\langle a,x_n\rangle \langle x_n,a\rangle=\sum_{n=1}^{\infty}a^*e_{2n}a=a^*p_2a,
$$
with the convergence in norm, so by Proposition~\ref{Bessel relaxed}, $(x_n)_n$ is a Bessel sequence in $X.$ Let $U$ be its analysis operator.

Let us first show that $(x_n)_n$ does not admit a finite extension to a frame for $X.$ To prove this, suppose the opposite. Then by Corollary~\ref{fin_ext}, there exist $f_1,\ldots,f_N,g_1,\ldots,g_N\in X$ for some $N\in \Bbb N,$ and an operator $V\in \Bbb B(X,\ell^2(\Bbb K(H)))$ (which is the analysis operator of a Bessel or an outer Bessel sequence $(y_n)_n$ in $X$) such that $I-V^*U=\sum_{n=1}^N\theta_{f_n,g_n}.$
This means that
$$
a-\sum_{n=1}^{\infty}y_n\langle x_n,a\rangle =\sum_{n=1}^Nf_n\langle g_n,a\rangle,\quad \forall a\in X.
$$
Denote $b=\sum_{n=1}^Nf_ng_n^*$ and observe that $b\in X=\Bbb K(H).$ Now the preceding equality can be rewritten as
$$
a-\sum_{n=1}^{\infty}y_ne_{2n}a=ba,\quad\forall a\in X.
$$
In particular, if $a$ is any operator in $\Bbb K(H)$ whose range is contained in $H_1,$ we have $a=ba.$  This in turn implies $b\epsilon_{2n-1}=\epsilon_{2n-1}$ for all $n\in \Bbb N$ wherefrom we conclude that a closed infinite dimensional subspace $H_1$ of $H$ is contained in the range of a compact operator $b$, which is a contradiction.

Next we show that $\xn$ can be extended to an outer frame for $X$ by adding a single vector from $M(X)\setminus X.$ Namely, if $c\in M(X)=\Bbb B(H)$ is invertible, then $cc^*\ge \frac{1}{\|c^{-1}\|^2}e$, so
$$\la a,c\ra\langle c,a\rangle+\sum_{n=1}^{\infty}\la a,x_n\ra\langle x_n,a\rangle\ge \la a,c\ra\langle c,a\rangle = a^*cc^*a\ge \frac{1}{\|c^{-1}\|^2}a^*a$$
for all $a\in X.$ Thus, $c, x_1,x_2,\ldots$  is an outer frame for $X.$

Moreover, $\xn$ can be extended to an outer Parseval frame for $X,$ again by adding just one vector from $M(X)\setminus X.$ Indeed, we have for each $x\in X$
$$\la a,p_1\ra\langle p_1,a\rangle+\sum_{n=1}^{\infty}a^*e_{2n}a=a^*p_1a+a^*p_2a=a^*a$$
so the sequence $p_1,x_1,x_2,x_3,\ldots$ is an outer Parseval frame for $X.$
\end{example}

\vspace{.1in}

Our next goal is to describe those Bessel sequences that admit finite extensions to Parseval frames. Note that this question, in contrast to the preceding one, is non-trivial even for AFG Hilbert $C^*$-modules.
First we need some auxiliary results. We begin with a lemma which is a variant of Lemma 5.5.4 from \cite{T}.

\begin{lemma} \label{lemma_from_MT}
Let $X$ be a Hilbert $\textsf{A}$-module. Let $x\in X$ and $a\in\textsf{A}$ be such that $0\le a\le \la x,x\ra.$ Then there exists $z\in X$ such that $a=\la z,z\ra.$
\end{lemma}
\begin{proof} Let $v\in X$ be such that $x=v\la v,v\ra.$ Let $y=v\la v,v\ra^{\frac{1}{4}}.$ Then
$$
\la y,y\ra^2=\la v,v\ra^{\frac{1}{4}}\la v,v\ra \la v,v\ra^{\frac{1}{4}}\la v,v\ra^{\frac{1}{4}}\la v,v\ra \la v,v\ra^{\frac{1}{4}}=\la v,v\ra^3=\la x,x\ra.
$$
Write $\la y,y\ra =c.$ Then we have $0\le a\le c^2.$ Put
$$b_n=(c+\frac{1}{n}e)^{-\frac{1}{2}}a^{\frac{1}{2}},\quad n\in \Bbb N.$$
Here $e$ denotes the unit in $\textsf{A}$ or in $\tilde{\textsf{A}}$, but in both cases $b_n\in \textsf{A}$ for all $n.$
Then for all $m,n\in \Bbb N,$ $n\ge m$ we have

\begin{align*}
&\|b_n-b_m\|^2=\|(b_n-b_m)(b_n-b_m)^*\|\\
&=\left\|\left(\left(c+\frac{1}{n}e\right)^{-\frac{1}{2}}-\left(c+\frac{1}{m}e\right)^{-\frac{1}{2}}\right) a \left(\left(c+\frac{1}{n}e\right)^{-\frac{1}{2}}-\left(c+\frac{1}{m}e\right)^{-\frac{1}{2}}\right)  \right\|\\
&\le\left\|\left(\left(c+\frac{1}{n}e\right)^{-\frac{1}{2}}-\left(c+\frac{1}{m}e\right)^{-\frac{1}{2}}\right) c^2 \left(\left(c+\frac{1}{n}e\right)^{-\frac{1}{2}}-\left(c+\frac{1}{m}e\right)^{-\frac{1}{2}}\right)  \right\|\\
&=\left\|c\left(c+\frac{1}{n}e\right)^{-\frac{1}{2}}-c\left(c+\frac{1}{m}e\right)^{-\frac{1}{2}}\right\|^2.
\end{align*}

The sequence $(f_n)_n,$ $f_n(t)=t(t+\frac{1}{n})^{-\frac{1}{2}}$ is an increasing sequence of positive continuous functions that converges pointwise for $t \in [0,\|c\|]$ to the continuous  function $f(t)=\sqrt{t}.$ By Dini's theorem $(f_n)_n$ converges to $f$ uniformly on $[0,\|c\|]$; hence,
\begin{equation}\label{korijen_iz_c}
\lim_{n\rightarrow \infty} c\left(c+\frac{1}{n}e\right)^{-\frac{1}{2}}=c^{\frac{1}{2}}.
\end{equation}
Now the above computation shows that $(b_n)_n$ is a Cauchy sequence in $\textsf{A}.$ Put $b=\lim_{n\rightarrow \infty} b_n.$ Then we have
$\la y,y\ra^{\frac{1}{2}}b=\lim_{n\rightarrow \infty}\la y,y\ra^{\frac{1}{2}}b_n$ and $b^*\la y,y\ra^{\frac{1}{2}}=\lim_{n\rightarrow \infty}b_n^*\la y,y\ra^{\frac{1}{2}}$ which implies
$b^*\la y,y\ra b=\lim_{n\rightarrow \infty}b_n^*\la y,y\ra b_n,$ that is,
\begin{equation}\label{predzadnja jednakost}
\la yb,yb\ra=\lim_{n\rightarrow \infty}\la yb_n,yb_n\ra.
\end{equation}

On the other hand,

\begin{align*}
&\|a-\la yb_n,yb_n\ra\|=\|a-b_n^*\la y,y\ra b_n\|\\
&=\left\|a-a^{\frac{1}{2}}\left(c+\frac{1}{n}e\right)^{-\frac{1}{2}}c\left(c+\frac{1}{n}e\right)^{-\frac{1}{2}}a^{\frac{1}{2}}\right\|\\
&=\left\|\left(a^{\frac{1}{2}}\left(e-c\left(c+\frac{1}{n}e\right)^{-1}\right)^{\frac{1}{2}} \right) \left(a^{\frac{1}{2}}\left(e-c\left(c+\frac{1}{n}e\right)^{-1}\right)^{\frac{1}{2}} \right)^*\right\|\\
&=\left\|\left(a^{\frac{1}{2}}\left(e-c\left(c+\frac{1}{n}e\right)^{-1}\right)^{\frac{1}{2}} \right)^* \left(a^{\frac{1}{2}}\left(e-c\left(c+\frac{1}{n}e\right)^{-1}\right)^{\frac{1}{2}} \right)\right\|\\
&=\left\|\left(e-c\left(c+\frac{1}{n}e\right)^{-1}\right)^{\frac{1}{2}}a\left(e-c\left(c+\frac{1}{n}e\right)^{-1}\right)^{\frac{1}{2}}\right\|\\
&\le \left\|c^2\left(e-c\left(c+\frac{1}{n}e\right)^{-1}\right)\right\|\\
&=\left\|c^2-c^3\left(c+\frac{1}{n}e\right)^{-1}\right\|.
\end{align*}

It follows from  \eqref{korijen_iz_c} that $\lim_{n\rightarrow \infty}c^3(c+\frac{1}{n}e)^{-1}=c^2.$
Hence, the above computation shows that $\lim_{n\rightarrow \infty}\la yb_n,yb_n\ra=a.$ This, together with \eqref{predzadnja jednakost}, gives us $\la yb,yb\ra =a.$ Put $z=yb.$
\end{proof}

\vspace{.1in}

\begin{prop} \label{positive_and_finite}
Let $X$ be a Hilbert $\textsf{A}$-module and $a\in \textsf{A}, a\ge 0$, such that $a=\sum_{n=1}^N\la u_n,v_n\ra$ for some  $N\in \Bbb N$ and $u_1,\ldots,u_N,v_1,\ldots,v_N\in X.$
Then there exist $x_1,\ldots,x_N\in X$ such that $a=\sum_{n=1}^N\la x_n,x_n\ra.$
\end{prop}

\begin{proof}
 It follows from the polarization formula and self-adjointness of $a$ that
$$4a=\sum_{n=1}^N \la u_n+v_n,u_n+v_n\ra-\sum_{n=1}^N \la u_n-v_n,u_n-v_n\ra,$$
wherefrom we get
\begin{equation}\label{a_le_xx}
a\le \sum_{n=1}^N \la\frac{1}{2}u_n+\frac{1}{2}v_n,\frac{1}{2}u_n+\frac{1}{2}v_n\ra.
\end{equation}
Let $X_N=\oplus_{n=1}^N X$ be a direct sum of $N$ copies of $X,$ which is a Hilbert $\textsf{A}$-module with the inner product defined by $\la x,y\ra=\sum_{n=1}^N \la x_n,y_n\ra,$ where $x=(x_1,\ldots,x_N)$ and  $y=(y_1,\ldots,y_N).$

If we denote $u=(\frac{1}{2}u_1+\frac{1}{2}v_1,\ldots,\frac{1}{2}u_N+\frac{1}{2}v_N),$ then \eqref{a_le_xx} reads as $0\le a\le \la u,u\ra.$ By Lemma~\ref{lemma_from_MT}, applied in $X_N,$  there exists $z\in X_N$ such that
$a=\la z,z\ra.$
If we put $z=(x_1,\ldots,x_N)$ then $a=\sum_{n=1}^N\la x_n,x_n\ra.$
\end{proof}

\vspace{.1in}

Regarding a right Hilbert $\textsf{A}$-module $X$ as a left Hilbert $\Bbb {K}(X)$-module we immediately get the following corollary. It refines the statement of Corollary~\ref{oblik pozitivnog kompaktnog operatora} in a natural way.

\begin{cor}\label{pozitivan konacnog ranga}
Let $X$ be a Hilbert $\textsf{A}$-module and $T\in\Bbb F(X)$ such that $T\ge 0.$ Then there exist $N\in \Bbb N$ and $x_1,\ldots,x_N\in X$ such that $T=\sum_{n=1}^N \theta_{x_n,x_n}.$
\end{cor}

\vspace{.1in}

We are now ready to characterize Bessel sequences and outer Bessel sequences in Hilbert $C^*$-modules that admit finite extensions to Parseval frames or outer Parseval frames.

\begin{theorem}\label{fin_ext_Pars_outer}
Let $(x_n)_{n=1}^{\infty}$ be a Bessel or an outer Bessel sequence in a Hilbert $\textsf{A}$-module $X$ with the analysis operator $U$ and the optimal Bessel bound $B.$
Then there is a finite extension of $\xn$ to a Parseval or an outer Parseval frame for $X$ if and only if and there exists $\theta\in \Bbb F(M(X))$ such that  $I-U^*U=\theta|_X$ and  $B\le 1$.
\end{theorem}
\begin{proof}
Suppose there exists a finite sequence $(f_n)_{n=1}^{N}$ in $M(X)$ such that
  $(f_n)_{n=1}^{N} \cup (x_n)_{n=1}^{\infty}$ is a Parseval frame or an outer Parseval frame for $X.$ Then for every $x\in X$ it holds
\begin{equation}\label{Pars_eq}
\sum_{n=1}^N \la x,f_n \ra \la f_n,x\ra+\sum_{n=1}^\infty \la x,x_n \ra\la x_n,x\ra = \la x,x\ra.
\end{equation}
This implies $\sum_{n=1}^\infty \la x,x_n\ra\la x_n,x\ra\le \la x,x\ra$ for all $x\in X,$ so $B\le 1.$
Further, if we denote $\theta=\sum_{n=1}^N\theta_{f_n,f_n}\in \Bbb F(M(X)),$ then \eqref{Pars_eq} gives us $\theta|_X+U^*U=I,$ that is, $I-U^*U=\theta|_X.$

Conversely, suppose $B\le 1$ and $I-U^*U=\theta|_X$ for some $\theta\in \Bbb F(M(X)).$ Since $U^*U\le B\cdot I=I$ we have $I-U^*U\ge 0.$ Then its extension $I_{M(X)}-U_M^*U_M$ is positive and $I_{M(X)}-U_M^*U_M=\theta,$ so we can apply Corollary~\ref{pozitivan konacnog ranga} to $M(X)$ and  $I_{M(X)}-U_M^*U_M.$ Therefore,  $I_{M(X)}-U_M^*U_M=\sum_{n=1}^N\theta_{f_n,f_n}$ for some $N\in\Bbb N$ and $f_1,\ldots, f_N\in M(X).$ Now $I_{M(X)}=U_M^*U_M+\sum_{n=1}^N\theta_{f_n,f_n}$ gives us
$$\sum_{n=1}^N \la x,f_n \ra \la f_n,x\ra+\sum_{n=1}^\infty \la x,x_n \ra\la x_n,x\ra=\la x,x\ra,\quad\forall x\in X.$$
Hence, $(f_n)_{n=1}^N \cup (x_n)_{n=1}^{\infty}$ is a Parseval frame or an outer Parseval frame for $X$ depending on whether all $f_n$'s are in $X$ or not.
\end{proof}

The following corollary is concerned with Bessel sequences and their finite extensions to Parseval frames (so, again, as in Corolarry~\ref{fin_ext} we are now interested only in extensions obtained by finitely many elements of the original module $X$). It is convenient to split the statement into two cases: when $X$ is not AFG, and when $X$ is AFG.

\begin{cor}\label{fin_ext_Pars} Let $X$ be a Hilbert $\textsf{A}$-module.  Let $(x_n)_{n=1}^{\infty}$ be a Bessel sequence in $X$  (either finite or infinite) with the analysis operator $U$ and the optimal Bessel bound $B.$
\begin{enumerate}
  \item[(a)] If $X$ is not AFG, then $\xn$ is finitely extendable to a Parseval frame for $X$ if and only if $I-U^*U\in \Bbb F(X)$ and $B=1.$
  \item[(b)] If $X$ is AFG, then $\xn$ is finitely extendable to a Parseval frame for $X$ if and only if $B\le 1.$
\end{enumerate}
\end{cor}
\begin{proof}
First, in the same fashion as in the preceding proof we obtain in both cases: $\xn$ is finitely extendable to a Parseval frame for $X$ if and only if $I-U^*U=\theta\in \Bbb F(X)$ and $B\le 1.$ We now proceed by specific arguments in each of the above two cases.

(a) Suppose $X$ is not AFG and $\xn$ has a finite extension to a Parseval frame for $X.$ Then $I-U^*U$ is non-invertible in $\Bbb  B(X),$ since otherwise we would have
$$I= (I-U^*U)^{-1}(I-U^*U)=(I-U^*U)^{-1}\theta,$$
which, by the ideal property of $\Bbb F(X),$ gives $I\in \Bbb  F(X).$ But this would imply that $X$ is AFG, contrary to our assumption.
Now, non-invertibility of $I-U^*U$ means that $1\in\sigma(U^*U),$ so $B=\|U^*U\|\ge 1.$

(b) Suppose $X$ is AFG and $\xn$ has a finite extension to a Parseval frame for $X.$ Here we observe that a general condition $I-U^*U\in \Bbb F(X)$ obtained at the beginning of the proof is automatically satisfied. Indeed, since $X$ is AFG, we have $\Bbb B(X)=\Bbb K(X)=\Bbb F(X),$ so $I-U^*U\in \Bbb F(X)$ for all $U\in \Bbb B(X,\ell^2(\A)).$
\end{proof}

\begin{remark}\label{pa_naravno_2}
Recall that it can happen that $X$ is not an AFG Hilbert $\A$-module, but $M(X)$ is an AFG Hilbert $M(\A)$-module. (As an example, one can take $X=\A$, where $\A$ is a non-unital $C^*$-algebra). In such cases, each Bessel sequence or an outer Bessel sequence with the optimal Bessel bound $B< 1$ admits a finite extension to
an outer Parseval frame for $X.$
To see this, denote the corresponding analysis operator by $U$. Then, since $M(X)$ is an AFG module, $(I-U^*U)_M\in \Bbb F(M(X))$, so the remaining condition from Theorem~\ref{fin_ext_Pars_outer}, namely
 $I-U^*U=\theta|_X$ for some $\theta\in \Bbb F(M(X))$ is also satisfied.

On the other hand, such sequences cannot allow finite extensions to Parseval frames because of $B<1$ (see Corollary~\ref{fin_ext_Pars}(a)).

\end{remark}


\vspace{.1in}

We conclude with an application of Corollary~\ref{fin_ext_Pars}.

\begin{cor}\label{appr_unit}
Let $\textsf{A}$ be a non-unital $\sigma$-unital $C^*$-algebra. Let $(a_n)_n$ be a sequence in $\textsf{A}$ such that
for all $a\in \textsf{A}$ the series $\sum_{n=1}^\infty a^*a_na_n^*a$ converges in norm and  $\|\sum_{n=1}^\infty a^*a_na_n^*a\|\le \|a\|.$
Then the series $\sum_{n=1}^{\infty}a_na_n^*$ strictly converges to an element $f\in M(\textsf{A})$ such that $f\le e.$ Moreover, if
  $e-f\in \textsf{A},$ then there exists $b\in \textsf{A},$ $b\ge 0,$  such that the sequence
$(b+\sum_{n=1}^N a_na_n^*)_N$ is an  approximate unit for $\textsf{A}.$
\end{cor}

\begin{proof}
Consider $\textsf{A}$ as a Hilbert $\textsf{A}$-module. Since $\textsf{A}$ is non-unital and $\sigma$-unital, $\textsf{A}$ is countably generated and not AFG.

First, let us prove that the series $\sum_{n=1}^\infty a_na_n^*$ is strictly convergent.
 By the first assumption, the series $\sum_{n=1}^\infty \la a,a_n\ra\la a_n,a\ra$ converges in norm for every $a\in\textsf{A}.$  By Theorem~\ref{Bessel relaxed}, $(a_n)_n$ is a Bessel sequence in $\textsf{A}.$ If $U\in \Bbb B(\A,\ell^2(\textsf{A}))$ is its analysis operator, then $U^*U\in\Bbb  B(\textsf{A})$ is given by
\begin{equation}\label{22lipnja}
U^*Ua=\sum_{n=1}^\infty a_n\la a_n,a\ra=\sum_{n=1}^\infty a_na_n^*a,\quad \forall a\in\textsf{A}.
\end{equation}
Thus, the series  $\sum_{n=1}^\infty a_na_n^*a$ converges in norm for all $a\in \textsf{A}.$ By taking adjoints, we conclude that the series $\sum_{n=1}^\infty a^*a_na_n^*$ is also norm-convergent for all $a$ in $\textsf{A}.$ In other words, there exists  $f=\mbox{(strict)} \sum_{n=1}^\infty a_na_n^*\in M(\A).$ From this we conclude that $U^*Ua=fa$ for all $a\in \textsf{A}.$

\vspace{0.1in}
Now, assuming that $e-f\in \textsf{A},$ we shall prove that $(a_n)_n$ is a Bessel sequence in $\textsf{A}$ that admits a finite extension to a Parseval frame for $\textsf{A}.$

First, recall that $\Bbb K(\textsf{A})=\Bbb F(\textsf{A}).$ Therefore, the equality
$(I-U^*U)a=(e-f)a,\,a\in \textsf{A},$ since $e-f\in \textsf{A},$ implies that $I-U^*U\in \Bbb  F(A).$

Since $\A$ is non-unital, each operator from $\Bbb F(\textsf{A})$ is non-invertible;  in particular, $I-U^*U$ is non-invertible, so  $1\in\sigma(U^*U)$ and then $\|U\|\ge 1.$ By \eqref{22lipnja} and the second assumption of the corollary, $\|U\|\le 1,$ so $\|U\|=1.$

By Corollary~\ref{fin_ext_Pars}(a), there exists a finite sequence $(b_n)_{n=1}^M$ in $\textsf{A}$ such that $(b_n)_{n=1}^M\cup(a_n)_{n=1}^\infty$ is a Parseval frame for $\textsf{A}.$ Denote $b=\sum_{n=1}^M b_nb_n^*.$ Since $\Bbb K(\A)$ and $\A$ are isomorphic as $C^*$-algebras, by Proposition~\ref{parseval frames are approximate units} the sequence $(b+\sum_{n=1}^N a_na_n^*)_N,$ as a subsequence of an approximate  unit for $\A,$ is itself an approximate unit for $\A.$
\end{proof}

\vspace{0.2in}


\end{document}